\documentclass[10pt]{amsart}
\usepackage[utf8]{inputenc}
\usepackage[T1]{fontenc}
\usepackage[english]{babel}
\usepackage{amsmath}
\usepackage{amsfonts}
\usepackage{amssymb}
\usepackage{graphicx, overpic}
\usepackage{amsthm}
\usepackage{textcomp}
\usepackage{mathrsfs}
\usepackage{hyperref,cleveref}
\usepackage{enumitem}
\usepackage{wrapfig}
\usepackage{multicol}
\usepackage[margin=1.3in]{geometry}

\theoremstyle{plain}
\newtheorem{thm}{Theorem}[section]
\newtheorem{lem}[thm]{Lemma}
\newtheorem{prop}[thm]{Proposition}
\newtheorem{cor}[thm]{Corollary}
\theoremstyle{definition}
\newtheorem{defn}[thm]{Definition}
\newtheorem{example}[thm]{Example}
\newtheorem{remark}[thm]{Remark}

\newtheorem{claim}[thm]{Claim}

\newtheorem{notation}[thm]{Notation}
\newtheorem{question}[thm]{Question}

\newtheorem{sassumption}[thm]{Standing Assumptions}

\newtheorem{openquestion}[thm]{Open Problem}

\newcommand{\Z}{\mathbb{Z}}
\newcommand{\R}{\mathbb{R}}
\newcommand{\N}{\mathbb{N}}
\newcommand{\Q}{\mathbb{Q}}

\newcommand{\Hy}{\mathbb{H}}

\newcommand{\C}{\mathbb{C}}

\newcommand{\cC}{\mathcal{C}}

\newcommand{\cE}{\mathcal{E}}
\newcommand{\cF}{\mathcal{F}}
\newcommand{\cG}{\mathcal{G}}
\newcommand{\cH}{\mathcal{H}}

\newcommand{\cN}{\mathcal{N}}

\newcommand{\cP}{\mathcal{P}}

\newcommand{\cR}{\mathcal{R}}

\newcommand{\cU}{\mathcal{U}}
\newcommand{\cV}{\mathcal{V}}
\newcommand{\cW}{\mathcal{W}}

\newcommand{\E}{\mathbb{E}}

\newcommand{\G}{\Gamma}

\newcommand{\p}{\partial}

\newcommand{\la}{\left\langle}
\newcommand{\ra}{\right\rangle}
\newcommand{\ba}{\boldsymbol{a}}

\makeatletter
\newsavebox{\@brx}
\newcommand{\llangle}[1][]{\savebox{\@brx}{\(\m@th{#1\langle}\)}%
  \mathopen{\copy\@brx\kern-0.5\wd\@brx\usebox{\@brx}}}
\newcommand{\rrangle}[1][]{\savebox{\@brx}{\(\m@th{#1\rangle}\)}%
  \mathclose{\copy\@brx\kern-0.5\wd\@brx\usebox{\@brx}}}
\makeatother

\usepackage{mathtools}

\DeclareMathOperator{\cd}{cd}

\DeclareMathOperator{\Isom}{Isom}

\DeclareMathOperator{\Lk}{Lk}

\DeclareMathOperator{\CAT}{CAT}
\DeclareMathOperator{\diam}{diam}
\DeclareMathOperator{\Hdim}{Hausdim}
\DeclareMathOperator{\Cdim}{Confdim}
\DeclareMathOperator{\boxd}{box}
\DeclareMathOperator{\covering}{covering}

\newcommand{\St}{\operatorname{St}}

\newcommand{\arccosh}{\operatorname{arccosh}}

\usepackage{tikz-cd} 
\tikzset{
	labl/.style={anchor=south, rotate=90, inner sep=.5mm}
}

\usepackage{snotez}

\usepackage[colorinlistoftodos]{todonotes}

\definecolor{amethyst}{rgb}{0.6, 0.4, 0.8}

\newcommand{\hide}[1]{}

\title{Visual metrics on boundaries of hyperbolic spaces}
\author{Emily Stark}
\date{\today}

\begin{document}

\begin{abstract}
    This is an expository article on visual metrics on boundaries of hyperbolic metric spaces. We discuss the construction of visual metrics, quasisymmetries and their invariants, Hausdorff and conformal dimension, and constructions and applications of Gromov's round trees. There is a focus on providing examples throughout. These notes are based on the material of a minicourse given by the author at the 2024 Riverside Workshop in Geometric Group Theory.
\end{abstract}

 \maketitle

    \section{Introduction}

        Hyperbolicity is central to geometric group theory. This robust notion of negative curvature, introduced by Gromov~\cite{gromov}, captures and unifies features of negatively curved manifolds and combinatorial objects like trees. Hyperbolic groups abound. This family includes groups that act properly and cocompactly by isometries on hyperbolic space $\Hy^n$, other rank-1 symmetric spaces, $\CAT(-1)$ spaces, and piecewise Euclidean or piecewise hyperbolic cellular complexes that satisfy natural local combinatorial curvature conditions. The family includes many Coxeter groups, free-by-cyclic groups, and random groups in the sense of Gromov \cite{moussong-thesis, brinkmann00, gromov, gromov93}.

        Consequences of hyperbolicity have further justified its study. Hyperbolic groups are finitely presented and act properly and cocompactly on a finite-dimensional contractible simplicial complex. These groups have linear isoperimetric inequalities, are biautomatic, and have solvable word problem, conjugacy problem, and isomorphism problem \cite{dehn12, gromov, cannon84, wordprocessing, sela95, dahmaniguirardel-iso}. Powerful programs within the field rely on aspects of hyperbolicity, including the study of relative hyperbolicity, acylindrical hyperbolicity, and hierarchical hyperbolicity.

    The geometry of real hyperbolic space $\Hy^{n+1}$ and the structure of its isometry group are intimately tied to its boundary sphere $S^{n}$. Isometries of hyperbolic space induce conformal maps of the boundary, and the converse holds as well: every conformal map of the boundary is an extension of an isometry for $n \geq 2$. In analogy, a hyperbolic group has a natural boundary compactification. The boundary admits a family of metrics inducing the topology. Quasi-isometries between hyperbolic groups extend to quasisymmetric homeomorphisms between their boundaries, and the converse holds as well.

    Hyperbolic groups are effectively studied through the lens of their boundaries. The topology alone captures algebraic information. Celebrated work of Bowditch~\cite{bowditch} proves the local cut point structure of the boundary of a hyperbolic group captures the system of splittings of the group over $2$-ended subgroups. The topology yields a quasi-isometry invariant JSJ decomposition, generalizing work of Sela~\cite{sela-JSJ} in the torsion-free case. The JSJ decomposition can be used, for example, to understand the automorphism group of a hyperbolic group.

 \noindent {\bf Quasisymmetries.} The boundary of a hyperbolic group admits a family of metrics, called visual metrics, which induce the topology on the boundary.  Quasisymmetries are the naturally associated maps, which is made precise in the following theorem.

    \begin{thm} \cite{buyaloschroeder} \cite{paulin}    \label{thm:introQIQS}
         Let $X$ and $X'$ be proper geodesic hyperbolic metric spaces admitting geometric actions by hyperbolic groups. Let $d_a$ and $d_a'$ be visual metrics on $\p X$ and $\p X'$, respectively. Then the spaces $X$ and $X'$ are quasi-isometric if and only if the boundaries $(\p X, d_a)$ and $(\p X', d_a')$ are quasisymmetric.
    \end{thm}

    \Cref{thm:introQIQS} is crucial in proving rigidity results, where one wants to show that additional algebraic or metric structure is preserved by a quasi-isometry. For example, the renowned Mostow Rigidity Theorem \cite{mostow}, proves that any isomorphism between fundamental groups of compact, negatively curved, locally symmetric manifolds is induced by an isometry. In a different direction, any finitely generated group quasi-isometric to $\Hy^{n+1}$ is virtually isomorphic to a uniform lattice in $\Isom(\Hy^{n+1})$ by work of Tukia~\cite{tukia86} for $n \geq 2$. The proofs utilized the quasiconformal structure of the boundary and the possibility of promoting the quasisymmetries induced by a quasi-isometry to conformal maps. These ideas have been vastly extended, producing remarkable results in the field; we refer the reader to the surveys of Kleiner~\cite{kleiner-ICM} and Bourdon~\cite{bourdon-mostowtype} for detailed discussion of this direction and open problems.

\vskip.05in

    \noindent {\bf Conformal dimension.}    Foundational to geometric group theory is the search for quasi-isometry invariants. \Cref{thm:introQIQS} proves that any quasisymmetry invariant yields one. Conformal dimension is a powerful quasisymmetry invariant introduced by Pansu~\cite{pansu} to study the structure of the rank-1 symmetric spaces. Pansu computed the conformal dimension of the boundary of each such space and, consequently, proved that non-isometric non-compact rank-1 symmetric spaces are not quasi-isometric. For example, while $\p \Hy_{\R}^4 \cong \p \Hy_{\C}^2 \cong S^3$, these spaces are not quasi-isometric.

    The conformal dimension of a metric space $(Z,d)$ is defined to be the infimum of the Hausdorff dimension amongst all metric spaces quasisymmetric to $(Z,d)$. Conformal dimension is often viewed as a uniformization or optimization problem: amongst all quasisymmetric deformations of a metric space, find the metric that is ``least fractal,'' as measured by Hausdorff dimension.
    For example, the Hausdorff dimension of the von Koch snowflake is greater than the Hausdorff dimension of the unit circle. There is a quasisymmetry from the von Koch snowflake to the unit circle, and the conformal dimension of the von Koch snowflake equals one.

    Conformal dimension has algebraic consequences. Carrasco--Mackay~\cite{carrascomackay} characterized the hyperbolic groups without 2-torsion whose boundaries have conformal dimension equal to one as exactly the family of such hyperbolic groups that admit a hierarchy over elementary subgroups that terminates in elementary and Fuchsian groups (including free groups). Bonk--Kleiner~\cite{bonkkleiner05S2} proved that Cannon's Conjecture, that every hyperbolic group with 2-sphere boundary is virtually Kleinian, can be viewed as a question of attaining (Ahlfors regular) conformal dimension.

\vskip.05in

 \noindent {\bf Gromov's round trees.}   In this article we discuss a primary tool to produce lower bounds on conformal dimension by finding a sufficiently rich curve family in the metric space. For example, the Hausdorff dimension of the two-thirds Cantor set equals $\frac{\log 2}{\log 3}$, while its conformal dimension equals zero. However, taking the product with an interval has a remarkable consequence: the resultant space is one whose Hausdorff dimension achieves the conformal dimension; see \Cref{thm:cantorInterval}. This type of result is known as a {\it stabilization theorem} and has many useful consequences.

    Gromov's round tree construction allows one to construct the product of a Cantor set and an interval in the boundary of a hyperbolic metric space so that the Hausdorff dimension of the Cantor set can be computed. Lower bounds on the conformal dimension of the boundary of a hyperbolic group can then be obtained by quasi-isometrically embedding a suitable round tree in a space the group acts on geometrically. 

    Following Gromov's article~\cite[Section 7.C3]{gromov93}, the round tree technique was first employed by Bourdon~\cite{bourdon95} to prove there are infinitely many quasi-isometry classes among hyperbolic groups that act geometrically on right-angled Fuchsian buildings. We note that subsequent work of Bourdon~\cite{bourdon} computed the conformal dimension explicitly, illustrating a family of groups for which conformal dimension of the boundary achieves a dense set of values in $(1,\infty)$.
    Mackay~\cite{mackay10} utilized the round tree technique to prove that the conformal dimension of the boundary of a one-ended hyperbolic group that does not virtually split over a 2-ended subgroup is strictly greater than one. Mackay~\cite{mackay} also used round trees to prove there are infinitely many quasi-isometry classes among certain small cancellation and random groups.

    With Field--Gupta--Lyman~\cite{fieldguptalymanstark}, the author used round trees to produce lower bounds on the conformal dimension of certain Bowditch boundaries of non-hyperbolic, relatively hyperbolic group pairs. We proved there are infinitely many quasi-isometry classes among hyperbolic groups with Pontryagin sphere boundary. Work joint with Cashen--Dani--Schreve~\cite{cashendanischrevestark} proves an analogous result for right-angled Coxeter groups with Pontryagin sphere boundary also using round trees.

\vskip.05in

\noindent {\bf Further reading.} This article presents only a tiny fraction of the power of the analytic perspective to studying boundaries. There are many excellent sources on the structure of boundaries of hyperbolic groups and their quasisymmetric structure. This list includes texts by the following authors and the references therein: Bonk~\cite{bonk-ICM}; Bourdon~\cite{bourdon-mostowtype}; Bridson--Haefliger~\cite{bridsonhaefliger}; Buyalo--Schroeder~\cite{buyaloschroeder}; Dru\c{t}u--Kapovich~\cite{drutukapovich}; Ha\"{i}ssinsky~\cite{haissinsky-bourbaki, haissinsky-QM}; Heinonen~\cite{heinonen-Lectures};  Kapovich--Benakli \cite{kapovichbenakli}; Kleiner~\cite{kleiner-ICM}; Mackay~\cite{mackay-survey}, and Mackay--Tyson~\cite{mackaytyson}.

\vskip.05in

\noindent {\bf Remarks on the text.} This article aims to provide a unified, elementary treatment to serve those new to the topic of visual metrics on boundaries and conformal dimension. In particular, we intended to present the background required to understand lower bounds on conformal dimension achieved via the round tree technique. In the last section we survey such results.
We aimed to provide a plethora of figures and examples as well as proofs, especially in the case that the author could not find one freely available in English.

\subsection*{Acknowledgements} The author thanks Matthew Durham and Thomas Koberda for organizing the Riverside Workshop in Geometric Group Theory 2024 and thanks the participants for many interesting questions and discussions. The author thanks Arianna Zikos for providing an introductory lecture to the minicourse and for many useful discussions. The author is thankful for helpful discussions with Lydia Ahlstrom, Dave Constantine, Yu-Chan Chang, Adrienne Nolt, Felipe Ramirez, Eleanor Rhoads, Elvin Shrestha, and Daniel Woodhouse, and thanks John Mackay for comments on a draft of the article. The author was supported by NSF Grant No. DMS-2204339.

\clearpage
\tableofcontents

\clearpage

    \section{Hyperbolicity}

    This section recalls three notions of hyperbolicity: $\delta$-hyperbolicity, $(\delta)$-hyperbolicity, and the $\CAT(-1)$ condition. The relationship between these notions is relevant to the construction of visual metrics and is discussed as well. Additional background can be found in the texts of Bridson--Haefliger~\cite{bridsonhaefliger}, Buyalo--Schroeder~\cite{buyaloschroeder}, and Ghys--de la Harpe~\cite{ghysdelaharpe}.

\subsection{$\delta$-hyperbolicity}

        \begin{defn}[Geodesics]
            Let $X$ be a metric space. A {\it geodesic} from $x \in X$ to $y \in X$ is a map $c:[0,\ell] \rightarrow X$ so that $c(0) = x$, $c(\ell) = y$ and $d\bigl(c(t), c(t')\bigr) = |t - t'|$ for all $t, t' \in [0,\ell]$. The image of $c$ in $X$ is called a {\it geodesic segment} from $x$ to $y$. A metric space $X$ is a {\it geodesic metric space} if there is a geodesic segment from $x$ to $y$ for all $x,y \in X$. A {\it geodesic triangle} in a metric space $X$ consists of three points in $X$ together with a choice of a geodesic segment (a {\it side}) between each pair of distinct points.
        \end{defn}

            The first notion of hyperbolicity applies to geodesic metric spaces and is given as follows.
            \begin{wrapfigure}{r}{0.25\textwidth}
                 \begin{center}
                    \begin{overpic}[width=.25\textwidth,tics=5]{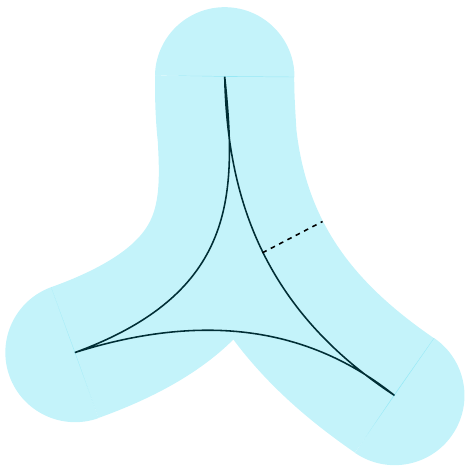} 
                        \put(64,42){\small{$\delta$}}
                    \end{overpic}
                 \end{center}
            \end{wrapfigure}

        \begin{defn}[$\delta$-slim triangles and $\delta$-hyperbolicity]
            Let $\delta>0$. A geodesic triangle in a metric space is {\it $\delta$-slim} if each of its three sides is contained in the $\delta$-neighborhood of the union of the other two sides. A geodesic metric space $X$ is {\it $\delta$-hyperbolic} if every triangle in $X$ is $\delta$-slim. A finitely generated group is {\it $\delta$-hyperbolic} (or {\it hyperbolic}) if it acts geometrically, ie properly discontinuously and cocompactly by isometries, on a $\delta$-hyperbolic metric space.
        \end{defn}

        \begin{example}
            A geodesic metric space is $0$-hyperbolic if and only if it is a metric tree.
        \end{example}

        \begin{example}
            The Cayley graph of the free product $\Z/n\Z \, * \, \Z/n\Z = \la \, a,b \,|\, a^n,\,  b^n \, \ra$ with respect to the generating set $\{a,b\}$ is $n$-hyperbolic, and $n$ is not the optimal $\delta$.
        \end{example}

            Foundational to geometric group theory is that hyperbolicity is preserved by quasi-isometries.

    \begin{defn}
        Let $K \geq 1$ and $C \geq 0$. A map $f:X \rightarrow X'$ is a {\it $(K,C)$-quasi-isometry} if the following two conditions hold.
        \begin{enumerate}
            \item For all $x,y \in X$,
            \[\frac{1}{K}d(x,y) - C \leq d\bigl(f(x),f(y)\bigr) \leq Kd(x,y) +C.\]
            \item For all $x' \in X'$ there exists $x \in X$ so that $d\bigl(f(x),x'\bigr) \leq C$.
        \end{enumerate}
    \end{defn}

        \begin{thm}[Quasi-isometry invariance] \cite[Theorem III.H.1.9]{bridsonhaefliger}
            Let $X$ and $X'$ be quasi-isometric geodesic metric spaces. If $X$ is $\delta$-hyperbolic, then $X'$ is $\delta'$-hyperbolic. \qed
        \end{thm}

    \subsection{Gromov product}

            The Gromov product is sometimes referred to as the {\it Gromov overlap function} and has natural geometric interpretations in a hyperbolic metric space. The Gromov product is central to defining metrics on the boundary of such a space.

        \begin{defn}[Gromov product]
            Let $X$ be a metric space, and let $p \in X$. The {\it Gromov product} of $x,y \in X$ based at $p$ is denoted $(x,y)_p$ and is given by the formula
                \[(x,y)_p := \frac{1}{2}\Bigl( d(p,x) + d(p,y) - d(x,y) \Bigr). \]
        \end{defn}

        \begin{example}
            The Gromov product has a simple meaning in a tree: the quantity $(x,y)_p$ equals the distance from $p$ to the geodesic $[x,y]$. Further, the Gromov product $(x,y)_p$ measures the distance the geodesics $[p,x]$ and $[p,y]$ fellow-travel before diverging.
        \end{example}

        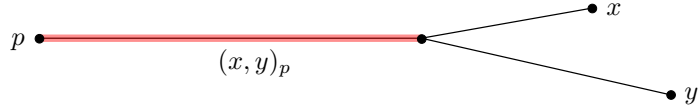
\begin{figure}[h]
                \begin{centering}
	            \begin{overpic}[width=.6\textwidth, tics=5]{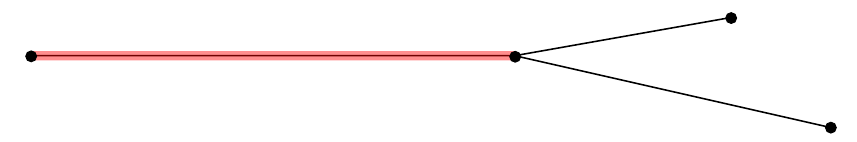} 
                    \put(-.5,9){$p$}
                    \put(87.5,13.5){$x$}
                    \put(99,1){$y$}
                    \put(30,5.5){$(x,y)_p$}
                \end{overpic}
	           \caption{\small{The Gromov product in a metric tree. }}
                \end{centering}
            \end{figure}

            The above example generalizes to a $\delta$-hyperbolic space: the Gromov product between $x$ and $y$ with respect to a point $p$ captures, up to a few $\delta$, how far $p$ is from the geodesic between $x$ and $y$.

        \begin{lem} \cite[Page 410; Proposition III.H.1.17]{bridsonhaefliger}
            Let $X$ be a $\delta$-hyperbolic geodesic metric space. Let $p,x,y \in X$, and let $[x,y]$ denote a geodesic segment from $x$ to $y$. Then, for all $x,y, p \in X$,
                \[ | \,(x,y)_p - d\bigl(p, [x,y]\bigr) \,| \leq 4\delta. \]
        \end{lem}

        In a hyperbolic metric space, one may also view the Gromov product $(x,y)_p$ as measuring (up to a few $\delta$) how long the geodesic segments from $p$ to $x$ and $y$ fellow-travel before diverging at close to maximal speed.

        In an arbitrary geodesic metric space, the Gromov product $(x,y)_p$ measures the distance between the point $p$ and the {\it equiradial points} on the sides of the triangle $\Delta(pxy)$ from $p$ to $x$ and $y$, as indicated on the triangle below ~\cite[Section 1.2]{buyaloschroeder}; these are also known as {\it internal points}~\cite[Section III.H.1]{bridsonhaefliger}. That is, the triangle inequality implies there are three numbers $(x,y)_p$, $(p,x)_y$, and $(p,y)_x$ so that

        \begin{figure}[h]
                \begin{centering}
	            \begin{overpic}[width=.9\textwidth,  tics=5]{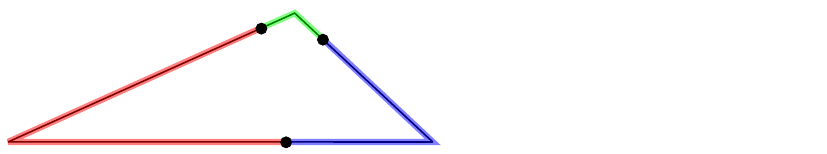} 
                    \put(-2,1){$p$}
                    \put(34.8,18.5){$x$}
                    \put(54,1){$y$}
                    \put(15,4){$(x,y)_p$}
                    \put(31.5,11){$(p,y)_x$}
                    \put(38,4){$(p,x)_y$}
                    \put(65,15){$d(x,p) \,=\, (x,y)_p + (p,y)_x$}
                    \put(65,10){$d(y,p) \,=\, (x,y)_p + (p,x)_y$}
                    \put(65,5){$d(x,y) \,=\, (p,x)_y + (p,y)_x$.}
                \end{overpic}
                \end{centering}
            \end{figure}

    \subsection{$(\delta)$-hyperbolicity}

        An alternative definition of hyperbolicity, denoted ($\delta$)-hyperbolicity and defined below, applies to metric spaces that need not be geodesic. This second notion is given using the Gromov product. The $(\delta)$-hyperbolic definition is a 4-point condition that will also be needed to prove the existence of visual metrics.
        \begin{defn}[$\delta$-inequality and $(\delta)$-hyperbolicity]
            Let $\delta>0$. Let $X$ be a metric space and $p \in X$. A triple of points $x,y,z \in X$ satisfy the {\it $\delta$-inequality} if
            \[(x,y)_p \geq \min\bigl\{ (x,z)_p, (y,z)_p)\bigr\} - \delta.\]      A metric space $X$ is {\it $(\delta)$-hyperbolic} if for all $p \in X$, every triple of points $x,y,z \in X$ satisfies the $\delta$-inequality.
        \end{defn}

            Note that $x,y,z$ satisfy the $\delta$-inequality with respect to a basepoint $p$ if and only if the two smallest of the numbers $(x,y)_p$, $(x,z)_p$, and $(y,z)_p$ differ by at most $\delta$.

        \begin{example}
            ($\delta$-inequality in a tree.) The $\delta$-inequality holds in a tree $T$ with constant $\delta= 0$. The two cases to consider, up to symmetry, when all three points $x,y,z$ are in the same component of $T - \{p\}$ are illustrated below. Otherwise, two of the three terms in the inequality equal zero.
                \begin{figure}[h]
                \begin{centering}
	            \begin{overpic}[width=.9\textwidth, tics=5]{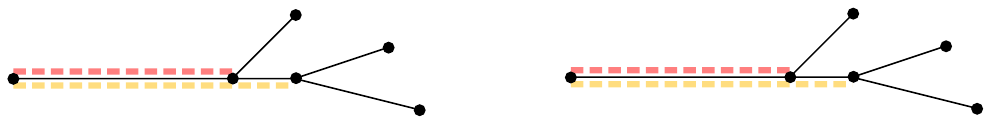} 
                    \put(-1.5,4){$p$}
                    \put(30.8,10.5){$x$}
                    \put(40,7){$y$}
                    \put(43,.7){$z$}
                    \put(5,7.5){$(x,y)_p,\, (x,z)_p$}
                    \put(10,.5){$(y,z)_p$}
                    \put(-53,4){$p$}
                    \put(86,10.5){$z$}
                    \put(95.8,7){$x$}
                    \put(98.5,1){$y$}
                    \put(61,7.5){$(x,z)_p,\, (y,z)_p$}
                    \put(69,.5){$(x,y)_p$}
                \end{overpic}
	           \caption{\small{Checking the $\delta$-inequality in a tree. }}
                \end{centering}
            \end{figure}
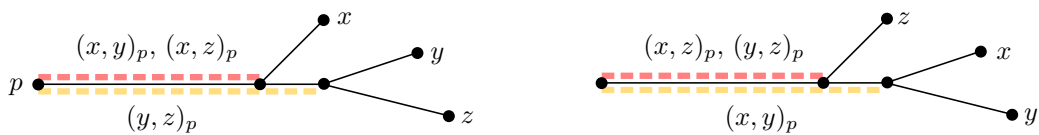

        \end{example}

         \begin{remark}
             The $(\delta)$-hyperbolicity definition also has a natural accompanying picture that indicates a notion of thinness. The definition of Gromov product implies that a metric space $X$ is $(\delta)$-hyperbolic if and only if for all $p,x,y,z \in X$,
            \[d(x,p)+d(y,z) \leq \max\bigl\{d(x,y)+d(z,p), \, d(x,z) + d(y,p) \bigr\} + 2\delta. \]
            As shown in the figure below, the inequality fails in $\mathbb{E}^2$ with the four points the corners of a square of increasing side length. The inequality holds in a tree with $\delta=0$. In the $\delta$-hyperbolic space, the geodesics $[x,z]$ and $[y,p]$ track close to the union of the geodesics $[x,p]$ and $[y,z]$.
            \begin{figure}[h]
                \begin{centering}
	            \begin{overpic}[width=.9\textwidth,tics=5]{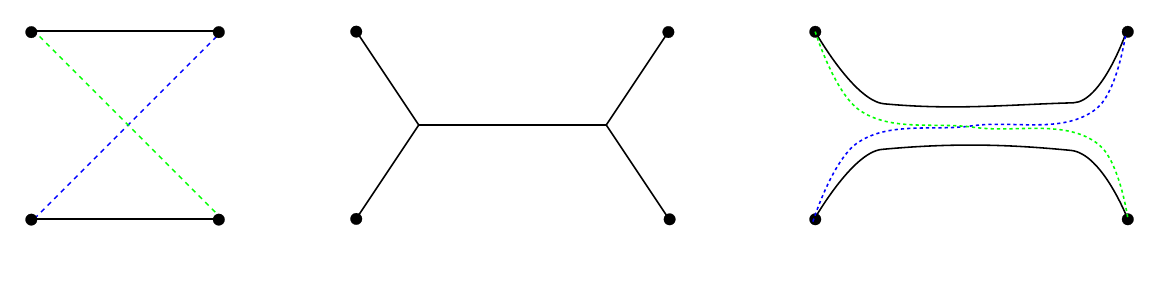} 
                    \put(0,22.5){$x$}
                    \put(0,6.5){$y$}
                    \put(20.5,22.5){$p$}
                    \put(20.5,6.5){$z$}
                    \put(28,22.5){$x$}
                    \put(28,6.5){$y$}
                    \put(59.5,22.5){$p$}
                    \put(59.5,6.5){$z$}
                    \put(68,22.5){$x$}
                    \put(68,6.5){$y$}
                    \put(99,22.5){$p$}
                    \put(99,6.5){$z$}
                    \put(8,2){$\mathbb{E}^2$}
                    \put(42,2){Tree}
                    \put(78,2){$\delta$-hyperbolic}
                \end{overpic}
	           \caption{\small{The 4-point hyperbolicity condition. }}
                \end{centering}
            \end{figure}
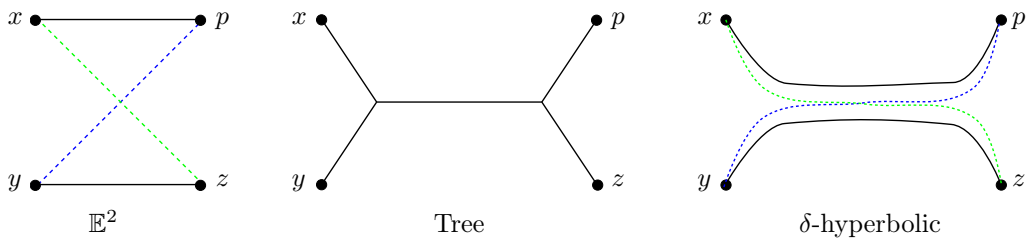
          \end{remark}

            The notions of $\delta$-hyperbolicity and $(\delta)$-hyperbolicity are related via the following lemma.

        \begin{lem} \cite[Proposition III.H.1.17, III.H.1.22]{bridsonhaefliger}
            Let $X$ be a geodesic metric space. Let $\delta>0$.
            \begin{enumerate}
                \item If $X$ is $\delta$-hyperbolic, then $X$ is $(4\delta)$-hyperbolic.
                \item If $X$ is $(\delta)$-hyperbolic, then $X$ is $6\delta$-hyperbolic. \qed
            \end{enumerate}
       \end{lem}

    \subsection{$\CAT(-1)$ metric spaces.}

            A quasi-isometry invariant, $\delta$-hyperbolicity is a robust notion of negative curvature. A finer, yet still pervasive, notion of hyperbolicity is given by the $\CAT(-1)$ condition.

        \begin{defn}
           Let $X$ be a metric space, and let $\Delta = \Delta(xyz)$ be a geodesic triangle in $X$ with vertices $x,y,z \in X$. A geodesic triangle $\bar{\Delta} = \Delta(\bar{x}\bar{y}\bar{z}) \subset \Hy^2$ with vertices $\bar{x}, \bar{y}, \bar{z} \in \Hy^2$ is a {\it comparison triangle} for $\Delta$ if $d(x,y) = d(\bar{x},\bar{y})$, $d(y,z) = d(\bar{y},\bar{z})$, and $d(x,z) = d(\bar{x},\bar{z})$. If $p$ is contained in the geodesic side of $\Delta$ from $x$ to $y$, then its {\it comparison point} $\bar{p} \in \bar{\Delta}$ is the unique point in the geodesic side from $\bar{x}$ to $\bar{y}$ so that $d(x,p) = d(\bar{x},\bar{p})$. Comparison points on the other sides are defined analogously.

           A geodesic triangle $\Delta$ in $X$ satisfies the {\it $\CAT(-1)$ inequality} if for all $p,q \in \Delta$, their comparison points $\bar{p},\bar{q} \in \bar{\Delta}$ satisfy $d(p,q) \leq d(\bar{p},\bar{q})$. A metric space $X$ is $\CAT(-1)$ if $X$ is a geodesic metric space and all of its geodesic triangles satisfy the $\CAT(-1)$ inequality.
        \end{defn}

        We will primarily be interested in $\CAT(-1)$ metric spaces in this article, though there is a well-studied notion of a $\CAT(\kappa)$ space for all $\kappa \in \R$. In particular, in the definition above, the hyperbolic plane is replaced with the model space $M_{\kappa}^2$, which equals the Euclidean plane if $\kappa =0$ and is obtained by scaling either the hyperbolic plane or the unit sphere otherwise.

        Because the geometry of the hyperbolic plane is explicit and well-studied, there are important consequences of $\CAT(-1)$ geometry in the study of visual metrics. A $\CAT(-1)$ metric space is $\delta$-hyperbolic, where $\delta$ can be computed using the geometry of the hyperbolic plane.
        Whether these notions of hyperbolicity agree in the setting of groups is a long-standing open question.

        \begin{openquestion}
                Does every hyperbolic group act geometrically on a $\CAT(-1)$ metric space?
        \end{openquestion}

        In fact, it is also open whether every hyperbolic group acts geometrically on a $\CAT(0)$ metric space.

    \section{The boundary of a hyperbolic space}

 The boundary of a hyperbolic metric space is a topological space that captures the set of directions to infinity in the space. This section recalls two definitions of the boundary and its topology. The first definition applies to geodesic metric spaces and is given in terms of geodesic rays. The second definition applies to arbitrary $(\delta)$-hyperbolic metric spaces and is given in terms of sequences of points. Visual metrics on the boundary will be defined with respect to the second definition. Throughout this section, let $X$ be a $(\delta)$-hyperbolic metric space.

    \subsection{Definition via rays}

        \begin{defn}
            Let $X$ be a metric space. Let $c,c':[0,\infty) \rightarrow X$ be geodesic rays in $X$. The rays $c$ and $c'$ are {\it asymptotic} if there exists a constant $D\geq 0$ so that $d\bigl(c(t),c'(t)\bigr) \leq D$ for all $t \in [0,\infty)$.
        \end{defn}

        Being asymptotic is an equivalence relation on geodesic rays, leading to the following definition.

        \begin{defn}
            Let $X$ be a geodesic metric space. The {\it boundary} of $X$, denoted $\partial X$, is the set of equivalence classes of geodesic rays in $X$, where two rays are equivalent if they are asymptotic. If $c:[0,\infty)\rightarrow X$ is a geodesic ray, let $c(\infty)$ denote the equivalence class of $c$.
        \end{defn}

        \begin{remark}[Quasi-geodesic rays]
            In a hyperbolic metric space, the boundary can equivalently be defined in terms of quasi-geodesic rays. Recall, a map $c:[0,\infty) \rightarrow X$ is a {\it quasi-geodesic ray} if $c$ is a quasi-isometric embedding. Quasi-geodesic rays are defined to be {\it asymptotic} if the Hausdorff distance between their images is finite. Let $\partial_q X$ denote the set of equivalence classes of quasi-geodesic rays. Then, the map $\partial X \rightarrow \partial_q X$ is a bijection; see \cite[Lemma III.H.3.1]{bridsonhaefliger}.
        \end{remark}

        \begin{example}
            The boundary of the line $\R$ with the standard metric consists of two points. The boundary of the infinite ladder graph below, where edges have length one, also consists of two points. Additional examples are given at the end of this section.

            \begin{figure}[h]
                \begin{centering}
	            \begin{overpic}[width=.6\textwidth,tics=5]{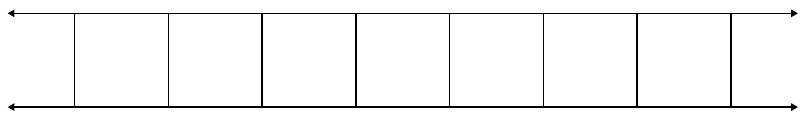} 

                \end{overpic}
	           \caption{\small{A metric space with visual boundary consisting of two points. }}
                \end{centering}
            \end{figure}
        \end{example}

    \subsection{Topology on the boundary}

        The boundary of $X$ admits a topology that compactifies the metric space $X$. We first specify notation for the union of $X$ and its boundary.

        \begin{defn}
            Let $X$ be a metric space. A {\it generalized geodesic ray} in $X$ is a geodesic $c:I \rightarrow X$, where either $I = [0,R]$ for some $R\geq 0$ or $I = [0,\infty)$. If $I = [0,R]$, define $c(t) := c(R)$ for all $t \in [R,\infty]$. Let
            \[\overline{X} = X \cup \p X = \{\, c(\infty) \,|\, c \textrm{ is a generalized geodesic ray} \,\}.\]
        \end{defn}

    The topology on $\overline{X}$ can be defined via convergence. Take the topology on the set of generalized geodesic rays to be the topology of uniform convergence on compact sets. Then take the quotient topology on $\overline{X}$.

    The topology on $\overline{X}$ can alternatively be specified via the {\it neighborhood topology}, by specifying a collection of neighborhoods of each point in $\overline{X}$. These neighborhoods need not be open in the boundary; see~\Cref{ex:NotOpen}. Nonetheless, they satisfy four natural axioms, recalled here for completeness, from which a topology can be defined; see \cite[Chapter 1.2]{bourbaki}. That these topologies are the same follows from ~\cite[Lemma III.H.3.6]{bridsonhaefliger}.

    \begin{defn}[Neighborhood topology]
        Let $X$ be a set, and let $\cN$ be a function from $X$ to $\cP\bigl(\cP(X)\bigr)$, where $\cP(X)$ denotes the powerset of $X$. Elements of $\cN(x)$ are called {\it neighborhoods} of $x$. Then $\cN$ is called a {\it neighborhood topology} if it satisfies the following four axioms:
            \begin{enumerate}
                \item If $N \in \cN(x)$, then $x \in N$.
                \item If $N \subset N'$ and $N \in \cN(x)$, then $N' \in \cN(x)$.
                \item If $N, N' \in \cN(x)$, then $N \cap N' \in \cN(x)$.
                \item If $N \in \cN(x)$, then there exists $M \in \cN(x)$ so that if $y \in M$, then $N \in \cN(y)$.

                This axiom captures the notion that a neighborhood of a point $x$ is also a neighborhood of every point sufficiently close to $x$.
            \end{enumerate}
        If $\cN$ is a neighborhood topology, then a topology on $X$ can be specified by $U \subset X$ is {\it open} if $U$ is a neighborhood of every point in $U$. In this topology, every element of $\cN(x)$ is a {\it neighborhood} of $x$ in the sense that if $V \in \cN(x)$, then there exists an open set $U$ so that $x \in U \subset V$.
    \end{defn}

    \begin{defn}[Neighborhoods of a point] \label{defn:neighborhoods}
        Let $X$ be a hyperbolic proper geodesic metric space, and let $p \in X$. Let $\cN$ be the function from $X$ to $\cP\bigl(\cP(X)\bigr)$ defined as follows. For each $x \in \overline{X}$, let $\cN(x)$ the the collection of subsets of $\overline{X}$ that contain a set in $\cN'(x)$ given as follows.

        \begin{enumerate}
            \item (Points in $X$.) If $x \in X$, let $\cN'(x)$ be the family of metric balls centered $x$.

\vskip.07in

            \item (Points in $\p X$.) Fix $k>2\delta$.

            Let $c(\infty) \in \p X$. Fix $c_0:[0,\infty) \rightarrow X$ a geodesic ray with $c_0(0) =p$ and $c_0(\infty) = c(\infty)$. For each $n \in \Z_+$, let
            \[ V_n(c_0):= \left\{
                            \begin{array}{c}
                                \\
                                \,\, c'(\infty) \,\, \\
                                \\
                            \end{array}
                            :
                             \begin{array}{c}
                                c' \textrm{ is a generalized geodesic ray }\\
                                c'(0) =p \\
                                d\bigl(c'(n), c_0(n) \bigr) < k\\
                            \end{array}
                            \right\}\]
            Let $\cN'\bigl(c(\infty)\bigr) = \bigl\{V_n(c_0) \,|\, n \in \Z_+\bigr\}$.
        \end{enumerate}
    \end{defn}

    \begin{figure}[t]
                \begin{centering}
	            \begin{overpic}[width=.6\textwidth,  tics=5]{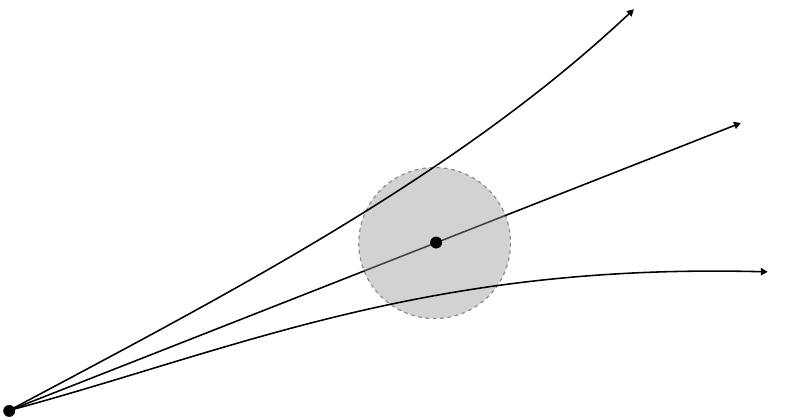} 
                    \put(-2,0){\small{$p$}}
                    \put(55,19){\small{$c_0(n)$}}
                    \put(90,32){$c_0(t)$}
                \end{overpic}
	           \caption{\small{A neighborhood in $\p X$ of the boundary point represented by the ray $c_0$ consists of equivalence classes of rays with a representative passing through the shaded ball. }}
                \end{centering}
            \end{figure}

    \begin{lem}
        The function $\cN$ given in \Cref{defn:neighborhoods} is a neighborhood topology.
    \end{lem}

    For the remainder of the article, the topology on $\overline{X}$ is the topology induced by the neighborhood topology, and the topology on $\p X$ is the subspace topology.

    \begin{example}[Neighborhoods need not be open] \label{ex:NotOpen}
        As in \Cref{defn:neighborhoods}, let $c_0:[0,\infty) \rightarrow X$ be a geodesic ray based at a point $p \in X$. Then the neighborhood $V_n(c_0)$ consists of all generalized geodesic rays $c'$ based at $p$ and so that
            \[d\bigl(c'(n), c_0(n)\bigr) < k.\]
       The strict inequality above may (at first) suggest these sets should be open. Indeed, these sets are open in trees and in the hyperbolic plane, for example. However, they may fail to be open if geodesics in the space $X$ are not unique. Morally, a point on the boundary may have {\it some} geodesic representative passing through the open ball, but there are boundary points that are arbitrarily close for which there is no geodesic representative passing through that open ball. Indeed, this occurs in the following example. See \Cref{figure-ex_neigh}.

       \begin{figure}
                \begin{centering}
	            \begin{overpic}[width=.4\textwidth,  tics=5]{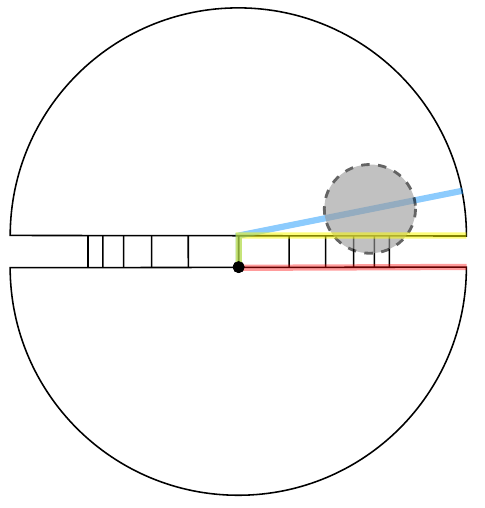} 
                    \put(0,90){$X$}
                    \put(93,62.5){$c_0(\infty)$}
                    \put(45,42){$p$}
                    \put(9,50){$\ldots$}
                    \put(80,50){$\ldots$}
                \end{overpic}
	           \caption{\small{An example illustrating that neighborhoods in the boundary need not be open. The hyperbolic space $X$ is formed by cutting the hyperbolic plane into two pieces and gluing them together along intervals of length one evenly spaced along the boundary geodesics. The neighborhood of $c_0(\infty)$ in $\p X$ given by equivalence classes of rays passing through the open shaded ball is not open in $\p X$; it is a half-open interval contained in the circle boundary of $X$. }}
	           \label{figure-ex_neigh}
                \end{centering}
            \end{figure}
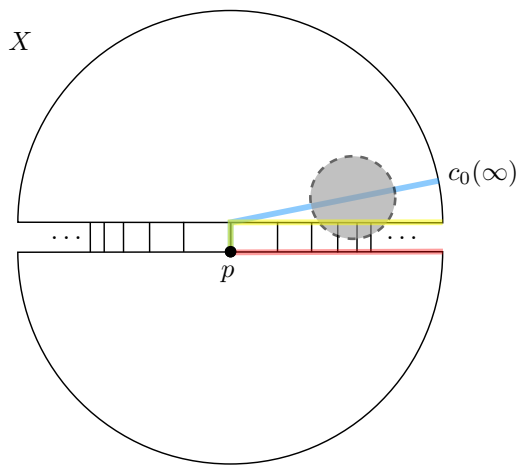

       View the hyperbolic plane $\Hy^2$ in the unit disk model.
       Let $\gamma: \R \rightarrow \Hy^2$ be the geodesic line in the hyperbolic plane traveling along the real axis and with $\gamma(0)$ the origin. Cut the hyperbolic plane along $\gamma$ into two halfplanes; label the two copies of the line $\gamma$ by $\gamma$ and $\gamma'$. Glue these spaces back together by attaching, for each $n \in \Z$, an interval of length one so that one endpoint is glued to $\gamma(n)$ and the other endpoint is glued to $\gamma'(n)$. The resulting space $X$ is a $\delta$-hyperbolic proper geodesic metric space.

       Let $V_n(c_0)$ be the following set. Let $p = \gamma'(0)$.
       Let $c_0$ be the concatenation of the geodesic segment from $p$ to $\gamma(0)$ and a geodesic ray making a small angle with $\gamma|_{[0,\infty)}$. Then $c_0$ is a geodesic ray. Fix $n \in N$ so that the open ball of radius $k$ about $c_0(n)$ intersects $\gamma|_{[0,\infty)}$. Then $c' := \gamma|_{[0,\infty)} \subset V_n(c_0)$. However, there are points in the boundary arbitrarily close to $c'(\infty)$, those coming from geodesic rays in the other half of $\Hy^2$ that are arbitrarily close to $\gamma'|_{[0,\infty)}$, with no geodesic representative passing through this open ball. Indeed, while geodesic rays can cross the infinite ladder, no geodesic will cross twice. In particular $V_n(c_0)$ is not open because $V_n(c_0)$ is not a neighborhood of $c'$. The boundary $\p X$ is a circle, and $V_n(c_0)$ restricted to the boundary is a half-open interval.

       The space $X$ does not admit a geometric group action, but one can easily alter the construction. Indeed, fix a hyperbolic metric on a genus two surface. Cut the surface along a separating curve and re-glue the boundary components to the boundary components of a cellular Euclidean annulus. Take the universal cover, restricting to the one-skeleton on each lift of the annulus.
    \end{example}

        Crucial to the study of boundaries is the following theorem; see \cite[Theorem III.H.3.9]{bridsonhaefliger}.

        \begin{thm}[Quasi-isometry invariance] \label{thm:QI_inv_boundary}
            Let $X$ and $X'$ be proper geodesic hyperbolic metric spaces. Let $f: X \rightarrow X'$ be a quasi-isometry. Then there is a well-defined map $f_\partial:\p X \rightarrow \p X'$ defined by $f_\partial([c]) = [f \circ c]$. The map $f_\partial$ is a homeomorphism.
        \end{thm}

    \subsection{Definition via sequences}

        \begin{defn}
            A sequence of points $\{x_i\}$ in $X$ {\it converges to infinity} if for some $p \in X$,
                \[ \lim_{i,j \rightarrow \infty} (x_i,x_j)_p = \infty.\]
        \end{defn}

        It follows from the definition of the Gromov product and the triangle inequality that for all $x_i,x_j, p,p' \in X$,
            \[ |(x_i,x_j)_p - (x_i,x_j)_{p'}| \leq d(p,p'). \]
        Hence, the above definition is independent of the choice of $p \in X$.

        \begin{defn} \label{defn:equivalence}
            Two sequences $\{x_i\}$ and $\{x_i'\}$ in $X$ that converge to infinity are {\it equivalent} if
                \[ \lim_{i \rightarrow \infty} (x_i,x_i')_p = \infty.\]
        \end{defn}

        As above, this notion is independent of the choice of basepoint $p$.

        \begin{example}
            Suppose the figure below is a metric graph with edge lengths equal to one. Then,
                    \[(x_i,x_j)_p = (x_i',x_j')_p = \min\{i,j\} \quad \textrm{ and } \quad (x_i,x_i')_p = i.\]

        \begin{figure}[h]
    \begin{centering}
	\begin{overpic}[width=.5\textwidth, tics=5]{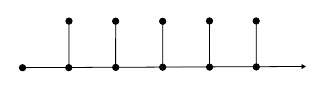} 
        \put(2,5){$p$}
        \put(19,2){\small{$x_1$}}
        \put(19,23.5){\small{$x_1'$}}
        \put(33,2){\small{$x_2$}}
        \put(33,23.5){\small{$x_2'$}}
        \put(48,2){\small{$x_3$}}
        \put(48,23.5){\small{$x_3'$}}
        \put(62,2){\small{$x_4$}}
        \put(62,23.5){\small{$x_4'$}}
        \put(77,2){\small{$x_5$}}
        \put(77,23.5){\small{$x_5'$}}
        \put(90,15){$\ldots$}
    \end{overpic}
	\caption{\small{Sequences converging to the same point on the boundary. }}
	\label{figure-sequence_ex}
    \end{centering}
    \end{figure}
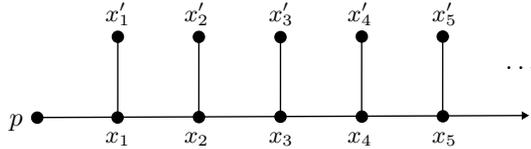

   \noindent Hence, both sequences $\{x_i\}$ and $\{x_i'\}$ converge to infinity, and the sequences are equivalent.
        \end{example}

        The $\delta$-inequality implies that being equivalent in the sense of Definition~\ref{defn:equivalence} is an equivalence relation on the set of sequences in $X$ converging to infinity. Indeed, suppose $\{x_i\}$, $\{x_i'\}$, and $\{x_i''\}$ are sequences converging to infinity and
            \[\lim_{i \rightarrow \infty} (x_i, x_i')_p = \infty \quad \textrm{ and } \quad \lim_{i \rightarrow \infty} (x_i', x_i'')_p = \infty.\]
        Since $(x_i,x_i'')_p \geq \min\bigl\{ (x_i,x_i')_p, (x_i', x_i'')_p\bigr\} - \delta$, this implies that $\lim_{i \rightarrow \infty} (x_i,x_i'')_p = \infty$.

        \begin{defn} (Sequential boundary.)
            The set of equivalence classes of sequences converging to infinity is called the {\it sequential boundary} of $X$ and is denoted $\partial_s X$.
        \end{defn}

        The sequential boundary $\partial_s X$ agrees with the boundary $\p X$ in the case that $X$ is a proper geodesic hyperbolic metric space:

        \begin{lem} \cite[Lemma 3.13]{bridsonhaefliger}
            If $X$ is a proper geodesic $\delta$-hyperbolic metric space, then there is a bijection $\p_s X \rightarrow \p X$.
        \end{lem}

    \subsection{Examples}

    \begin{defn}
         The {\it boundary} of a hyperbolic group $\Gamma$ is defined to be the boundary of any hyperbolic geodesic metric space on which $\Gamma$ acts geometrically.
    \end{defn}

      The definition above is well-defined by~\Cref{thm:QI_inv_boundary}.

        \begin{example}
            The boundary of the $n$-valent tree for $n \geq 3$ is homeomorphic to the Cantor set.
        \end{example}

        \begin{example}
            The boundary of real hyperbolic space $\Hy^n$ is homeomorphic to the sphere of dimension $n-1$.
        \end{example}

        \begin{example}[Free product of surface groups]
            Let $S$ and $S'$ be homeomorphic to the closed orientable surface of genus $g \geq 2$ and equipped with a hyperbolic metric. Let $p \in S$ and $p' \in S'$.
            Let $Y = (S \sqcup [0,1] \sqcup S')/\sim$ be the quotient space defined by $0 \sim p$ and $1 \sim p'$. Let $\widetilde{Y}$ be the universal cover of $Y$. Then $\p \widetilde{Y}$ is disconnected and consists of countably many circles and uncountably many singletons, each of which is a limit of circles.
        \end{example}

        The previous example sits in the broader framework of boundaries of free products of one-ended hyperbolic groups. This structure is detailed by Martin--\'{S}wi\k{a}tkowski \cite{martinswiatkowski}.

        \begin{example}[Surface amalgams] Let $S$ and $S'$ be homeomorphic to the closed orientable surface of genus $g \geq 2$ and equipped with a hyperbolic metric. Let $\gamma$ and $\gamma'$ be essential simple closed curves on $S$ and $S'$. Let $Y$ be the space obtained by gluing $\gamma$ and $\gamma'$ together by a homeomorphism. Then, the boundary of the universal cover $\widetilde{Y}$ is connected. A geodesic ray either terminates in a circle corresponding to the boundary of a hyperbolic plane covering $S$ or $S'$, or the geodesic ray travels through infinitely many copies of such hyperbolic planes.

            There are cut pairs in the boundary corresponding to the boundary of a geodesic line in $\widetilde{Y}$ that does not cross a lift of the glued curves; these cut pairs separate the boundary into two components. In addition, there are cut pairs in the boundary corresponding to lifts of the glued curves that split the boundary into four components. (Recall, a {\it cut pair} in a topological space $Z$ is a set of distinct points $\{z,z'\} \subset Z$ so that $Z \setminus \{z,z'\}$ is not connected.)
        \end{example}

        Bowditch utilized the structure of cut pairs in the boundary of a one-ended hyperbolic group to produce a quasi-isometry-invariant JSJ theory, which captures all splittings of the group over two-ended subgroups. Indeed, the boundary perspective is powerful as topological features of the boundary of a hyperbolic group correspond to algebraic and geometric properties of the group. The next theorem highlights two significant, celebrated cases of this relationship. We stated only a consequence of the influential work of Bowditch below and refer to the paper for a detailed description of the JSJ theory.

        \begin{thm}
            Let $\Gamma$ be a hyperbolic group. Then,
            \begin{enumerate}
                \item \cite[Theorem 6.2]{bowditch} Let $\Gamma$ be a one-ended hyperbolic group that is not cocompact Fuchsian. Then $\Gamma$ splits over a two-ended subgroup if and only if $\p \Gamma$ contains a local cut point.
                \item \cite[Corollary 1.4]{bestvinamess}
       Let $\Gamma$ be a hyperbolic group. Then
        \[ \dim \p \Gamma = \max\{n\,|\, H^n(\Gamma, \Z\Gamma) \neq 0\}. \]
        If $\Gamma$ is torsion-free, then $\dim \p \Gamma = \cd \Gamma -1$, where $\cd \Gamma$ denotes the cohomological dimension of $\Gamma$.

            \end{enumerate}
        \end{thm}

         In certain special cases, the topology of the boundary of a group completely determines the algebraic structure of the group, up to abstract commensurability.

        \begin{thm}
            Let $\Gamma$ be a hyperbolic group. Then,
            \begin{enumerate}
            \item $\p \Gamma$ is empty if and only if $\Gamma$ is finite.
            \item $\p \Gamma$ consists of two points if and only if $\Gamma$ contains $\Z$ as a finite-index subgroup.
            \item \cite{stallings, dunwoody} $\p \Gamma$ is homeomorphic to the Cantor set if and only if $\Gamma$ contains a free group of rank $\geq 2$ as a finite-index subgroup.
            \item \cite{tukia, gabai, cassonjungreis} $\p \Gamma$ is homeomorphic to the circle if and only if $\Gamma$ contains the fundamental group of a closed orientable surface of genus $\geq 2$ as a finite-index subgroup.
            \end{enumerate}
        \end{thm}

        Cannon's Conjecture~\cite{cannon-conj}, one of the preeminent open problems in the field, asks a higher-dimensional analog of the last case. Namely, if a hyperbolic group has $2$-sphere boundary, is the group virtually Kleinian? A positive resolution to Cannon's Conjecture would also imply a positive answer to the following question, asked by Kapovich--Kleiner~\cite{kapovichkleiner}. If a hyperbolic group has boundary homeomorphic to the Sierpinski carpet, is the group virtually Kleinian?

        More is known in the low-dimensional setting about the topological structure of boundaries. For example, the next theorem completely classifies $1$-dimensional boundaries of groups that do not virtually split over a $2$-ended subgroup.

        \begin{thm} \cite[Theorem 4]{kapovichkleiner}
            Suppose $\Gamma$ is a one-ended hyperbolic group, $\p \Gamma$ is one-dimensional, and has no local cut points. Then $\p \Gamma$ is homeomorphic to either the Sierpinski carpet or the Menger curve.
        \end{thm}

    The following question is widely open.

    \begin{question} \label{ques:whichspaces}
            Which topological spaces arise as the boundary of a hyperbolic group?
    \end{question}

     In addition to the examples above, Menger compacta in dimensions 1, 2, and 3 arise as boundaries of groups acting by isometries on certain right-angled buildings~\cite{dymaraosajda}. Many trees of manifolds, including the Pontriagin sphere, arise as boundaries of hyperbolic groups~\cite{swiatkowski-treesOfManifolds}, as do trees of graphs~\cite{hodaswiatkowski}.

    We end this section by collecting fundamental topological properties that are known to hold for the boundary of every hyperbolic group. These properties restrict spaces that can serve as boundaries as in \Cref{ques:whichspaces}.

        \begin{thm}
            Let $\Gamma$ be a hyperbolic group. Then,
            \begin{enumerate}
                \item $\Gamma$ is one-ended if and only if $\p \Gamma$ is connected.
                \item \cite[Proposition III.H.3.7]{bridsonhaefliger} $\p \Gamma$ is compact.
                \item \cite{bowditch-treelike, levitt-nonnesting, swarup-cutpoint} $\p \Gamma$ does not have a global cut point.
                \item \cite[Section 3]{bestvinamess} If $\Gamma$ is one-ended, then $\p \Gamma$ is locally connected.
            \end{enumerate}
        \end{thm}

    \section{Extending the Gromov product to the boundary}

    For a hyperbolic geodesic metric space $X$, the Gromov product $(x,y)_p$ roughly measures the distance from $p$ to the geodesic $[x,y]$. The definition extends to the boundary of $X$ as follows.

    \begin{sassumption}
            Throughout the remainer of the article, unless otherwise stated, the space $X$ will be a $(\delta)$-hyperbolic metric space and $\p X$ will be its boundary, viewed as the set of equivalence classes of sequences.
    \end{sassumption}

    \begin{defn}
        Let $X$ be a $(\delta)$-hyperbolic metric space. The Gromov product extends to $X \cup \p X$ by the formula
            \[ (x,y)_p := \inf \liminf_{i,j \rightarrow \infty} (x_i,y_j), \]
        where the infimum is taken over all sequences $\{x_i\} \in x$ and $\{y_j\} \in y$.
    \end{defn}

       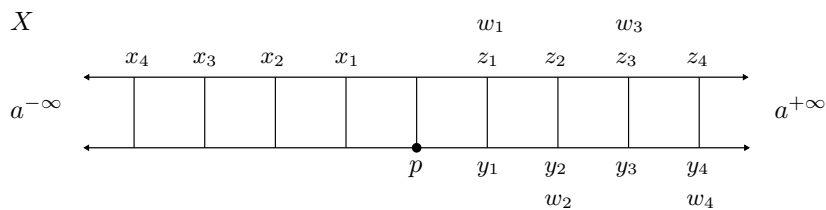
\begin{figure}[h]
                \begin{centering}
	            \begin{overpic}[width=.6\textwidth, tics=5]{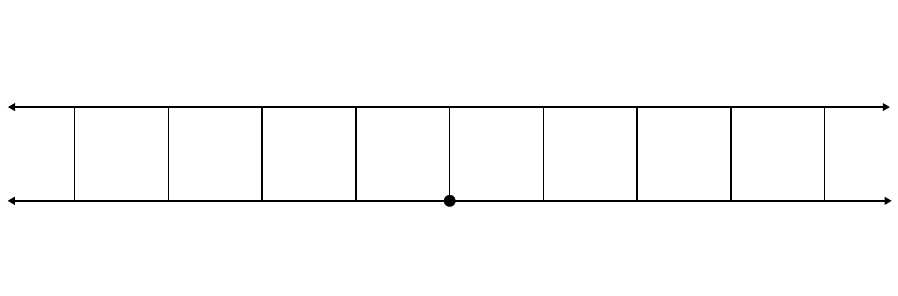} 
                    \put(-10,28){$X$}
                    \put(49,7){$p$}
                    \put(-10,15){$a^{-\infty}$}
                    \put(103,15){$a^{+\infty}$}
                    \put(59,23){\small{$z_1$}}
                    \put(59,28){\small{$w_1$}}
                    \put(59,7){\small{$y_1$}}
                    \put(69,7){\small{$y_2$}}
                    \put(69,2){\small{$w_2$}}
                    \put(69,23){\small{$z_2$}}
                    \put(79.5,7){\small{$y_3$}}
                    \put(79.5,23){\small{$z_3$}}
                    \put(79.5,28){\small{$w_3$}}
                    \put(90,7){\small{$y_4$}}
                    \put(90,2){\small{$w_4$}}
                    \put(90,23){\small{$z_4$}}
                    \put(38,23){\small{$x_1$}}
                    \put(27,23){\small{$x_2$}}
                    \put(17,23){\small{$x_3$}}
                    \put(7,23){\small{$x_4$}}
                \end{overpic}
	           \caption{\small{An example of Bridson--Haefliger that motivates the definition of the extension of the Gromov product to the boundary. }}
                \end{centering}
        \end{figure}

    \begin{example}[\cite{bridsonhaefliger}, Example III.H.3.16]
        The following example of Bridson--Haefliger illustrates why both an infimum and a $\liminf$ are needed to make the above definition well-defined. Consider the metric graph $X$ above, which can be viewed as the Cayley graph for the direct product $\Z \times \Z/2\Z$ (with the bigons collapsed to an edge). The boundary $\p X$ consists of two points, denoted $a^{+\infty}$ and $a^{-\infty}$. Then, $x_n \rightarrow a^{-\infty}$ while $y_n,z_n,w_n \rightarrow a^{+\infty}$. However, $\lim_{i,j \rightarrow \infty} (x_i,w_j)_p$ does not exist. Moreover, $\liminf_{i,j \rightarrow \infty}(x_i,y_j)_p \neq \liminf_{i,j \rightarrow \infty}(x_i,z_j)_p$.
    \end{example}

    \begin{remark}
        The Gromov product takes values in $[0,\infty]$, and $(\eta,\eta')_p = \infty$ if and only if $\eta,\eta' \in \p X$ and $\eta = \eta'$.
    \end{remark}

    The $\delta$-inequality extends to the boundary of $X$, as shown below in Lemma~\ref{lemma:delta_ineq_boundary}. This key property is needed to prove the existence of visual metrics on $\p X$. The proof uses the next helpful lemma, proven using a diagonal sequence argument, which says that, while the extension of the Gromov product to the boundary involves taking an infimum and a $\liminf$, there exist sequences that realize this infimum and via a limit rather than a $\liminf$.

    \begin{lem} \label{lemma:exist_seq}
        Let $X$ be a $(\delta)$-hyperbolic metric space. For all $\eta, \eta' \in \p X$, there exist sequences $\{x_n\}, \{x_n'\}$ of points in $X$ so that $\{x_n\} \in \eta$, $\{x_n'\} \in \eta'$, and
                \[(\eta,\eta')_p = \lim_{n \rightarrow \infty} (x_n,x_n')_p.\]
    \end{lem}

    We note that the above argument can be extended to any finite set of points in the boundary. The previous lemma is used to prove the next, which extends the $\delta$-inequality to the boundary.

    \begin{lem} \label{lemma:delta_ineq_boundary}
         Let $X$ be a $(\delta)$-hyperbolic metric space. For all $p \in X$ and $\eta, \eta', \eta'' \in \p X$
            \[ (\eta,\eta')_p \geq \min \bigl\{(\eta, \eta'')_p, (\eta'', \eta')_p\bigr\} - \delta. \]
    \end{lem}
    \begin{proof}
        Let $p \in X$ and $\eta, \eta', \eta'' \in \p X$.
        It follows from Lemma~\ref{lemma:exist_seq} and its proof that there exist sequences $\{x_i\} \in \eta$, $\{x_i'\} \in \eta'$, and $\{x_i''\} \in \eta''$ so that
             \[  (\eta,\eta')_p = \lim_{i \rightarrow \infty} (x_i,x_i')_p, \quad \quad
                    (\eta,\eta'')_p = \lim_{i \rightarrow \infty} (x_i,x_i'')_p, \quad \quad
                    (\eta',\eta'')_p = \lim_{i \rightarrow \infty} (x_i',x_i'')_p.\]
        Then, by the $\delta$-inequality,
            \begin{eqnarray*}
                (\eta, \eta')_p &=& \lim_{i \rightarrow \infty} (x_i,x_i')_p\\
                 &\geq &  \lim_{i \rightarrow \infty} \min \bigl\{ (x_i,x_i'')_p, (x_i'', x_i')_p \bigr\} - \delta \\
                 & \geq & \min \bigl\{ \lim_{i \rightarrow \infty} (x_i,x_i'')_p, \lim_{i \rightarrow \infty} (x_i'',x_i')_p\bigr\} - \delta \\
                 &= & \min \bigl\{ (\eta, \eta'')_p, (\eta'', \eta')_p  \bigr\} - \delta,
            \end{eqnarray*}
            as desired.
    \end{proof}

    \section{Visual metrics} \label{sec:visual_metrics}

    The Gromov product $(\eta,\eta')_p$ of two elements $\eta, \eta' \in \p X$ records how long the elements fellow-travel. So, the quantity
        \[ \rho(\eta,\eta'):= a^{-(\eta,\eta')_p}\]
    for $a>1$ is a good measure of separation: the longer $\eta$ and $\eta'$ fellow-travel, the closer the elements are to each other, and the smaller the value of $\rho(\eta,\eta')$. Moreover, $\rho$ is clearly symmetric, and $\rho(\eta,\eta') = 0$ if and only if $\eta = \eta'$. Thus, $\rho$ is a great candidate for a metric on the boundary. Indeed, Bourdon~\cite{bourdon-flot} proved that $\rho$ is a metric in the case that $X$ is $\CAT(-1)$ and $a=e$, which is discussed below. However, depending on the choice of $a$, the function $\rho$ can fail to satisfy the triangle inequality in $\delta$-hyperbolic spaces as seen in Example~\ref{example:tri_ineq}.

    This section details how a metric on the boundary can be obtained from the separation function~$\rho$. This construction is possible because the function $\rho$ satisfies a weak version of the triangle inequality, making $\rho$ a quasimetric. One obtains a metric by snowflaking the function $\rho$ and applying a chaining construction. Details of the construction are given in Subsection~\ref{subsec:vismet}; the $\CAT(-1)$ case is discussed in Subsection~\ref{subsec:CAT-1}; and examples are given in Subsection~\ref{subsec:VMexamples}, including a boundary with a visual metric that is not a geodesic metric space.

\subsection{Visual metric construction: metrics from quasimetrics} \label{subsec:vismet}

     \begin{defn} (Visual metric.)
        Let $X$ be a hyperbolic space with basepoint $p \in X$. A metric $d$ on $\p X$ is a {\it visual metric} if there exists a number $a>1$, called the {\it visual metric parameter}, and a constant $C\geq 1$ so that
            \[\frac{1}{C} a^{-(\eta,\eta')_p} \leq d(\eta,\eta') \leq Ca^{-(\eta,\eta')_p}\]
        for all $\eta,\eta' \in \p X$.
    \end{defn}

    The construction of visual metrics on boundaries follows from the more general setting of quasimetric spaces.

    \begin{defn} (Quasimetric.)
        Let $Z$ be a set. A function $q:Z \times Z \rightarrow [0,\infty)$ is a {\it quasimetric on $Z$ with parameter $K$} if
        \begin{enumerate}
            \item $q(z,w) = 0$ if and only if $z=w$,
            \item $q(z,w) = q(w,z)$ for all $z,w \in Z$, and
            \item $q(x,y) \leq K \max\big\{q(x,z), q(z,y)\big\}$ for all $x,y,z \in Z$ and $K \geq 1$.
        \end{enumerate}
    \end{defn}

    \begin{remark}
        Note that Condition~(3) in the quasimetric definition can also be expressed to more closely resemble the standard triangle inequality. Let Condition~(3') be \[q(x,y) \leq K' \bigl(q(x,z) + q(z,y)\bigr)\] for all $x,y,z \in Z$ and a constant $K' \geq 1$. Then, if $q$ satisfies Condition~(3), then $q$ satisfies Condition~(3') with $K' = K$, and if $q$ satisfies Condition~(3'), then $q$ satisfies Condition~(3) with $K = 2 K'$.
    \end{remark}

        Crucially, the separation function $\rho$ of points in the boundary of a $\delta$-hyperbolic space given above is a quasimetric with parameter dependent only on $\delta$:

    \begin{lem} \label{lem:hypqm}
        Let $X$ be a $(\delta)$-hyperbolic metric space, and let $p \in X$. Let $a>1$. The function $\rho:\p X \times \p X \rightarrow [0,\infty)$ given by \[\rho(\eta,\eta'):= a^{-(\eta,\eta')_p}\] is an $a^{\delta}$-quasimetric on $\p X$.
    \end{lem}
    \begin{proof}
        The function $\rho$ is clearly symmetric. Further, $(\eta, \eta')_p = \infty$ if and only if $\eta = \eta'$, so $\rho(\eta, \eta') = 0$ if and only if $\eta = \eta'$. The last condition follows from the extension of the $\delta$-inequality to the boundary as given in Lemma~\ref{lemma:delta_ineq_boundary}. That is, for all $\eta, \eta', \eta'' \in \p X$ the following hold:
        \begin{eqnarray*}
            (\eta,\eta')_p & \geq & \min \{ (\eta,\eta'')_p, (\eta'', \eta')_p\} - \delta, \\
            -(\eta,\eta')_p & \leq & \max \{ -(\eta,\eta'')_p, -(\eta'', \eta')_p\} + \delta, \\
            a^{-(\eta,\eta')_p} & \leq & a^{\delta}\max \{ a^{-(\eta,\eta'')_p}, a^{-(\eta'', \eta')_p}\},
        \end{eqnarray*}
        as desired.
    \end{proof}

    We now detail how a metric can be obtained from a quasimetric. We follow the proof of Buyalo--Schroeder~\cite[Section 2.2.2]{buyaloschroeder}; see also \cite[Proposition 14.5]{heinonen-Lectures}. The metric involves two simple constructions, the first of which is called the chain construction.

    \begin{defn}[Chain construction]
        Let $Z$ be a set, and let $q$ be a quasimetric on $Z$. The map {\it obtained from $q$ by the chain construction} is the map $d:Z \times Z \rightarrow [0,\infty)$ defined by
            \[ d(z,z'):= \inf \sum_i q(z_i,z_{i+1}),\]
        where the infimum is taken over all finite sequences of points $z=z_0, \ldots, z_{k+1} = z'$ in $Z$.
    \end{defn}

        It follows immediately from definition that $d$ is symmetric and satisfies the triangle inequality, which the quasimetric need not. However, the map $d$ may fail to be a metric as $d(z,z')$ could be zero for distinct $z,z' \in Z$; examples are given by Schroeder~\cite{schroeder06}. The next proposition shows that if the quasimetric constant is small enough, then the chain construction yields a metric. (Moreover, Schroeder's examples illustrate that $K\leq 2$ is the optimal constant.)

    \begin{prop} \label{prop:chain} \cite[Lemma 2.2.5]{buyaloschroeder}
        Let $q$ be a $K$-quasimetric on a set $Z$ with $K \leq 2$. Then the map $d$ obtained from $q$ by the chain construction is a metric on $Z$. Moreover, for all $z,z' \in Z$,
            \[\frac{1}{2K}q(z,z') \leq d(z,z') \leq q(z,z').\]
    \end{prop}
    \begin{proof}
        By definition, the map $d$ is symmetric, satisfies the triangle inequality, and $d(z,z') \leq q(z,z')$ for all $z,z' \in Z$. We will show that $q(z,z') \leq 2Kd(z,z')$ for all $z,z' \in Z$, which implies that $d$ is a metric on $Z$ since $q$ is a quasimetric on $Z$. To verify this claim we will show that if $\sigma = \{z=z_0, z_1, \ldots, z_k, z_{k+1} = z'\}$ is a sequence of points in $Z$, then
        \begin{eqnarray} \label{eqn:sigma}
            q(z,z') &\leq& \Sigma (\sigma) := K \left( q(z_0,z_1) + 2 \sum_{i=1}^{k-1} q(z_i,z_{i+1}) + q(z_k, z_{k+1}) \right).
        \end{eqnarray}
        Since $\Sigma(\sigma) \leq 2K \sum_{i=0}^k q(z_i,z_{i+1})$, it will follow that $q(z,z') \leq 2Kd(z,z')$.

        We prove Inequality~\ref{eqn:sigma} by induction. If $k=1$, then the inequality holds since $q$ is a quasimetric.

        Now let $k \in \N$, and assume Inequality~\ref{eqn:sigma} holds for all sequences with at most $k+1$ elements. Let $\sigma = \{z=z_0, z_1, \ldots, z_k, z_{k+1} = z'\}$ be a sequence with $k+2$ elements. Let $p \in \{1,\ldots, k-1\}$. Let
            \[\sigma_p' = \{z_0, z_1, \ldots, z_{p+1}\} \quad \textrm{ and } \quad
               \sigma_p'' = \{z_p, z_{p+1}, \ldots, z_{k+1} \}. \]
        Then $\Sigma(\sigma) = \Sigma(\sigma_p') + \Sigma(\sigma_p'')$ by definition. Since $q$ is a quasimetric, for each $p \in \{0, \ldots, k\}$
            \begin{eqnarray} \label{eqn:qm_max}
                q(z,z') &\leq& K \max\{ q(z,z_p), q(z_p, z') \}.
            \end{eqnarray}
        Thus, there exists a maximal $p \in \{0, \ldots, k\}$ so that
            \[ q(z,z') \leq K q(z_p, z'). \]
        It then follows from Equation~\ref{eqn:qm_max} that
            \[q(z,z') \leq Kq(z,z_{p+1}).\]
        Assume towards a contradiction that
            \begin{eqnarray} \label{qm_cont}
                q(z,z') &>& \Sigma(\sigma).
            \end{eqnarray}
        Then, by the definition of $\Sigma(\sigma)$
            \[ q(z,z') > Kq(z,z_1) \quad \textrm{ and } \quad q(z,z') > Kq(z_k, z').\]
        Thus, $p \in \{1, \ldots, k-1\}$. Therefore,
            \begin{eqnarray*}
                q(z,z_{p+1}) + q(z_p, z') &\leq & \Sigma(\sigma_p') + \Sigma(\sigma_p'') \\
                &=& \Sigma (\sigma) \\
                &<& q(z,z') \\
                &\leq& K \min\{ q(z,z_{p+1}), q(z_p, z')\} \\
                &\leq& q(z,z_{p+1}) + q(z_p, z'),
            \end{eqnarray*}
         a contraction. Above, the induction hypothesis was used in line 1, Equation~\ref{qm_cont} was used in line~3, the choice of $p$ was used in line 4, and that $K \leq 2$ was used in line 5. The conclusion of the proposition follows.
    \end{proof}

    The second operation involved in the metric-from-quasimetric construction is snowflaking, which serves to make small distances larger so that the infimum in the chain construction is not zero for distinct points. Snowflaking is an example of a quasisymmetry (see Section~\ref{sec:quasisymmetries}).

  \begin{defn}[Snowflake]
        Let $Z$ be a set and $q: Z \times Z \rightarrow [0,\infty)$. Let $\epsilon>0$. The {\it $\epsilon$-snowflake of~$q$} is the function $q^{\epsilon}: Z \times Z \rightarrow [0,\infty)$ given by $q^{\epsilon}(x,y) = q(x,y) ^{\epsilon}$.
    \end{defn}

    The next lemma follows immediately from the definition of quasimetric.

    \begin{lem} \label{lemma:snowflake_qm}
        Let $q$ be a $K$-quasimetric on a set $Z$, and let $\epsilon>0$. Then the snowflake $q^{\epsilon}$ is a $K^{\epsilon}$-quasimetric on $Z$. \qed
    \end{lem}

    The following proposition shows that an arbitrary quasimetric can be snowflaked appropriately to result in a suitable quasimetric to which the chain construction can be applied.

    \begin{prop} \label{prop:metric_from_qm}
        Let $q$ be a $K$-quasimetric on a set $Z$. Then for all $\epsilon \in (0, \frac{\ln 2}{\ln K}]$ there exists a metric $d_{\epsilon}$ on $Z$ bilipschitz equivalent to the snowflake $q^{\epsilon}$. In particular, for all $z,z' \in Z$,
            \[\frac{1}{2K^{\epsilon}}q^{\epsilon}(z,z') \leq d_{\epsilon}(z,z') \leq q^{\epsilon}(z,z').\]
    \end{prop}
    \begin{proof}
        Let $q$ be a $K$-quasimetric on a set $Z$. If $\epsilon \in (0, \frac{\ln 2}{\ln K}]$, hen the snowflake $q^{\epsilon}$ is a $K^\epsilon$-quasimetric on $Z$ with $K^{\epsilon} \leq 2$. Thus, the map obtained from $q^{\epsilon}$ by the chain construction is a metric on $Z$ as desired by Proposition~\ref{prop:chain}.
    \end{proof}

        The existence of visual metrics follows.

    \begin{cor}[Visual metric existence]
        Let $X$ be a $(\delta)$-hyperbolic metric space, and let $p \in X$. For all $a \in (1, e^{2/\delta}]$, there exists a visual metric $d_{a}$ on $\p X$ with visual metric parameter $a$. Further, for all $\eta, \eta' \in \p X$,
            \[\frac{1}{2a^{\delta}}a^{-(\eta,\eta')_p} \leq d_{\epsilon}(\eta,\eta') \leq  a^{-(\eta,\eta')_p}\]
    \end{cor}
    \begin{proof}
        The corollary follows from Lemma~\ref{lem:hypqm} and Proposition~\ref{prop:chain}.
    \end{proof}

 \subsection{The $\CAT(-1)$ setting} \label{subsec:CAT-1}

    In this subsection we sketch the proof of the following theorem of Bourdon.

    \begin{thm}[$\CAT(-1)$ visual metrics] \cite[Theorem 2.5.1]{bourdon-flot}
        \label{thm:CAT-1vm}
        Let $X$ be a $\CAT(-1)$ space and $p \in X$. Then the function $\rho:\p X \rightarrow \p X$ given by $\rho(\eta,\eta') = e^{-(\eta,\eta')_p}$ is a metric on $\p X$.
    \end{thm}

    The structure of the boundary of a $\CAT(-1)$ space is simpler than in the general $\delta$-hyperbolic setting. The first contributing factor is that geodesic segments, rays, and lines are unique in a $\CAT(0)$ space (and thus in a $\CAT(-1)$ space). This property follows easily from the $\CAT(0)$ definition.

    \begin{lem}[Unique geodesics] \cite[Remark 1.4.2]{bourdon-flot} \cite[Proposition II.8.2]{bridsonhaefliger}
        Let $X$ be $\CAT(0)$ space, let $p \in X$, and let $\eta \in \p X$. There exists a unique geodesic ray $[p,\eta)$ from $p$ to $\eta$.
    \end{lem}

    Following uniqueness of geodesic rays, one can prove that the extension of the Gromov product to the boundary of a $\CAT(-1)$ space also has a nice description. One can compare the next proposition to \Cref{lemma:exist_seq}.

    \begin{prop}[Gromov product extension] \cite[Proposition 2.4.3]{bourdon-flot} \cite[Proposition 3.4.2]{buyaloschroeder} \label{prop:catgr_ext}
        Let $X$ be a $\CAT(-1)$ space, let $p \in X$ and $\eta, \eta' \in \p X$. Let $y_n \in [p, \eta)$ and $y_n' \in [p, \eta')$ be sequences of points converging to $\eta$ and $\eta'$. Then $(y_n,y_n')_p$ converges to $(\eta,\eta')_p$.
    \end{prop}

    Bourdon's strategy to prove \Cref{thm:CAT-1vm} is to introduce an auxiliary function $\alpha_p$ on the $\CAT(-1)$ space, where $p$ is the basepoint of the Gromov product above. Defined below, $\alpha_p$ will serve two purposes:
        \begin{enumerate}
            \item First, $\alpha_p$ is defined using the sine function, which is easily seen to satisfy the triangle inequality for angles in $[0, \frac{\pi}{2}]$:
                \[\sin (\theta + \theta') \leq \sin (\theta) + \sin(\theta').\]
            The definition of $\alpha_p$ involves comparison triangles, so some care is required to prove the triangle inequality holds on spheres in a general $\CAT(-1)$ space.
            \item Second, hyperbolic trigonometry allows one to relate $\alpha_p$ to an expression involving hyperbolic sine and cosine
                \[ \sinh(x) = \frac{e^x - e^{-x}}{2} \quad \textrm{ and } \quad \cosh(x) = \frac{e^x + e^{-x}}{2},\]
            which will allow one to prove that the function $\alpha_p$ converges to the function $e^{-(\eta,\eta')_p}$ on the boundary.
        \end{enumerate}

    Let $X$ be a $\CAT(-1)$ space, and let $p \in X$.
    Define \[\alpha_p: X \backslash \{p\} \times X \backslash \{p\} \rightarrow [0,1]\]
    by
        \[\alpha_p(y,y'):= \sin \frac{\angle \bar{y}\bar{p}\bar{y}'}{2},  \]
    where $\Delta(\bar{y}\bar{x}\bar{y}')$ is a comparison triangle for $\Delta(yxy')$ in $\Hy^2$, and $\angle \bar{y}\bar{p}\bar{y}'$ is the angle based at $\bar{p}$ between the geodesics $\bar{p}\bar{y}$ and $\bar{p}\bar{y'}$. See Figure~\ref{figure-CAT-1metric}.

           \begin{figure}
                \begin{centering}
	            \begin{overpic}[width=.5\textwidth,tics=5]{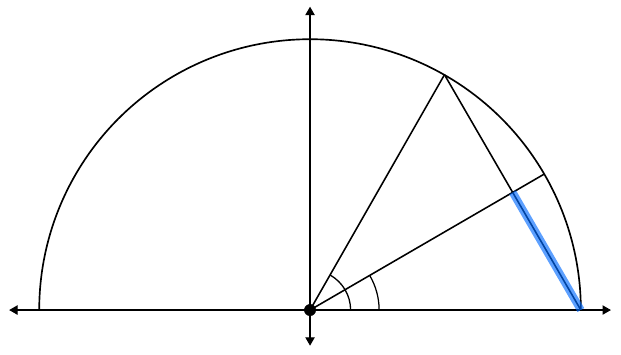} 
                    \put(46.5,3.5){\small{$p$}}
                    \put(72,47){\small{$y$}}
                    \put(92,3.5){\small{$y'$}}
                    \put(62,10){$\frac{\angle ypy'}{2} $}
                \end{overpic}
	           \caption{\small{The function $\alpha_p$ corresponds to half the chordal metric on the circle. For the points illustrated on the unit circle above, $\sin \frac{\angle ypy'}{2}$ is the length of the highlighted blue segment. }}
	           \label{figure-CAT-1metric}
                \end{centering}
        \end{figure}
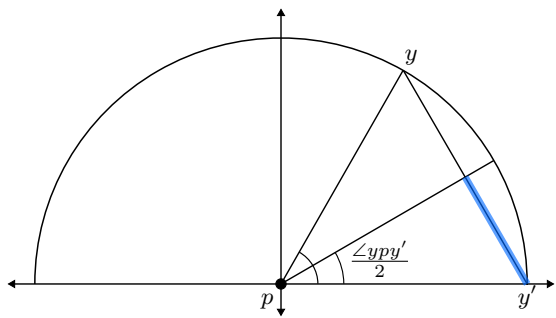

    \begin{proof}[Sketch of the proof of Theorem~\ref{thm:CAT-1vm}]
        The proof of the theorem follows from three claims. The first claim is a tool from hyperbolic trigonometry that will be needed in the subsequent two claims.  We refer to the reader to \cite{katok} for details on hyperbolic geometry.

        \begin{claim} \label{claimA}
            For all $y,y' \in X - \{p\}$,
                 \[\alpha_p(y,y') = \sqrt{ \frac{\cosh\bigl(d(y,y')\bigr)}{2\sinh\bigl(d(p,y)\bigr)\sinh\bigl(d(p,y')\bigr)} - \frac{\cosh\bigl( d(p,y) - d(p,y')\bigr)}{2\sinh\bigl(d(p,y)\bigr)\sinh\bigl(d(p,y')\bigr)} }. \]
        \end{claim}

           The next two claims illustrate the key properties of the function $\alpha_p$.

        \begin{claim} \label{claimB}
            The function $\alpha_p$ is a metric on $S(r,p)$, the sphere of radius $r>0$ about $p$ in $X$.
        \end{claim}

            To prove the claim, comparison triangles may be used to reduce it to a sphere in $\Hy^2$, where the triangle inequality is clear.

            The last claim proves that $\alpha_p$ converges to the function $e^{-(\eta,\eta')_p}$.

        \begin{claim} \label{claimC}
            Let $\eta, \eta' \in \p X$. Let $y_n \in [p, \eta)$ and $y_n' \in [p, \eta')$ be sequences of points converging to $\eta$ and $\eta'$. Then,
                \[\lim_{n \rightarrow \infty} \alpha_p(y_n, y_n') = e^{-(\eta,\eta')_p}.\]
        \end{claim}

            The claim follows by analyzing the description of $\alpha_p$ given in Claim~\ref{claimA} and applying Proposition~\ref{prop:catgr_ext}.
    The claims together complete the proof of the theorem.
    \end{proof}

    As described by Bourdon~\cite[Example 2.5.9]{bourdon-flot} and Buyalo--Schroeder~\cite[Section 2.4.3]{buyaloschroeder}, the metric above on the boundary of the hyperbolic plane corresponds to half the chordal metric on the circle. See \Cref{figure-CAT-1metric}.

 \subsection{Examples} \label{subsec:VMexamples}

    \begin{example}[Triangle inequality failure] \label{example:tri_ineq}
        Let $\Gamma = \Z/10\Z * \Z/10\Z = \la a,b \,|\, a^{10} = b^{10} \ra$, and let $X$ be the Cayley graph for $\Gamma$ with respect to the generating set above. Then the function $\rho:\p X \times \p X \rightarrow [0,\infty)$ given by $\rho(\eta,\eta') = e^{-(\eta,\eta')_p}$ does not satisfy the triangle inequality. Indeed, let $\eta, \eta', \eta''$ be three points in $\p X$ as shown in \Cref{figure-tri_inequality}.

        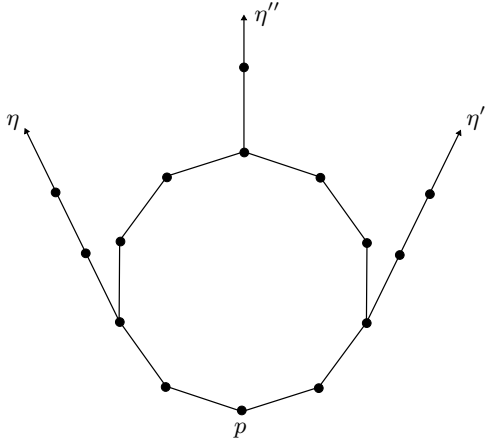
\begin{figure}
                \begin{centering}
	            \begin{overpic}[width=.4\textwidth, tics=5]{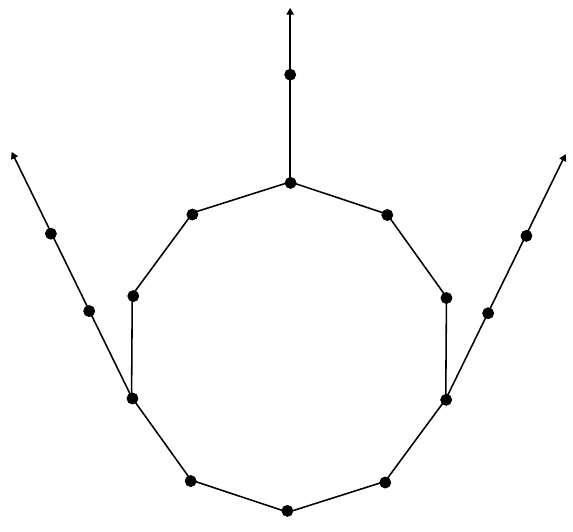} 
                    \put(48.5,-1.5){\small{$p$}}
                    \put(-2,67){\small{$\eta$}}
                    \put(100,67){\small{$\eta'$}}
                    \put(53,90){\small{$\eta''$}}
                \end{overpic}
	           \caption{\small{A portion of the Cayley graph for $\Z/10\Z * \Z/10\Z$ that illustrates the failure of the triangle inequality in a potential metric on the boundary. }}
	           \label{figure-tri_inequality}
                \end{centering}
        \end{figure}

        Then
            \[ (\eta,\eta')_p = 0, \quad  \quad (\eta,\eta'')_p = 2, \quad \quad(\eta', \eta'')_p = 2.\]
        So, $\rho(\eta,\eta') = 1$, $\rho(\eta, \eta'') = e^{-2}$, and $\rho(\eta', \eta'') = e^{-2}$. Hence, the triangle inequality fails for this triple:
            \[ \rho(\eta, \eta') > \rho(\eta, \eta'') + \rho(\eta', \eta''). \]
    \end{example}

 \begin{example}[Boundaries that are not geodesic metric spaces] \label{ex-nogeodesic}  There are $\CAT(-1)$ spaces whose boundaries with the Bourdon visual metric are not geodesic metric spaces. Such spaces can be easily constructed via polygonal complexes with large angle in the links of the vertices.

 For example, let $X$ be the polygonal complex with every cell isometric to a regular, right-angled hyperbolic hexagon and every vertex link a $10$-cycle, with total angle $5\pi$. Then, $X$ is a $\CAT(-1)$ space  with visual boundary homeomorphic to the circle. Note the automorphism group of this complex acts on $X$ geometrically and is therefore virtually Fuchsian.

 We claim that the boundary equipped with the visual metric $d(\eta,\eta') = e^{-(\eta,\eta')_p}$ is not a geodesic metric space.
 The space $X$ contains a configuration of geodesic rays $\eta_1, \eta_2, \eta_3$ in the one skeleton of~$X$ that each emanate from a vertex $p$ and pass through a vertex $q$ so that the angle at~$q$ between consecutive rays equals $\pi$. See Figure~\ref{figure-nogeodesic}. Moreover, the unique geodesic between pairs $(\eta_i,\eta_j)$ for $i \neq j$ is the concatenation of rays $[q,\eta_i) \cup [q, \eta_j)$. Thus, $(\eta_i, \eta_j)_p = d_X(p,q)$ by Proposition~\ref{prop:catgr_ext}. Hence,
    \[d(\eta_1,\eta_2) = d(\eta_2,\eta_3) = d(\eta_1,\eta_3). \]
 There are only two arcs in the boundary between $\eta_1$ and $\eta_3$, and a similar argument for the other arc proves neither is a geodesic. Thus, $(\p X, d)$ is not a geodesic metric space.

 Note that increasing the cone angle exacerbates this situation. See Figure~\ref{figure-nogeodesic}. Moreover, the boundary has self-similarity properties, and this nonexistence of geodesics exists at all scales. In particular, there is no geodesic between any pair of points in this boundary. One can visualize that the metric on the boundary looks like a snowflake.
  \end{example}

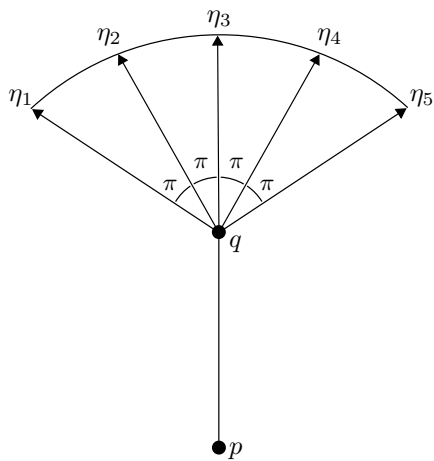
\begin{figure}
    \begin{centering}
	\begin{overpic}[width=.35\textwidth, tics=5]{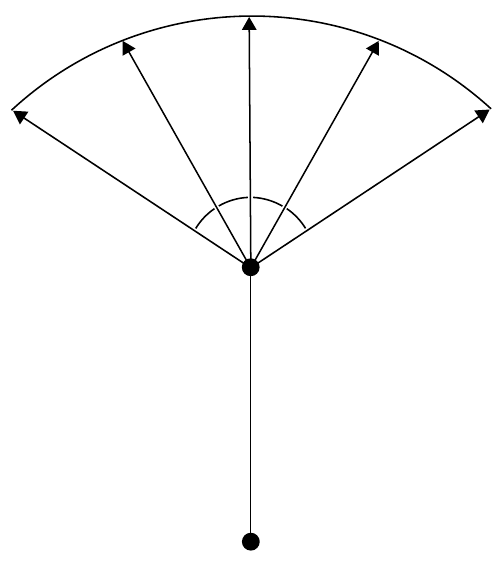} 
        \put(47,2){$p$}
        \put(47,49){$q$}
        \put(-3,82){$\eta_1$}
        \put(17,96){$\eta_2$}
        \put(42,100){$\eta_3$}
        \put(67,96){$\eta_4$}
        \put(88,82){$\eta_5$}
        \put(32,62){\small{$\pi$}}
        \put(39,67){\small{$\pi$}}
        \put(47,67){\small{$\pi$}}
        \put(54,62){\small{$\pi$}}
    \end{overpic}
	\caption{\small{An example of a circle boundary of a hyperbolic group so that the boundary is not a geodesic metric space when equipped with the visual metric. Illustrated are geodesic rays and geodesic lines in a $\CAT(-1)$ polygonal complex $X$ described in Example~\ref{ex-nogeodesic} with boundary $\p X \cong S^1$. The Gromov product between any pair of distint rays is equal: $(\eta_i,\eta_j)_p = d_X(p,q)$ for $i \neq j$. Therefore, the arc from $\eta_1$ to $\eta_5$ is not a geodesic in the Bourdon metric $d(\eta_i,\eta_j) = e^{-(\eta_i,\eta_j)_p}$.  }}
	\label{figure-nogeodesic}
    \end{centering}
    \end{figure}

    \begin{remark}[Quasi-arcs]
       While the boundary is not in general a geodesic metric space, boundaries of one-ended hyperbolic groups contain {\it quasi-arcs}, which are quasisymmetric embeddings of the interval $[0,1]$. See \Cref{sec:qs_invariants} for discussion, references, and applications.
    \end{remark}

    \section{Quasisymmetries} \label{sec:quasisymmetries}

    \begin{center}
        {\it ``The quasisymmetry condition is very subtle; it is a flexible condition and also guarantees strong properties of embeddings.'' - Heinonen~\cite[Page 79]{heinonen-Lectures} }
    \end{center}

\subsection{Boundary motivation}

     From the point of view of geometric group theory, one wishes to understand how quasi-isometries extend to the boundary of a hyperbolic space equipped with a visual metric. We start with a simple example.

    \begin{example}[Extension of isometries of $X$ to the boundary $(\p X, d_a)$]
        Let $X$ be isometric to the $4$-valent metric tree with every edge of length one. The function $d_e(\eta,\eta') := e^{-(\eta,\eta')_p}$ is a visual metric on $\p X$.
        Consider the rays $\eta, \eta' \in \p X$ as illustrated in \Cref{figure-QS_to_boundary}.

    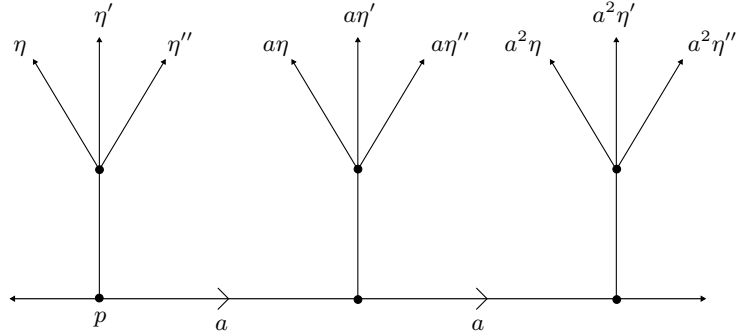
\begin{figure}
    \begin{centering}
	\begin{overpic}[width=.65\textwidth, tics=5]{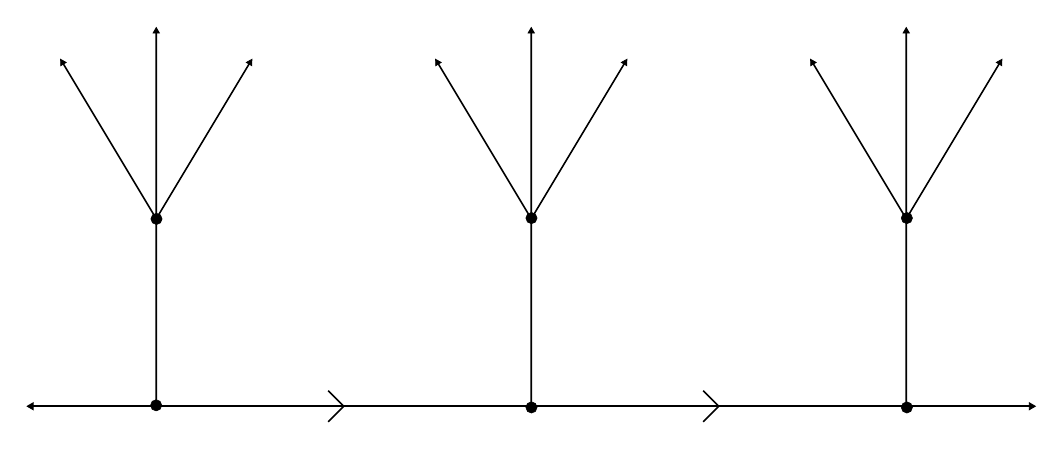} 
            \put(14,3){\small{$p$}}
            \put(30.5,2){\small{$a$}}
            \put(65.5,2){\small{$a$}}
            \put(3,40){\small{$\eta$}}
            \put(14,44){\small{$\eta'$}}
            \put(24,40){\small{$\eta''$}}
            \put(37,40){\small{$a\eta$}}
            \put(48,44){\small{$a\eta'$}}
            \put(60,40){\small{$a\eta''$}}
            \put(70,40){\small{$a^2\eta$}}
            \put(82,44){\small{$a^2\eta'$}}
            \put(95,40){\small{$a^2\eta''$}}
    \end{overpic}
	\caption{\small{ Extensions of isometries of the tree to the boundary distort the distance between pairs of these rays in the visual metric, but they preserve the ratio of distances between the rays in this triple.  }}
	\label{figure-QS_to_boundary}
    \end{centering}
    \end{figure}

        The space $X$ is the Cayley graph for the free group $F_2 = \la a,b \ra$, which acts on $X$ by isometries and extends to a map $\p X \rightarrow \p X$. Then,
            \[  d_e(\eta,\eta') = e^{-(\eta,\eta')_p} = e^{-1},\]
        while
            \[ d_e(a\eta,a\eta') = e^{-(a\eta,a\eta')_p} = e^{-2}. \]
        Notably, the extension of $a$ to the boundary is not an isometry. Indeed, the basepoint of the visual metric is fixed, while the rays translate.
        Moreover, the distortion in distance can be made arbitrarily bad:
            \[ d_e(a^k\eta,a^k\eta') = e^{-(a^k\eta,a^k\eta')_p} = e^{-(k+1)}. \]
        However, consider now a third ray $\eta'' \in \p X$ as in the figure. In this special case the relative distance - the ratio of distances - is not distorted by the extension of powers of $a$ to the boundary:
            \[ \frac{d_e(\eta,\eta'')}{d_e(\eta', \eta'')} = \frac{d_e(a^k\eta,a^k\eta'')}{d_e(a^k\eta', a^k\eta'').} \]
        This is the type of behavior one can expect in general from quasi-isometry extensions:

        \begin{center}
           {\it  Quasi-isometries distort distances in a controlled way.}

           {\it  Extensions of quasi-isometries to the boundary distort relative distances in a controlled way. }
        \end{center}
    \end{example}

        This example is particularly nice as the quasi-isometries were isometries and the metric space was a tree. Ratios of distances on the boundary were not distorted at all. In general, one must ask for a weaker condition on ratios of distances as given in the definition below.

\subsection{Quasisymmetries}

    \begin{defn}[Quasisymmetry]
        A homeomorphism $f:(Z,d) \rightarrow (Z',d')$ between metric spaces is a {\it quasisymmetry} if there exists a homeomorphism $\phi:[0,\infty) \rightarrow [0,\infty)$ so that for all $x,y,z \in Z$ with $x \neq z$,
            \[ \frac{d'\bigl(f(x),f(y)\bigr)}{d'\bigl(f(x),f(z)\bigr)} \leq \phi \left( \frac{d(x,y)}{d(x,z)} \right). \]
            The function $\phi$ is called the {\it control function} for the map $f$, and $f$ is also called a $\phi$-quasisymmetry.
    \end{defn}

    Equivalently, a homeomorphism $f:(Z,d) \rightarrow (Z',d')$ is a $\phi$-quasisymmetry if
        \[ \frac{d(x,y)}{d(x,z)} \leq t \quad\quad \textrm{ imples } \quad \quad\frac{d'\bigl(f(x),f(y)\bigr)}{d'\bigl(f(x),f(z)\bigr)} \leq \phi(t)\]
    for all $x,y,z \in Z$ with $x \neq z$.

    Quasisymmetries were introduced by Tukia--V\"ais\"al\"a~\cite{tukiavaisala} for general metric spaces, following work of Beurling--Ahlfors~\cite{beurlingahlfors}.
    Quasisymmetries can also be defined with respect to the distortion of annuli. We will use the following notation throughout.

    \begin{notation}
        Let $Z$ be a metric space. Let $B_r(z)$ denote the ball of radius $r$ about $z \in Z$. For $K \geq 0$, let $KB_r(z):= B_{Kr}(z)$.
    \end{notation}

    For example, let $r>0$ and $t>1$ and consider the annulus $A  = B_{rt}(x) \setminus B_r(x)$. Suppose $f:Z \rightarrow Z'$ is a $\phi$-quasisymmetry. As all points $z,y$ in the annulus $A$ satisfy $\displaystyle \frac{d(x,y)}{d(x,z)} \leq t$, the image of the annulus $A$ is contained in the annulus $A' = B_{s\phi(t)}\bigl(f(x) \bigr) \setminus B_s\bigl(f(x)\bigr)$. Note that the quasisymmetry condition gives no control on the radius $s$, only on the ratio of the inner and outer radii. The next lemma captures this discussion.

    \begin{figure}
    \begin{centering}
	\begin{overpic}[width=.7\textwidth, tics=5]{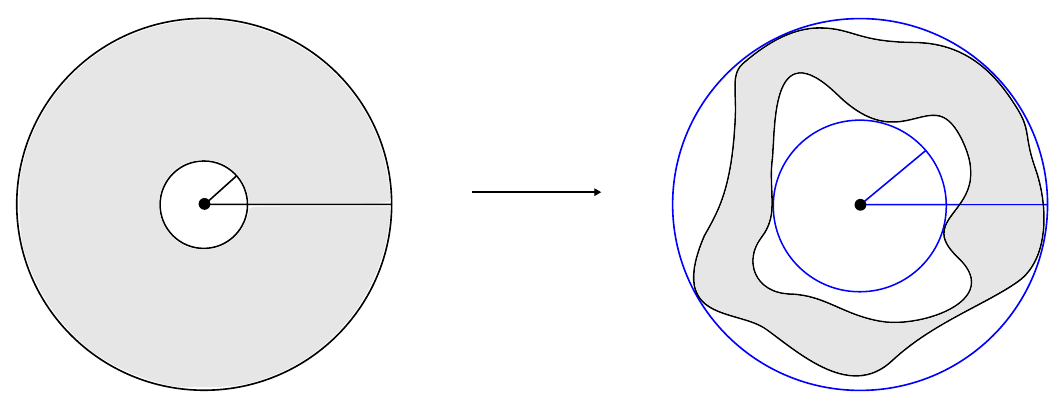} 
        \put(18,17){\small{$x$}}
        \put(19.5,21.5){\Small{$r$}}
        \put(27,17){\Small{$rt$}}
        \put(50,22){\small{$f$}}
        \put(77.5,16){\small{$f(x)$}}
        \put(83,23){\textcolor{blue}{\Small{$s$}}}
        \put(91,17){\textcolor{blue}{\Small{$s\phi(t)$}}}
    \end{overpic}
	\caption{\small{ Distortion of annuli under a quasisymmetry.  }}
	\label{figure-QS_to_boundary}
    \end{centering}
    \end{figure}
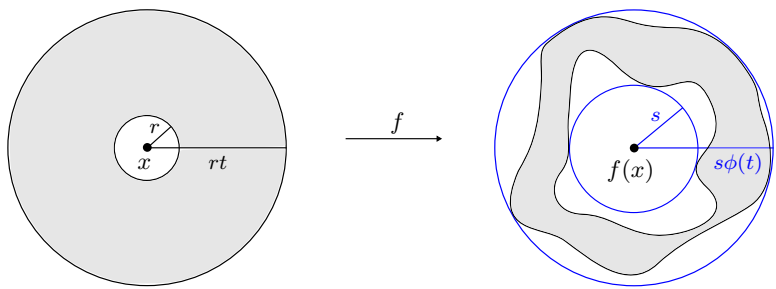

    \begin{lem} \cite[Lemma 1.2.18]{mackaytyson} \label{lemma:annuli}
        Let $Z$ and $Z'$ be metric spaces. A homeomorphism $f:Z \rightarrow Z'$ is $\phi$-quasisymmetric if and only if whenever
            \[ B_1:= B_{r_1}(z) \subset B_{r_2}(z) =: B_2\]
        are concentric balls in $Z$, there exist concentric balls
            \[ B_1':= B_{r_1'}(f(z)) \subset B_{r_2'}(f(z)) =: B_2' \]
        so that
            \[B_1' \subset f(B_1) \subset f(B_2) \subset B_2' \quad \quad \textrm{ and } \quad \quad \frac{r_2'}{r_1'} \leq \phi \left( \frac{r_2}{r_1} \right).  \]
    \end{lem}

    Simple computations show that many familiar, well-behaved maps are quasisymmetries.

    \begin{example}[Bilipschitz maps are quasisymmetries]
        A map $f:(Z,d) \rightarrow (Z', d')$ between metric spaces is {\it $K$-bilipschitz} if for all $x,y \in Z$,
            \[ \frac{1}{K}d(x,y) \leq d\bigl( f(x), f(y)\bigr) \leq Kd(x,y). \]
        A bilipschitz map distorts distances by at most a uniform multiplicative factor. If $f$ is $K$-bilipschitz, then $f$ is a quasisymmetry with control function $\phi(t) = K^2t$. Indeed,
        \[ \frac{d\bigl( f(x),f(y) \bigr)}{d\bigl( f(x),f(z) \bigr)}  \leq \frac{Kd(x,y)}{\frac{1}{K}d(x,z)} \leq K^2\cdot\frac{d(x,y)}{d(x,z)}.\]
    \end{example}

    \vskip.1in

        \begin{center}
            {\it ``Informally, one can define an embedding to be bilipschitz if it distorts both the shape and size of an object by a bounded amount, while quasisymmetry only preserves the approximate shape."

            - Heinonen~\cite[Page 79]{heinonen-Lectures}}
        \end{center}

    \begin{example}[Snowflake maps are quasisymmetries]
        Recall the definition of snowflaking, also used in the construction of visual metrics. Let $(Z,d)$ be a metric space. Let $\epsilon \in (0,1)$. The {\it snowflake of $(Z,d)$ with parameter~$\epsilon$} is the metric space $(Z,d^{\epsilon})$ where $d^{\epsilon}:Z\times Z \rightarrow [0,\infty)$ is given by
        \[d^{\epsilon}(z_1,z_2) = d(z_1,z_2)^{\epsilon}.\]

        If $(Z,d)$ is a metric space and $(Z,d^\epsilon)$ is its $\epsilon$-snowflake for some $\epsilon \in (0,1)$, then the identity map is a quasisymmetry with control function $\phi(t) = t^\epsilon$. Indeed, suppose $x,y,z \in Z$ with $x \neq z$.
            \[\textrm{If } \quad \frac{d(x,y)}{d(x,z)} \leq t \quad \textrm{then, } \quad \frac{d^{\epsilon}(x,y)}{d^{\epsilon}(x,z)}= \frac{d(x,y)^\epsilon}{d(x,z)^\epsilon} = \left(\frac{d(x,y)}{d(x,z)} \right)^{\epsilon} \leq t^{\epsilon}.\]
        \end{example}

    The above examples imply the following lemma. The isometry type of the visual metric on the boundary of a hyperbolic metric space depends on the choices in its construction: the basepoint $p$ in the Gromov product and on the choice of visual metric parameter. The quasisymmetry type is nonetheless well-defined.

    \begin{lem} \cite[Chapter 2.2]{buyaloschroeder}
       Let $X$ be a $\delta$-hyperbolic geodesic metric space.
       \begin{enumerate}
           \item Let $d_a$ and $d_{a'}$ be visual metrics on $\p X$ defined with respect to basepoints $p$ and $p'$, respectively. Then $(\p X,d_a)$ and $(\p X, d_{a'})$ are bilipschitz equivalent.
           \item Visual metrics $d_a$ and $d_{a'}$ on $\p X$ defined with respect to a basepoint $p$ and visual metric parameters $a$ and $a'$ are H\"{o}lder equivalent. That is, there exists a constant $K$ so that for all $\eta, \eta' \in \p X$,
                \[ \frac{1}{K}d_a^{\alpha}(\eta,\eta') \leq d_{a'}(\eta,\eta') \leq Kd_a^{\alpha}(\eta,\eta')\]
            where $\alpha = \frac{\log a'}{\log a}$.
       \end{enumerate}
    \end{lem}
    \begin{proof}
        Conclusion~(1) follows from the fact that for all $\eta, \eta' \in \p X$,
            \[|(\eta,\eta')_p - (\eta, \eta')_{p'}| \leq d_X(p,p').\]
        Indeed, it follows from the definition of the Gromov product together with the triangle inequality that for all $x_i,x_j,p,p' \in X$,
        $|(x_i,x_j)_p - (x_i,x_j)_{p'}| \leq d(p,p')$. The inequality above then follows from \Cref{lemma:exist_seq}.

        Conclusion~(2) follows from definitions and the relationship $a' = a^{\alpha}$.
    \end{proof}

    \begin{lem}\cite[Theorem 2.2]{tukiavaisala} \cite[Lemma 5.2.13]{buyaloschroeder}
        Suppose $f:X \rightarrow Y$ is a $\phi$-quasisymmetry and $g:Y \rightarrow Z$ is a $\psi$-quasisymmetry. Then $f^{-1}: Y \rightarrow X$ is a $\phi'$-quasisymmetry with $\phi'(t) = \frac{1}{\phi^{-1}(\frac{1}{t})}$. The composition $g \circ f: X \rightarrow Z$ is a $(\psi \circ \phi)$-quasisymmetry. In particular, quasisymmetry is an equivalence relation on metric spaces. \qed
    \end{lem}

    We note one immediate feature of quasisymmetries that distinguish them from the maps in the next subsection. Quasisymmetries take bounded spaces to bounded spaces.

    \begin{prop} \cite[Theorem 2.5]{tukiavaisala} \cite[Proposition 10.8]{heinonen-Lectures} \label{prop:QSBounded}
        If $f:X \rightarrow Y$ is a $\phi$-quasisymmetry and $B \subset X$ is a subset so that $0 <  \diam B < \infty$, then $\diam f(B) <\infty$.
    \end{prop}
    \begin{proof}
        Suppose that $B \subset X$ with $0 <  \diam B < \infty$. Let $b_1, b_1' \in B$ so that
            \[ \frac{1}{2}\diam B \leq d(b_1,b_1').   \]
        Let $b \in B$. It follows from the choice of the points $b_1, b_1'$ that
            \[d(b_1,b) \leq \diam B \leq 2 \cdot d(b_1,b_1').\]
        This yields the ratio
            \[\frac{d(b_1,b)}{d(b_1,b_1')} \leq 2.\]
        Since $f$ is a $\phi$-quasisymmetry,
            \[\frac{d\bigl(f(b_1), f(b) \bigr)}{d \bigl( f(b_1), f(b_1')\bigr)} \leq \phi(2). \]
        Thus,
            \[ d\bigl(f(b), f(b_1)\bigr) \leq \phi(2) \cdot d \bigl( f(b_1),f(b_2) \bigr).\]
        Therefore, $f(B)$ is bounded, as desired.
    \end{proof}

    There is a stronger version of the proposition above, appearing in \cite[Proposition 10.8]{heinonen-Lectures} and \cite[Theorem 2.5]{tukiavaisala}, which is used to prove that the doubling property is a quasisymmetry invariant; see \Cref{sec:qs_invariants}.

    \subsection{Relation to quasi-conformal and quasi-M\"{o}bius maps}

        There are two other well-studied notions related to quasisymmetries:

        \begin{enumerate}
            \item Quasi-conformal maps, defined by a local, infinitesimal condition on relative distances:

            \begin{defn}
                Let $f:Z \rightarrow Z'$ be a homeomorphism between metric spaces. Let $p \in Z$. The {\it dilatation} of $f$ at $p$ is
                \[ H(f,p):= \limsup_{r \rightarrow 0} \frac{\sup \bigl\{ d\bigl(f(x),f(p)\bigr) \,|\, x \in B_r(p) \bigr\}}{\inf \bigl\{ d\bigl(f(x),f(p)\bigr) \,|\, x \notin B_r(p) \bigr\}}. \]
                The homeomorphism $f$ is {\it $C$-quasi-conformal} if $H(f,p) \leq C$ for all $p \in X$. The map is {\it quasi-conformal} if it is $C$-quasi-conformal for some $C$.
            \end{defn}

            \item Quasi-M\"{o}bius maps, defined with respect to the cross-ratio:

            \begin{defn} \cite{vaisala-QM}
                Let $(Z,d)$ be a metric space. The {\it cross-ratio} of four pairwise distinct points $z_1, z_2, z_3, z_4 \in Z$ is
                    \[[z_1,z_2,z_3,z_4] := \frac{d(z_1,z_3)d(z_2,z_4)}{d(z_1,z_4)d(z_2,z_3)}.\]

                Let $f:Z \rightarrow Z'$ be a homeomorphism between metric spaces. The map $f$ is {\it quasi-M\"{o}bius} if there exists a homeomorphism $\phi:[0,\infty) \rightarrow [0,\infty)$ so that for every set of four pairwise distinct points $z_1, z_2, z_3, z_4 \in Z$,
                    \[ [f(z_1), f(z_2), f(z_3), f(z_4)] \leq \phi \bigl( [z_1,z_2,z_3,z_4] \bigr)]. \]
            \end{defn}
        \end{enumerate}

        For general metric spaces, the following implications hold~\cite{vaisala-QM}:
            \[ \textrm{Quasisymmetric} \Rightarrow \textrm{Quasi-M\"{o}bius} \Rightarrow \textrm{Quasi-conformal},\]
        and these notions need not agree in general. For example, quasisymmetries take bounded spaces to bounded spaces by \Cref{prop:QSBounded}, while maps in the other families need not. The familiar map $f:\C \rightarrow \C$ given by $f(z) = \frac{z-i}{z+i}$, which takes the upper halfplane to the unit disk, is a M\"{o}bius transformation and is thus both quasi-M\"{o}bius and quasi-conformal, but is not quasisymmetric.
        If $f:Z \rightarrow Z'$ is a quasi-M\"{o}bius homeomorphism between metric spaces and both $Z$ and $Z'$ are unbounded, then $f$ is quasisymmetric. Moreover, these two notions agree for boundaries of hyperbolic groups~\cite[Proposition 2.6]{haissinsky-QM}.

       The work of Heinonen--Koskela~\cite{heinonenkoskela} presents additional motivations for these notions and conditions under which they agree. These ideas were used in a fundamental way in the quasi-isometric rigidity work of Bourdon--Pajot~\cite{bourdonpajot} for right-angled Fuchsian buildings, and in subsequent work of Xie~\cite{xie06} for Fuchsian buildings.

        These distinct families of maps yield different perspectives on the analytic structure of the boundary of a hyperbolic group. Quasi-conformal maps are often studied and motivated through the lens of classical hyperbolic geometry, complex analysis, and conformal mappings. We refer the reader to \cite[Chapter 22]{drutukapovich} for a detailed discussion on quasi-conformal maps in geometric group theory and to \cite{haissinsky-QM} for quasi-M\"{o}bius maps. In particular, quasi-M\"{o}bius maps are useful to re-construct quasi-isometries of spaces via boundary maps.

    \subsection{Quasi-isometries and quasisymmetries} \label{sec:QIandQS}

    In this section we discuss the following foundational theorem. See the text of Buyalo--Schroeder~\cite[Theorem 5.2.17; Chapter 7]{buyaloschroeder} for variants for spaces without a geometric group action.

    \begin{thm} \cite{paulin} \label{thm:QI_QS}
        Let $X$ and $X'$ be proper geodesic hyperbolic metric spaces admitting geometric actions by hyperbolic groups. Let $d_a$ and $d_a'$ be visual metrics on $\p X$ and $\p X'$, respectively. Then the spaces $X$ and $X'$ are quasi-isometric if and only if the boundaries $(\p X, d_a)$ and $(\p X', d_a')$ are quasisymmetric.
    \end{thm}

    The proof utilizes the quasi-M\"{o}bius notion on the boundary $\p X$ and its relationship to distances in the space $X$. Indeed, this condition captures, up to a few $\delta$, the distance between geodesics connecting points on the boundary.
    These ideas can be illustrated in the simple case where $X$ is a tree and the visual metric is given by $d(\eta, \eta') = e^{-(\eta,\eta')_p}$ for some $p \in X$. Suppose $x,y,z,w\in \p X$ are pairwise disjoint. Suppose $\{x_i\}, \{y_i\}, \{w_i\}, \{z_i\}$ are points in $X$ so that
    \[ x_i \rightarrow x, \quad y_i \rightarrow y, \quad z_i \rightarrow z, \quad w_i \rightarrow w.\]
        Then,
        \begin{eqnarray*}
            \log[x,y,z,w] &=& \log\bigl(d(x,z)\bigr) + \log\bigl(d(y,w)\bigr)
            - \log\bigl(d(x,w)\bigr)-\log\bigl(d(y,z)\bigr) \\
            &=& -(x,z)_p - (y,w)_p + (x,w)_p + (y,z)_p \\
            &=& \frac{1}{2}\Bigl(d(x_i,z_i) + d(y_i,w_i) - d(w_i,x_i) - d(y_i,z_i) \Bigr)
              \end{eqnarray*}
        for large enough $i$. This last quantity corresponds to a distance in the tree between geodesics between points in $\{x,y,z,w\}$ (up to a sign); see \Cref{figure-quasiMobius}.

       \begin{figure}
                \begin{centering}
	            \begin{overpic}[width=.9\textwidth, tics=5]{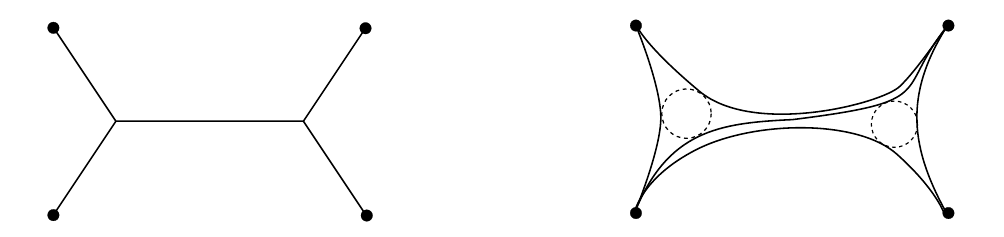} 
                    \put(2,21){$x$}
                    \put(1.5,1){$w$}
                    \put(38,21){$z$}
                    \put(38,1){$y$}
                    \put(60,21){$x$}
                    \put(60,1){$w$}
                    \put(96,21){$z$}
                    \put(96,1){$y$}
                \end{overpic}
	           \caption{\small{The logarithm of the cross ratio, $\log[x,y,z,w]$, corresponds in a tree to the distance in the tree between the geodesics with endpoints $x$, $w$ and $y$, $z$ on the boundary. In a $\delta$-hyperbolic space, this captures the same distance, up to a few $\delta$. }}
	           \label{figure-quasiMobius}
                \end{centering}
            \end{figure}
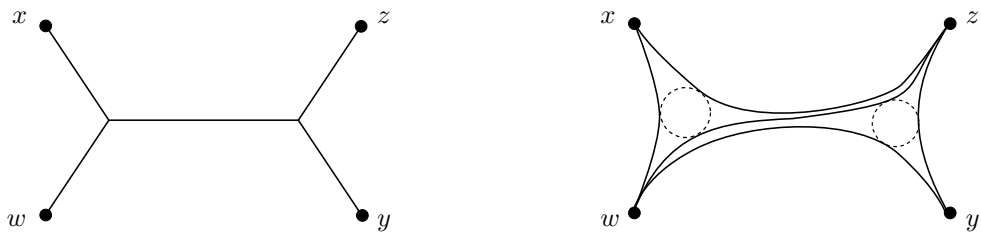

        Suppose $f:X \rightarrow X'$ is a quasi-isometry.
        Suppose the points are arranged so that $\log[x,y,z,w]$ equals the distance, call it $L$, between the geodesic lines $xw$ and $yz$.
        Using the Morse Lemma, one must show the distance between the geodesic lines $f(x)f(w)$ and $f(y)f(z)$ is bounded in terms of $L$ and the hyperbolicity constants. Exponentiating shows the boundary extension of $f$ satisfies the desired quasi-M\"{o}bius condition.

        In a general $\delta$-hyperbolic metric space, the logarithm of the cross ratio also measures, up to a constant multiple of $\delta$, the distance between geodesics~\cite[Section 4]{paulin}. Alternatively, this equals (up to some $\delta$) the distance between the centers of the ideal triangles with vertices in the set $\{x,y,z,w\}$. This point of view is also utilized in proving that every quasi-M\"{o}bius map between boundaries is induced by a quasi-isometry.

        Given a map $f: \p X \rightarrow \p X'$, one defines the desired map $X \rightarrow X'$ by realizing each point in $X$ as the center of an ideal triangle with vertices $z,z',z'' \in \p X$ and mapping it to the ideal triangle with vertices $f(z), f(z'), f(z'') \in \p X'$. One must then use the quasi-M\"{o}bius condition to verify that the map does not depend on the choice of ideal triangle, up to finite distance, and that the induced map is a quasi-isometry.

        We refer the reader to \cite[Section 1.6]{bourdon-flot} for an alternative proof that quasi-isometries induce quasisymmetries of the boundary.

    \subsection{Quasisymmetry invariants} \label{sec:qs_invariants}

        The quasisymmetry homeomorphism type of the boundary of a hyperbolic group is a quasi-isometry invariant. Hence, any quasisymmetry invariant yields a quasi-isometry invariant. This section collects well-known quasisymmetry invariants (while postponing discussion of the quasisymmetry invariant conformal dimension to the next section). Interestingly, all of the invariants in this section hold (or naturally do not hold) for the boundary of a hyperbolic group. While they are thus not interesting from the point of view of quasi-isometric classification, they lead to some powerful results about the structure of {\it all} hyperbolic groups.

        The following properties are quasisymmetry invariants:

        \begin{enumerate}
            \item Doubling property
            \item Linear connectedness
            \item Uniformly perfect property
            \item Uniformly disconnected property
        \end{enumerate}

\vskip.1in

    \begin{defn}[Doubling]
        A metric space $(Z,d)$ is {\it doubling} if there exists a constant $C>0$  so that every ball $B_r(z)$ in $Z$ can be covered by at most $C$ balls of radius $\frac{r}{2}$.
    \end{defn}

    The doubling property is a quasisymmetry invariant~\cite[Theorem 10.18]{heinonen-Lectures}.

    \begin{example}[Trees and their boundaries]
        As a first example, the $4$-regular tree $T_4$ is not a doubling metric space, while its boundary is. See \Cref{figure-doubling}.

        Indeed, fix a basepoint $p \in T_4$, and consider $B_r(p)$. In order to cover the leaves of this ball with balls of radius $\frac{r}{2}$, one needs at least as many balls as the number of points in the sphere of radius $\frac{r}{2}$ around $p$. The number of points on this sphere grows exponentially, so the space $T_4$ is not doubling.

        Fix a visual metric on $\p T_4$. For convenience, suppose the metric is given by $d_2(\eta,\eta'):= 2^{-(\eta, \eta')_p}$. Then, $(\p T_4, d_2)$ is doubling with doubling constant $C=3$. To see this, fix a ball $B$ of radius $r = 2^{-k}$ in $\p T_4$. This ball corresponds to the set of all geodesic rays in $T_4$ based at $p$ that pass through a particular vertex $v \in T_4$ at distance $k$ from $p$. There are three vertices $v_1, v_2, v_3 \in T_4$ at distance one from $v$ and at distance $k+1$ from $p$. These vertices correspond to three balls of radius $\frac{r}{2} = 2^{-(k+1)}$ that cover $B$.

        \begin{figure}[t]
                \begin{centering}
	            \begin{overpic}[width=.85\textwidth, tics=5]{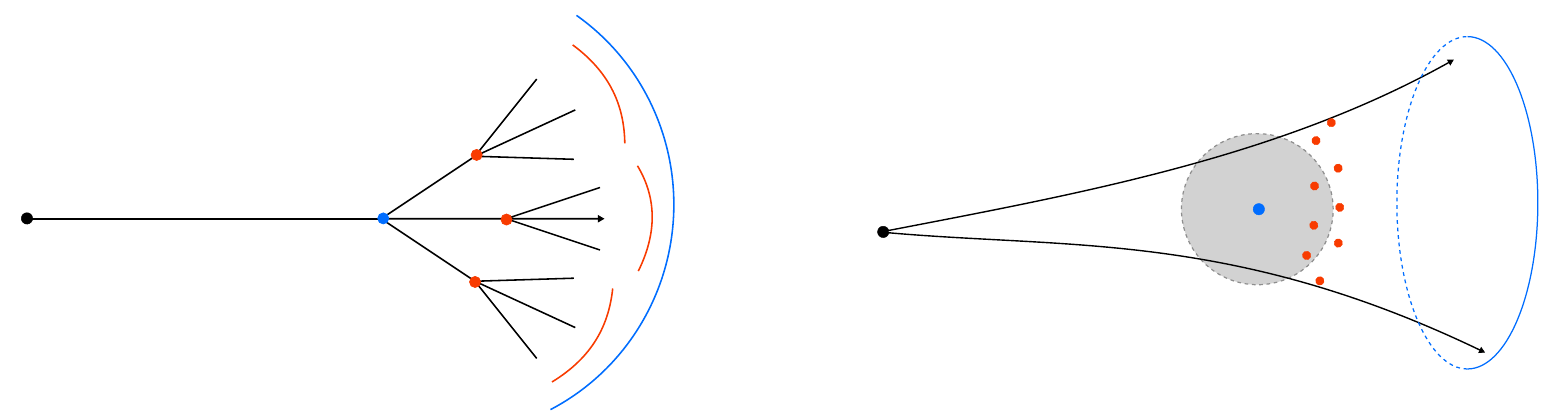} 
                    \put(-1,12.5){\small{$p$}}
                    \put(12,14){\small{$k$}}
                    \put(23,10.5){\small{$v$}}
                    \put(39,12.2){\small{$z$}}
                    \put(40,2){\small{$B_{2^{-k}}(z)$}}
                \end{overpic}
	           \caption{\small{The boundary of a hyperbolic group is a doubling metric space. On the left, the ball of radius $2^{-k}$ about $z \in \p T$ corresponds to all rays in the tree passing through the vertex $v$. This ball is covered by the three balls of radius $\frac{1}{2}2^{-k}$ corresponding to the three red vertices at distance $1$ from $v$. A similar picture occurs in any boundary, where the shadow of a ball in the space can be covered by shadows of points at a fixed farther distance from the basepoint.}}
	           \label{figure-doubling}
                \end{centering}
            \end{figure}
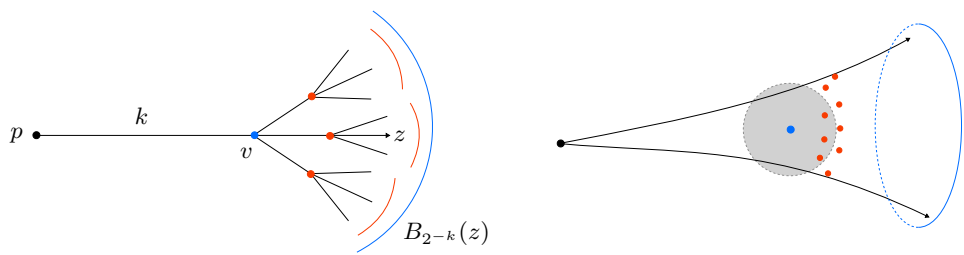
    \end{example}

    \begin{example}
        The hyperbolic plane is not a doubling metric space. Indeed, a 4-valent tree quasi-isometrically embeds in $\Hy^2$. On the other hand, Euclidean space $\E^n$ is doubling for all $n$.
    \end{example}

    The ideas in the argument above that $\p T_4$ is doubling generalize to boundaries of hyperbolic groups (and more generally to the boundary of a hyperbolic metric space $X$ that has {\it bounded growth at some scale}, meaning that there are constants $r$ and $R$ with $R > r > 0$, and $N \in \N$ such that every open ball of radius $R$ in $X$ can be covered by $N$ open balls of radius $r$).

    \begin{prop} \cite[Proposition 9.2]{bonkschramm}
        Let $G$ be a hyperbolic group, and let $d_a$ be a visual metric on $\p G$. Then $(\p G, d_a)$ is a doubling metric space.
    \end{prop}

    To prove the proposition, one can compare a ball in $\p X$ to the {\it shadow} of a ball in $X$.

    \begin{defn} \label{defn:shadow}
        Let $X$ be a $\delta$-hyperbolic proper geodesic metric space. Let $p \in X$. Let $x \in X$. The {\it shadow} of the ball $B_r(x)$ in $\p X$ consists of the set of equivalence classes that contain a representative geodesic ray based at $p$ and passing through $B_r(x)$.
    \end{defn}

    A ball of radius $r$ in the boundary lies between the shadows of two points, whose distance from the basepoint can be computed in terms of the visual metric parameter; see \cite[Lemme 1.6.2]{bourdon-flot}.
    In particular, suppose $X$ is a graph. If the visual metric parameter of $\p X$ equals $a$, and $d(p,x) = k$, then the shadow of $B_r(x)$ in $\p X$ is a set of diameter roughly $a^{-k}$.  This set can be covered by shadows of balls of radius $\frac{a^{-k}}{2}$ corresponding to ``children'' of $x$, ie shadows based at vertices $\{v_i\}$ at distance $k+d$ from $p$, so that $x$ lies near a geodesic $[p,v_i]$, and $d$ depends on $a$. See \Cref{figure-doubling}. Since the Cayley graph of a hyperbolic group is uniformly locally finite, one concludes the boundary is doubling.

    The doubling property has powerful embedding consequences that has lead to beautiful theorems about the structure of all hyperbolic groups. The first theorem below proves that all hyperbolic groups quasi-isometrically embed into $\Hy^n$ for sufficiently large $n$. The other proves $\Hy^2$ quasi-isometrically embeds into every one-ended hyperbolic group.

    \begin{thm} \cite[Theorem 1.1]{bonkschramm} \label{thm:hyp_embedding}
        If $G$ is a hyperbolic group, then there exists $n\geq 1$ so that the group $G$ quasi-isometrically embeds into a convex subset of $\Hy^n$.
    \end{thm}

    A key ingredient in the proof of the theorem above is Assouad's Embedding Theorem~\cite{assouad}, which states that if $(Z,d)$ is a doubling metric space and $\epsilon \in (0,1)$, then there exists $n<\infty$ so that the snowflaked space $(Z,d^{\epsilon})$ admits a biLipschitz embedding into $\R^n$. Thus, if $G$ is a hyperbolic group and $d_a$ is a visual metric on $\p G$, then $(\p G, d_{a'})$ admits a biLipschitz embedding into $\R^n$ for some $n$. (Here, $\R^n$ is viewed as part of the boundary sphere of $\Hy^{n-1}$.) We also note that Bonk--Schramm prove there is a {\it rough similarity} from $G$ into $\Hy^n$, which is a stronger notion than quasi-isometry.

    Gromov's famous {\it Surface Subgroup Conjecture} asks whether the fundamental group of a closed hyperbolic surface is a subgroup of every one-ended hyperbolic group. If the subgroup is quasi-convex, then a hyperbolic plane quasi-isometrically embeds in the group. The next theorem solved the geometric version of the Surface Subgroup Conjecture in a strong way.

    \begin{thm} \cite[Theorem 1]{bonkkleiner-planes} \label{thm:hyp_planes}
        If $G$ is a one-ended hyperbolic group, then the hyperbolic plane quasi-isometrically embeds in $G$.
    \end{thm}

     Again, the doubling property and Assouad's Embedding Theorem enter the proof of this theorem in an important way. One other analytic ingredient is needed:

    \begin{defn}[Linearly connected]
        A metric space $(Z,d)$ is {\it linearly connected} if there exists a constant $L>0$ so that for all $x,y \in Z$ there exists a connected subset $S \subset Z$ of diameter at most $Ld(x,y)$ containing $\{x,y\}$.
    \end{defn}

    The linearly connected property is also known as {\it bounded turning} and is a quasisymmetry invariant~\cite[Theorem 2.11]{tukiavaisala}. Bonk--Kleiner~\cite[Proposition 4]{bonkkleiner-planes} showed that the boundary of a one-ended hyperbolic group is linearly connected.

    \begin{defn}
        A {\it quasi-arc} in a metric space is the quasisymmetric image of $[0,1]$.
    \end{defn}

    \begin{thm} \cite{tukia-turning, mackay-arcs}
        If $(Z,d)$ is a complete, doubling, and linearly connected metric space, then any two points in $Z$ are the endpoints of a quasi-arc.
    \end{thm}

    Bonk--Kleiner then use the quasi-arc in the boundary to find the hyperbolic plane in the space. We note that the existence of quasi-arcs in the boundary is also utilized by Mackay~\cite{mackay10} to give lower bounds on the conformal dimension of one-ended hyperbolic groups that do not split over a virtually cyclic subgroup; see \Cref{sec:round_trees}.

    The doubling and linearly connected properties are also important in providing a quasisymmetric classification of the unit circle.

    \begin{thm} \cite[Theorem 4.9]{tukiavaisala}
        A metric circle is quasisymmetric to the Euclidean circle $S^1$ if and only if it is doubling and linearly connected.
    \end{thm}

    An analogous statement was achieved for 2-spheres by Bonk--Kleiner, and requires an additional assumption; see \Cref{def:ahlfors}.

    \begin{thm} \cite[Theorem 1.2]{bonkkleiner02}
       Let $Z$ be an Ahlfors 2-regular metric space homeomorphic to $S^2$. Then $Z$ is quasisymmetric to $S^2$ if and only if $Z$ is linearly connected.
    \end{thm}

    This work was subsequently used by Bonk--Kleiner to prove a version of Cannon's Conjecture~\cite{bonkkleiner05S2}. They showed if $G$ is a hyperbolic group with 2-sphere boundary and the Ahlfors regular conformal dimension of $G$ is attained, then $G$ is virtually Kleinian.

    \begin{remark}[Doubling, embeddings, and relatively hyperbolic groups]
        The Bowditch boundary of a relatively hyperbolic group pair $(G, \cP)$ equipped with a visual mertric is doubling if and only if every peripheral subgroup is virtually nilpotent, as shown by Mackay--Sisto~\cite[Proposition 4.5]{mackaysisto-planes} using work of Dahmani--Yaman~\cite{dahmaniyaman}.

        The embedding theorems, \Cref{thm:hyp_planes} and \Cref{thm:hyp_embedding}, have been extended to relatively hyperbolic group pairs where each peripheral subgroup is virtually nilpotent by work of Mackay--Sisto. In this setting Mackay--Sisto~\cite{mackaysisto-planes} extended the proof strategy of Bonk--Kleiner (\Cref{thm:hyp_planes}) to show that a hyperbolic plane quasi-isometrically embeds in the Cayley graph for the group. This question is open for general relatively hyperbolic groups. In addition, Mackay--Sisto~\cite{mackaysisto-maps_relhyp} proved that a group is hyperbolic relative to virtually nilpotent subgroups if and only if it embeds into truncated real hyperbolic space with at most polynomial distortion. Again, the proof utilizes the doubling property and extends the work of Bonk--Schramm (\Cref{thm:hyp_embedding}) to the relatively hyperbolic setting.
    \end{remark}

    We now mention two other quasisymmetry invariants that are important in the quasisymmetric classification of the standard Cantor set.

    \begin{defn}[Uniformly perfect]
                A metric space $(Z,d)$ is {\it uniformly perfect} if there is a constant $c >0 $ so that the set $B_r(z) \setminus B_{cr}(z) \neq \emptyset$
                for each $z \in Z$ and each $0<r<\diam Z$.
    \end{defn}

        Every connected space is uniformly perfect. As explained by Heinonen, ``the condition of uniform perfectness forbids isolated islands in a uniform manner'' \cite[Page 88]{heinonen-Lectures}. The annulus definition of quasisymmetry can be used to show this property is a quasisymmetry invariant.  Ahlflors regular spaces are uniformly perfect, hence boundaries of non-elementary hyperbolic groups are uniformly perfect by work of Coornaert~\cite{coornaert93}.

    \begin{defn}[Uniformly disconnected]
        Let $(Z,d)$ be a metric space. A sequence of points $z_0, z_1, \ldots, z_n \in Z$ is a {\it $\delta$-chain} if $d(z_{i-1}, z_i) \leq \delta d(z_0,z_n)$ for all $i \in \{1, \ldots, n\}$. The points need not be distinct, and the chain is called {\it trivial} if all of the points coincide. The metric space $(Z,d)$ is {\it uniformly disconnected} if there exists $\delta>0$ so that there are no nontrivial $\delta$-chains in $Z$.
    \end{defn}

        A path connected space is not uniformly disconnected, as any pair of points can be joined by a nontrivial $\delta$-chain for any $\delta>0$. As explained by Mackay--Tyson, ``uniform disconnectedness is a quantitative, scale-invariant version of total disconnectedness: it asserts the existence of separating annuli of a definite modulus at all scales and locations'' \cite[Page 9]{mackaytyson}. The two-thirds Cantor set is uniformly disconnected.

    There is a well-known topological characterization of the Cantor set: any nonempty, compact, totally disconnected metric space without isolated points is homeomorphic to the standard two-thirds Cantor set.     The above analytic properties can be used to provide a quasisymmetric classification of the standard two-thirds Cantor set.

     \begin{thm}[Quasisymmetric classification of the Cantor set] \cite{davidsemmes}
         A metric space is quasisymmetric to the two-thirds Cantor set if and only if it is compact, doubling, uniformly perfect, and uniformly disconnected.
     \end{thm}

    \section{Hausdorff and conformal dimension} \label{sec:dimensions}

Introduced by Felix Hausdorff in 1918, Hausdorff dimension is a notion of fractal dimension that captures the number of sets of a given diameter required to cover a metric space.

    \subsection{Hausdorff dimension}

Dimension can be viewed as an exponent: the dimension of $\R^d$ equals~$d$. Hausdorff dimension is a {\it critical exponent} - it is the infimal exponent that makes a natural series converge. For additional background, we recommend the book by Falconer~\cite{falconer}.

    \begin{center}
       {\it ``Very roughly, a dimension provides a description of how much space a set fills.''

       - Falconer~\cite[Page xx]{falconer} }
    \end{center}

We begin with a simple, motivating example.

    \begin{example} \label{ex:01dim} (The metric space $[0,1]^d$).
        Let $Z = [0,1]^d$. For $\epsilon>0$ let $N(\epsilon)$ be the number of cubes of side length $\epsilon$ needed to cover $Z$. We record this count in the following table.
        \vskip.1in

        \begin{center}
        $\begin{array}{c|c|c|c|c}
                        &  \,\,d=1 \,\,      & \,\,d=2 \,\,   & \ldots   & d \\
            \epsilon    & N(\epsilon)   & N(\epsilon)   &   & \,\, N(\epsilon) \,\,  \\
            \hline
            \frac{1}{2} & 2         & 4 = 2^2   &  & 2^d \\
            \frac{1}{2^2} & 2^2       & (2^2)^2   &  & (2^2)^d \\
            \frac{1}{2^3} & 2^3       & (2^3)^2   & \,\,\ldots \,\, & (2^3)^d \\
            \vdots & \vdots& \vdots& & \vdots\\
            \frac{1}{2^n} & 2^n       & (2^n)^2   &  & (2^n)^d \\

        \end{array}$
        \end{center}
    \vskip.1in

        In the table above, for $Z = [0,1]^d$, we have
            \[N(\epsilon) = \left(\frac{1}{\epsilon}\right)^d.\]
        Equivalently, taking logs,
            \[ \log N(\epsilon) = d\log\left(\frac{1}{\epsilon}\right), \]
        so
            \[ d = \frac{\log N(\epsilon)}{\log(\frac{1}{\epsilon})}.\]
        The right-hand side of the expression above keeps track of the number of sets needed to cover the metric space $Z$, relative to the diameter of the sets. Taking logs recovers the exponent.
    \end{example}

    The above example generalizes to the notion of box-counting dimension (also known as {\it Minkowski dimension} or {\it Minkowski–Bouligand dimension}), which gives a good first intuition for Hausdorff dimension.

    \begin{defn} [Box-counting dimension]
        Let $Z \subset \R^n$ be a bounded subset. For $\epsilon>0$, let $N(\epsilon)$ be the minimum number of cubes of side length $\epsilon$ needed to cover $Z$. The {\it box-counting dimension} of $Z$ is
            \[ \dim_{\boxd}(Z,d) = \lim_{\epsilon \rightarrow 0} \frac{\log N(\epsilon)}{\log(\frac{1}{\epsilon})}. \]
    \end{defn}

    \begin{example}[Box-counting dimension of the two-thirds Cantor set] \label{ex:CantorPacking}
        A computation similar to that of \Cref{ex:01dim} with covers of diameter $\epsilon_k = \frac{1}{3^k}$ and $N(\epsilon_k) = 2^k$ computes that the box-counting dimension of the two-thirds Cantor set embedded in the unit interval is $\frac{\log 2}{\log 3}$.
    \end{example}

    The box-counting definition generalizes to arbitrary metric spaces via {\it covering dimension} and {\it packing dimension}. In these cases, cubes in $\R^n$ are replaced by open balls in the metric space. For covering dimension, one counts the minimal number of balls required to cover the space. For packing dimension, one counts the maximum number of disjoint balls that can fit in the space. The limit above need not exist, and one can define the upper and lower box dimension via a $\liminf$ and $\limsup$, respectively.

    While box-counting dimension and covering dimension are simpler and easier to work with than Hausdorff dimension, they are only finitely stable in general. That is,
        \[\dim_{\boxd}(A_1 \cup \ldots \cup A_n) = \max\{ \dim_{\boxd}(A_1), \ldots, \dim_{\boxd}(A_n).\} \]
    Hausdorff dimension, on the other hand, is defined via a measure and is countably stable:
        \[ \Hdim \bigcup_{i=1}^{\infty} A_i = \sup_{1 \leq i \leq \infty} \Hdim(A_i). \]
    For example, the box-counting dimension and the covering dimension of the rationals $\Q \subset \R$ equals one, while the Hausdorff dimension of the rationals equals zero.

    A key difference in the definition of Hausdorff dimension is that diameter of the sets used to cover the metric space can vary - their diameter is simply bounded above by $\epsilon$ - and, for each $\epsilon>0$, one can take a countable collection of sets rather than a finite one.

    \begin{defn}[Hausdorff dimension]
        Let $(Z,d)$ be a metric space. Let $\epsilon>0$ and $s>0$. Let
            \[\cH^s_{\epsilon}(Z) := \inf \left\{ \, \sum_{i=1}^{\infty} (\diam U_i)^s \, \, \Bigg| \,\,  \diam(U_i) \leq \epsilon \,\, \textrm{ and } \,\,  Z \subset \bigcup_{i=1}^{\infty} U_i \, \right\}.\]
        The {\it $s$-dimensional Hausdorff measure} is
            \[\cH^s(Z):= \lim_{\epsilon \rightarrow 0} \cH^s_{\epsilon}(Z).\]
        The {\it Hausdorff dimension} of $Z$ is
            \begin{eqnarray*}
                \Hdim(Z) &:=& \inf \{s \geq0 \,\,|\,\, \cH^s(Z) = 0\} \\
                        &=& \sup \{s \geq0 \,\, |\,\, \cH^s(Z) = \infty\}.
            \end{eqnarray*}
        \end{defn}

    There is a limit in the definition of the $s$-dimensional Hausdorff measure. For each $s>0$, the function $\epsilon \mapsto \cH^s_{\epsilon}$ is non-decreasing as $\epsilon \rightarrow 0$, since the family of possible covers decreases as $\epsilon \rightarrow 0$. Hence, this limit exists and lies in $[0,\infty]$. By definition, Hausdorff dimension is a critical exponent.

    \begin{remark}[Computing Hausdorff dimension using balls]  \label{rem:balls} 
        In the defintion of Hausdorff dimension given above, the sets $U_i$ covering the metric space $X$ can be arbitrary. One can define Hausdorff dimension by covering $Z$ using balls instead; both definitions appear frequently in the literature and yield an equal dimension. This observation follows since a cover by balls yields a potential cover in the computation of $\cH^s_{\epsilon}$, and conversely, if $\{U_i\}_{i=1}^{\infty}$ is a cover of $Z$ with $\diam(U_i)<\epsilon$, then each $U_i$ is contained in a ball of radius at most $\epsilon$.
    \end{remark}

    Hausdorff dimension is a bilipschitz invariant of a metric space, but is not invariant under general quasisymmetries, as seen in the next example.

    \begin{example}[Hausdorff dimension increases under snowflaking]   Snowflaking increases the Hausdorff dimension of a metric space. In particular, there is no supremum to the Hausdorff dimension within a quasisymmetry class of a metric space.

    Specifically, let $(Z,d)$ be a metric space. If $\alpha \in (0,1)$, then
        \[ \Hdim(Z,d^\alpha) = \frac{1}{\alpha}\Hdim(Z,d). \]
    Intuition can be obtained by carrying out the same computations as in \Cref{ex:01dim} with $\alpha = \frac{1}{2}$. Lengths are increased by the same exponent at all scales. For example, it now requires 4 sets to cover $[0,1]$ with intervals of diameter $\frac{1}{2}$ in the metric $d^{\alpha}$. Indeed, the distance between consecutive points in the list $0, \frac{1}{4}, \frac{1}{2}, \frac{3}{4}, 1$ now equals $\frac{1}{2}$, and so on.

    To compute the equality rigorously, first show that $\cH^{s/\alpha}(Z,d^{\alpha}) =\cH^s(Z,d)$ for all $s>0$. To see this, let $\epsilon >0$. Suppose $\{U_i\}_{i=1}^{\infty}$ is a cover of $(Z,d)$ by sets of diameter at most $\epsilon$. Then $\{U_i\}_{i=1}^{\infty}$ is a cover of $(Z,d^{\alpha})$ of sets of diameter at most $\epsilon^{\alpha}$. For $U \subset Z$, let $\diam U$ and $\diam_{\alpha} U$ denote the diameter of $U$ in the metrics $d$ and $d^{\alpha}$, respectively. Then $\diam_\alpha U = (\diam U)^{\alpha}$, so
        \[ \sum_{i=1}^{\infty} (\diam_{\alpha} U_i)^{s/\alpha} = \sum_{i=1}^{\infty} (\diam U_i)^s. \]
    Thus, $\cH_{\epsilon^{\alpha}}^{s/\alpha}(Z,d^{\alpha}) \leq \cH_{\epsilon}^s(Z,d)$. Since $\epsilon \rightarrow 0$ as $\epsilon^{\alpha}\rightarrow 0$, this implies $\cH^{s/\alpha}(Z,d^{\alpha}) \leq \cH^s(Z,d)$. Switching the roles of $(Z,d)$ and $(Z,d^{\epsilon})$ gives the other inequality.

    Finally, to see that $\Hdim(Z,d^\alpha) = \frac{1}{\alpha}\Hdim(Z,d)$, suppose that $s>\Hdim(Z,d)$. Then, $\cH^{s/\alpha}(Z,d^{\alpha})  = \cH^s(Z,d) = 0$, so $\Hdim(Z,d^{\alpha}) \leq \frac{s}{\alpha}$. Similarly, if $s<\Hdim(Z,d)$, then $\cH^{s/\alpha}(Z,d^{\alpha})  = \cH^s(Z,d) = \infty$. Thus, $\Hdim(Z,d^{\alpha}) = \frac{1}{\alpha}\Hdim(Z,d)$.
    \end{example}

    \begin{remark}[Upper bounds on Hausdorff dimension via specific covers] \label{rem:HdimUpperbound}
        Upper bounds on the Hausdorff dimension of a metric space are easier to produce than lower bounds.  Upper bounds can be obtained from specifying a family of covers of the space by sets of diameter tending to zero as follows. Let $(Z,d)$ be a metric space. For $\epsilon>0$, let $\{U_i^{\epsilon}\}$ be a cover of $Z$ by sets of diameter at most~$\epsilon$. Then, $\cH^s_{\epsilon}(Z) \leq \sum_i (\diam U_i^{\epsilon})^s$. In practice, typically each cover $\{U_i^{\epsilon}\}$ has finitely many sets, say at most $N(\epsilon)$ of them. Then, $\sum_i (\diam U_i^{\epsilon})^s \leq N(\epsilon) \epsilon^s$. One then needs to find $s$ so that $\lim_{\epsilon \rightarrow 0} \sum_i (\diam U_i^{\epsilon})^s = 0$. Again, in practice, the choice of $s$ is often clear from the relationship between $\epsilon$ and $N(\epsilon)$ as in the examples given earlier in this section. Then, with this choice of $s$,
            \[ \cH^s(Z)  := \lim_{\epsilon \rightarrow 0} \cH^s_{\epsilon}(Z) \leq \lim_{\epsilon \rightarrow 0} \sum_i (\diam U_i^{\epsilon})^s  = 0.\]
        Thus, $\Hdim(Z,d) \leq s$.
    \end{remark}

    \begin{remark}[Lower bounds on Hausdorff dimension via the Mass Distribution Principle] A lower bound on Hausdorff dimension can be obtained by arguing that each cover in a family of covers achieves the infimum in $\cH^s_{\epsilon}(Z)$. This strategy works in special cases like the two-thirds Cantor set as described in \cite[Example 2.7]{falconer}. However, this method is challenging in general, and there is an alternative tool called the Mass Distribution Principle that one can apply in many cases instead. This theorem, whose proof follows from definitions and is given below, relates Hausdorff dimension (and the Hausdorff measures) to the existence of a measure for which the measure of balls behaves nicely.
    \end{remark}

    \begin{thm}[Mass Distribution Principle]
        Suppose there exists a positive measure $\mu$ on $(Z,d)$ so that $\mu\bigl(B_r(z)\bigr) \leq Cr^s$ for some $C>0$ and for all $r>0$ and $z \in Z$. Then, $\cH^s(Z) >0$, and hence $s \leq \Hdim(Z,d)$.
    \end{thm}
    \begin{proof}
        Let $\epsilon>0$, and let $\{U_i\}$ be a cover of $Z$ of sets of diameter at most $\epsilon$. We first show that
            \begin{eqnarray} \label{eqn:MDP}
                0 < \frac{1}{C}\mu(Z) \leq \sum_i (\diam U_i)^s.
            \end{eqnarray}
        Indeed, let $z_i \in Z$ so that $U_i \subset B_{\diam U_i}(z_i)$. Then
            \[ 0  < \mu(Z) = \mu \left( \bigcup_i U_i \right) \leq \sum_i \mu(U_i) \leq \sum_i \mu\bigl(B_{\diam U_i}(z_i)\bigr) \leq \sum_i C(\diam U_i)^s. \]
        Equation~\ref{eqn:MDP} implies that $\cH^s_{\epsilon}(Z) \geq \frac{1}{C}\mu(Z)>0$ for all $\epsilon$. Therefore, $\cH^s(Z) := \lim_{\epsilon \rightarrow 0}\cH^s_{\epsilon}(Z) \geq \frac{1}{C}\mu(Z) >0$, so $s \leq \Hdim(Z,d)$.
    \end{proof}

    The Mass Distribution Principle together with a natural family of covers allows one to compute the Hausdorff dimension of familiar fractals as follows.

    \begin{example}[Hausdorff dimension of the unit interval]
        The computation carried out with covers of the unit interval in \Cref{ex:01dim} shows that 1 is an upper bound for the Hausdorff dimension of the unit interval. The Mass Distribution Principle with respect to Lebesgue measure yields a lower bound of $1$.
    \end{example}

    \begin{example}[Hausdorff dimension of the two-thirds Cantor set] \cite[Page 55]{falconer}. \label{ex:HDimCantor}
        An upper bound of $\frac{\log 2}{\log 3}$ on the Hausdorff dimension of the two-thirds Cantor set $\cC$ follows from the natural covers described in \Cref{ex:CantorPacking}. There is also a natural measure on the Cantor set to which the Mass Distribution Principle can be applied. More specifically, let $\mu$ be the probability measure on the Cantor set constructed by assigning measure $\frac{1}{2^k}$ to each of the $2^k$ intervals of length $\frac{1}{3^k}$ in the construction of the Cantor set. Fix $B_r(z) \subset \cC$ for some $0<r$ and $z \in \cC$. Let $k \in \N$ such that $\frac{1}{3^{k+1}} \leq 2r \leq \frac{1}{3^k}$. Then $B_r(z)$ intersects at most one of the intervals of length $\frac{1}{3^k}$ in the construction of $\cC$. Therefore,
            \[ \mu\bigl(B_r(z)\bigr) \leq \frac{1}{2^k} = \left( \frac{1}{3^k}\right)^{\log 2/\log 3} \leq (6r)^{\log 2/\log 3}.  \]
        Thus, $\Hdim(\cC) = \frac{\log 2}{\log 3}$, where the lower bound follows from the Mass Distribution Principle.
    \end{example}

    The above discussion, on upper bounding Hausdorff dimension via specific covers and using a nice measure to give lower bounds, can be combined in the case that the metric space is Ahlfors regular in the following sense.

    \begin{defn}[Ahlfors regular] \label{def:ahlfors}
        A metric space $(Z,d)$ is {\it Ahlfors $s$-regular} for $s>0$ if there exists a Borel measure $\mu$ on $Z$ and a constant $C<\infty$ so that
            \[\frac{1}{C}r^s \leq \mu(\overline{B}_r(z)) \leq Cr^s\]
         for every closed ball of radius $r$ in $Z$.
    \end{defn}

    \begin{lem} \cite[Chapter 8.7]{heinonen-Lectures}
        If $(Z,d)$ is an Ahlfors $s$-regular metric space, then the Hausdorff dimension of $Z$ equals $s$.
    \end{lem}

    Topological dimension, or {\it Lebesgue covering dimension}, is defined for any topological space and provides a useful lower bound on the Hausdorff dimension of a metric space. As with Hausdorff dimension, topological dimension is also defined with respect to covers of the space. We recall the definition here for convenience.

    \begin{defn}[Topological dimension]
        Let $Z$ be a topological space.  The {\it order} of an open cover $\cU$ of $Z$ is equal to the smallest number $n$, if it exists, for which each point of $Z$ is contained in at most $n$ open sets in $\cU$. A {\it refinement} of an open cover $\cU = \{U_{\alpha}\}$ is an open cover $\cV = \{V_\beta\}$ so that for all $V_\beta \in \cV$ there exists $U_\alpha \in \cU$ so that $V_\beta \subset U_\alpha$. The {\it topological dimension} of $Z$ is the minimum $n$, if it exists, such that every finite open cover $\cU$ of $Z$ has an open refinement $\cV$ with order $n+1$. If no such minimal $n$ exists, then the space is defined to have infinite topological dimension.
    \end{defn}

    \begin{thm} \label{thm:TopdimHausdim} \cite[Theorem 8.14]{heinonen-Lectures}
        Let $Z$ be a metric space equipped with the metric topology. Then the topological dimension of $Z$ is less than or equal to the Hausdorff dimension of $Z$.
    \end{thm}

    \subsection{Hausdorff dimension of boundaries}

    Coornaert~\cite{coornaert93} computed the Hausdorff dimension of the boundary of a hyperbolic group equipped with a visual metric. This result, which is stated below, relates the Hausdorff dimension of the boundary to the {\it volume entropy} or {\it critical exponent of the Poincar\'{e} series} for the group. This correspondence generalizes work of Patterson~\cite{patterson76} in the case of Fuchsian groups and Sullivan~\cite{sullivan79,sullivan84} for limit sets of geometrically finite Kleinian groups acting by isometries on $\Hy^n$. See also work of Bishop--Jones~\cite{bishopjones} for arbitrary Kleinian groups.

    Before stating the general results, we begin with a special, motivating case.

    \begin{example}[Covering dimension of $(\p T_4, d_a)$] \label{ex:CovDimTree}

        Let $T_4$ denote the $4$-regular tree where every edge has length one. Fix a basepoint $p \in T_4$. For each $a>1$, the function
            \[ d_a(\eta,\eta') := a^{-(\eta,\eta')_p}\]
        defines a metric on $\p T_4$.
        For a metric space $(Z,d)$ and $\epsilon>0$, let $N(\epsilon)$ be the minimal number of sets of diameter $\epsilon$ needed to cover $Z$. Then the {\it covering dimension} of $Z$ is given by
            \[\dim_{\covering}(Z,d):= \lim_{\epsilon \rightarrow 0}\frac{N(\epsilon)}{\log\frac{1}{\epsilon}}.\]

        As a first example, we suggest to find the covering dimension of the following metric spaces:  $(\p T_4, d_e)$,  $(\p T_4, d_2)$,  $(\p T_4, d_a)$, $(\p T_n, d_a)$. We discuss the solution to (3), as the others are similar and are intended to build intuition for the Hausdorff dimension of the boundary as the above parameters vary.
        Let
            \[\epsilon_k = a^{-k}.\]
        The key observation in this example is that
            \begin{eqnarray*}
                N(\epsilon_k) &:=& \textrm{minimal number of sets of diameter } \epsilon_k \textrm{ needed to cover } \p T_4 \\
                &=& \textrm{the number of points on the sphere of radius } k \textrm{ in } T_4  \\
                &=& 4\cdot 3^{k-1}.
            \end{eqnarray*}

        \begin{figure}
                \begin{centering}
	            \begin{overpic}[width=.4\textwidth, tics=5]{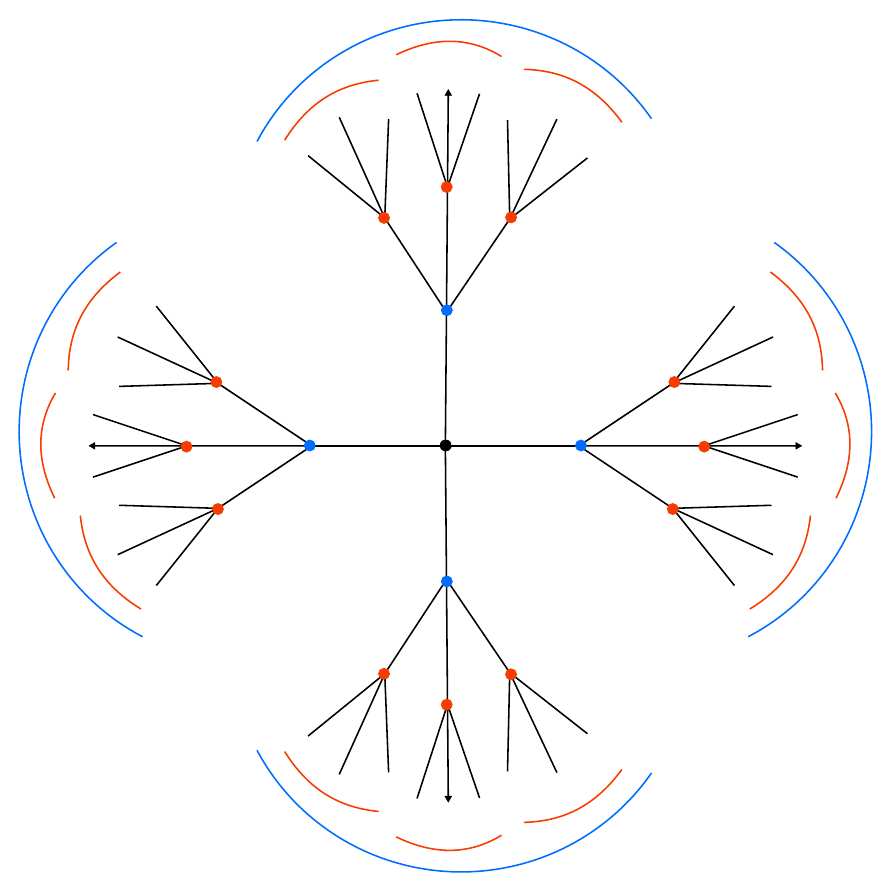} 
                    \put(51.5,46){\small{$p$}}
                \end{overpic}
	           \caption{\small{Covers of the boundary of the tree corresponding to vertices on spheres about $p$.}}
	           \label{figure-HausdimTree}
                \end{centering}
            \end{figure}

        Indeed, let $v_1, \ldots, v_n$ denote the $4 \cdot 3^{k-1}$ vertices on the sphere of radius $k$ about the point $p$ in~$T_4$. The collection $U_i$ of rays emanating from $p$ and passing through vertex $v_i$ is a set of diameter $a^{-k}$. See \Cref{figure-HausdimTree}. The collection $U_1, \ldots, U_n$ covers $\p T_4$ and there is no cover of sets of diameter $a^{-k}$ with fewer elements. Therefore,
            \begin{eqnarray*}
               \dim_{\covering}(\p T_4,d_a)
                &:=& \lim_{\epsilon_k \rightarrow 0}\frac{N(\epsilon_k)}{\log(\frac{1}{\epsilon_k})} \\
                &=& \lim_{k \rightarrow \infty} \frac{\log (4\cdot3^{k-1})}{\log a^{k}} \\
                &=& \lim_{k \rightarrow \infty} \frac{(k-1)\log 3 + \log 4}{k \log a} \\
                &=& \frac{\log 3}{\log a}.
            \end{eqnarray*}
    \end{example}

    For the boundary of the regular tree equipped with a visual metric, the Hausdorff dimension is equal to the covering dimension computed in the previous example. Indeed, the collection of covers of $\p T_4$ yield an upper bound on the Hausdorff dimension, and a natural measure on this Cantor set yields the lower bound via the Mass Distribution Principle. We provide details in the next example.

    \begin{example}[Hausdorff dimension of $(\p T_4, d_a)$]   \label{ex:HdimCantorBoundary}
        As above, let $T_4$ denote the $4$-regular tree where every edge has length one. Fix a basepoint $p \in T_4$. For each $a>1$, the function
            \[ d_a(\eta,\eta') := a^{-(\eta,\eta')_p}\]
        defines a metric on $\p T_4$.

        \noindent {\sc Upper bound:} We use the strategy described in \Cref{rem:HdimUpperbound} to find an upper bound on the Hausdorff dimension of $(\p T_4, d_a)$ via specific covers of the boundary.  Let $\epsilon>0$. Suppose $k \in \N$ so that $\epsilon<a^{-k}$. As described in \Cref{ex:CovDimTree}, there exists a cover of $(\p T_4, d_a)$ by $N(\epsilon) = 4\cdot 3^{k-1}$ sets of diameter at most $\epsilon$. Let $\delta>0$. Let $s = \frac{\log 3}{\log a} + \delta$. Then
            \begin{eqnarray*}
                \cH^s_\epsilon(\p T_4) &\leq & N(\epsilon)\epsilon^s \\
                &\leq & 4 \cdot 3^{k-1}\left(\frac{1}{a^k}\right)^{(\log 3 / \log a) + \delta} \\
                & = & 4 \cdot 3^{k-1}\left(\frac{1}{a^k}\right)^{\log 3 / \log a
                } a^{-k\delta} \\
                &=& 4 \cdot 3^{k-1}\frac{1}{3^k} a^{-k\delta},
            \end{eqnarray*}
        which tends to zero as $k \rightarrow \infty$. Therefore, $\Hdim (\p T_4, d_a) \leq \frac{\log 3}{\log a}$.

        \noindent {\sc Lower bound:} A lower bound is obtained via the Mass Distribution Principle, similar to the Cantor set in \Cref{ex:HDimCantor}. Let $\mu$ be the probability measure on $\p T_4$ constructed by assigning measure $\frac{1}{4\cdot 3^{k-1}}$ to each of the sets in the cover $\p T_4$ that corresponds to the $4 \cdot 3^{k-1}$ points on the $k$-sphere, as descried in \Cref{ex:CovDimTree}. Let $B_r(z) \in (\p T_4, d_a)$ for some $r>0$ and $z \in \p T_4$. Let $k \in \N$ so that $\frac{1}{a^{k+1}} \leq r \leq \frac{1}{a^k}$. Then, if $z' \in B_r(z)$, then $d_a(z,z') \leq r$, so the Gromov product satisfies $(z,z')_p \geq k$. Thus, the points $z$ and $z'$ must be in the same set in the cover corresponding to points in the sphere of radius $k$ about $p$ in the tree. Therefore,
            \[\mu\bigl(B_r(z) \bigr) \leq \frac{1}{4\cdot 3^{k-1}}  = \frac{3}{4}\cdot \frac{1}{3^k}  = \frac{3}{4} \left( \frac{1}{a^k} \right)^{\log 3 / \log a} \leq \frac{3}{4} \cdot r^{\log 3 / \log a}.  \]
        Therefore, by the Mass Distribution Principle, $\Hdim(\p T_4, d_a) \geq \frac{\log 3}{\log a}$.

        Together, these bounds show
            \[ \Hdim(\p T_4, d_a) = \frac{\log 3}{\log a}. \]
    \end{example}

    In the above example, the Hausdorff dimension of the boundary of the tree equals to the logarithm of the exponential growth rate in the tree divided by the logarithm of the visual metric parameter. This behavior holds in general, as shown by Coornaert~\cite{coornaert93}.

    \begin{defn}
        Let $G$ be a group acting by isometries on a metric space $X$. Let $p \in X$. Let $|G \cdot p \,\cap\, B_n(p)|$ denote the cardinality of the intersection of the orbit $G \cdot p$ with the ball of radius $n$ about $p$.
        The {\it volume entropy} or {\it critical exponent} of the action is defined by
            \[ h = h(G \curvearrowright X) = \limsup_{n \rightarrow \infty} \frac{\log |G \cdot p \,\cap\, B_n(p)|}{n}. \]
    \end{defn}

    \begin{thm} \cite[Corollary 7.6]{coornaert93} \label{thm:HausdimBoundary}
        Let $G$ be a hyperbolic group, and suppose $G$ acts on $X$ geometrically. Let $h = h(G \curvearrowright X)$ be the critical exponent of the action of $G$ on $X$. Let $d_a$ be a visual metric on $\p X$ with parameter $a$. Then
            \[ \Hdim(\p X, d_a) = \frac{h}{\log a}. \]
    \end{thm}

    As in the previous examples, the proof of \Cref{thm:HausdimBoundary} gives both an upper and lower bound on the Hausdorff dimension. The upper bound is obtained by covering the boundary by specific sets and counting them using the growth rate of the group. In the tree examples, the cover corresponding to a vertex on the sphere of radius $k$ around the basepoint $p$ consisted of all rays emanating from $p$ and passing through the vertex. In the general setting, one covers the boundary by the set of shadows of balls of a fixed radius $R$ around points on the sphere of radius $k$ around the basepoint $p$; see \Cref{defn:shadow}.  These sets have size at most $Ca^{-k}$, where $C$ is a fixed constant, and $a$ is the visual metric parameter. The growth rate of the group counts the number of these sets. The lower bound uses the Mass Distribution Principle, where the measure is the Patterson--Sullivan measure.

    An alternative proof of \Cref{thm:HausdimBoundary} follows from work of Paulin, who generalized Coornaert's result to the conical limit set of a group acting by isometries on an almost extendable hyperbolic metric space \cite[Theorem 2.1]{paulin97}. 
    Paulin's lower bound argument also uses the Mass Distribution Principle, but by embedding a tree in the hyperbolic space and using a natural measure on the boundary of the tree. Calegari also gave a proof of \Cref{thm:HausdimBoundary} using the Patterson--Sullivan measure in the notes \cite[Section 2.5]{calegari13}.

    \begin{remark}[Relatively hyperbolic group case]
        Potyagailo--Yang~\cite{potyagailoyang} analyze the Hausdorff dimension of boundaries of relatively hyperbolic group pairs. They consider the Floyd boundary together with the Floyd metric as well as the Bowditch boundary together with the shortcut metric. In both cases, they obtain a result similar to \Cref{thm:HausdimBoundary}; the Hausdorff dimension equals an explicit constant times the growth rate of the group. They note, with references in \cite[Section 1.2]{potyagailoyang}, that the Hausdorff dimension of the Bowditch boundary equipped with the visual metric of a relatively hyperbolic group pair can be infinite when the peripherals contain nonabelian free groups, for example.
    \end{remark}

    \subsection{Conformal dimension}

    \begin{defn}[Conformal dimension]
        Let $Z$ be a metric space. The {\it conformal dimension} of~$Z$, denoted $\Cdim(Z)$, is the infimal Hausdorff dimension of a metric space $Z'$ quasisymmetric to~$Z$.
    \end{defn}

    Crucially, conformal dimension is a quasisymmetry invariant of metric spaces by definition. Thus, conformal dimension yields a quasi-isometry invariant of hyperbolic groups by \Cref{thm:QI_QS}.

    \begin{example}[Conformal dimension of the boundary of free groups]
        Every finitely generated free group of rank at least two acts geometrically on the 4-valent tree. \Cref{ex:HdimCantorBoundary} computes the Hausdorff dimension of the boundary of this tree equipped with a visual metric:
            \[ \Hdim(\p T_4, d_a) = \frac{\log 3}{\log a}. \]
        Since the visual metric parameter $a$ can be taken arbitrary large, and changing the parameter is a quasisymmetry, the conformal dimension of the boundary of a free group equals zero.

        Alternatively, one can see that the conformal dimension is zero by fixing a visual metric parameter and changing the metric on the tree by increasing the edge lengths towards infinity. These transformations are quasi-isometries that extend to quasisymmetries of the boundary. The Hausdorff dimension can be computed as above and tends to zero.
    \end{example}

    \begin{remark}[Lower bounds on conformal dimension]
        The topological dimension of a metric space yields a first lower bound on the conformal dimension of the metric space by~\Cref{thm:TopdimHausdim}.

        We will discuss Gromov's Round Trees in the next section, which often provides a much larger lower bound on the conformal dimension of a metric space. For example, these can be used to prove the conformal dimension of various one-dimensional boundaries, including boundaries of Fuchsian buildings, tends to infinity across certain families.
    \end{remark}

    \begin{remark}[Upper bounds on conformal dimension]
        An upper bound on the conformal dimension of a metric space can be obtained by finding (or finding an upper bound on) the Hausdorff dimension of any metric in the quasisymmetry class of the space.

        The Hausdorff dimension of the boundary of a hyperbolic group with a visual metric with parameter~$a$ is given by Theorem~\ref{thm:HausdimBoundary}:
        \[ \Hdim(\p X, d_a) = \frac{h}{\log a}, \]
        where $h$ is the volume entropy.
        As detailed in \Cref{sec:visual_metrics}, the allowable visual metric parameters for $\p X$ depend on the $\delta$-hyperbolicity constant of the space $X$. Thus, to find an upper bound on Hausdorff dimension of the boundary, it suffices to find an upper bound on the exponential growth rate of the group action together with an upper bound on the $\delta$-hyperbolicity constant of the space. The former is typically easy to obtain using the smallest translation length of a generator. However, computing hyperbolicity constants is often quite hard. We note that the presence of a $\CAT(-1)$ metric is especially useful here, as one can take the visual metric parameter to equal~$e$. Then, an upper bound on the Hausdorff dimension is given by finding an upper bound on the volume entropy.

        There are stronger techniques to provide upper bounds on conformal dimension, including $\ell^p$-cohomology and combinatorial modulus, which are not discussed in this article. These powerful ideas have been applied by Carrasco~\cite{carrasco}, Bourdon--Kleiner~\cite{bourdonkleiner13, bourdonkleiner15}, Mackay~\cite{mackay16}, and Carrasco--Mackay~\cite{carrascomackay}.
    \end{remark}

    \section{Round trees and arcs in the boundary} \label{sec:round_trees}

    In this section we describe the round tree technique, and then we survey its applications. We refer the reader to \cite[Chapter 4]{mackaytyson} for a general discussion of lower bounds on conformal dimension.

\subsection{The product of a Cantor set and an interval}

    We present a motivating special case before giving more general results.
    Recall that the Hausdorff dimension of the two-thirds Cantor set equals $\frac{\log 2}{\log 3}$, while its conformal dimension equals zero. The next example illustrates that taking the product with an interval has a remarkable consequence: the resultant space is one that achieves the conformal dimension. See \Cref{figure-CantorProduct}.

    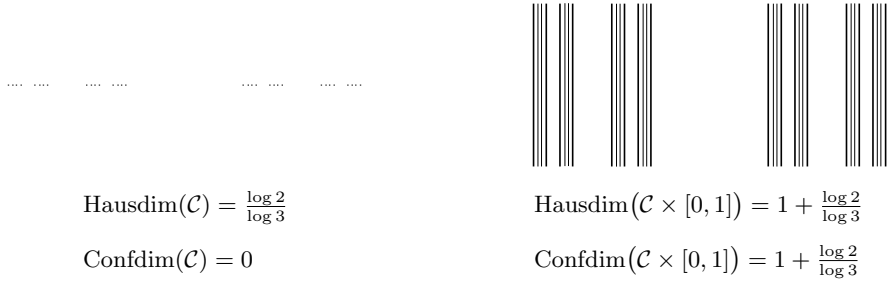
\begin{figure}
                \begin{centering}
	            \begin{overpic}[width=.8\textwidth, tics=5]{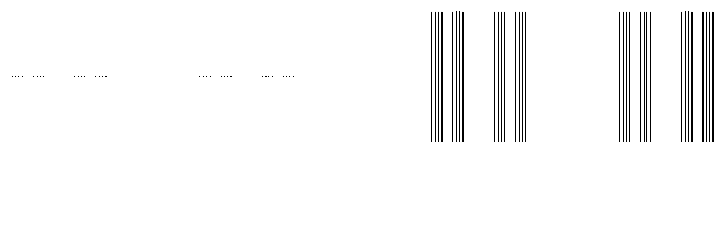} 
                    \put(10,2){\small{$\Cdim(\cC) = 0$}}
                    \put(10,8){\small{$\Hdim(\cC) = \frac{\log 2}{\log 3}$}}
                    \put(60,2){\small{$\Cdim\bigl(\cC\times [0,1]\bigr) =  1+ \frac{\log 2}{\log 3}$}}
                    \put(60,8){\small{$\Hdim\bigl(\cC \times [0,1]\bigr) = 1+ \frac{\log 2}{\log 3}$}}
                \end{overpic}
	           \caption{\small{A stabilization occurs with respect to conformal dimension upon taking the product with an interval. While the conformal dimension of the two-thirds Cantor set is zero, the conformal dimension of the product of the Cantor set and the interval is achieved by the Hausdorff dimension of this space.}}
	           \label{figure-CantorProduct}
                \end{centering}
            \end{figure}

    \begin{thm} \label{thm:cantorInterval}
        Let $\cC$ be the two-thirds Cantor set, and let $Z = \cC \times [0,1]$ equipped with the product metric. Then
            \[\Cdim(Z) = \Hdim(Z) = \Hdim (\cC) + 1 = \frac{\log 2}{\log 3} + 1.\]
    \end{thm}
    \begin{proof}
        The definition of conformal dimension implies $\Cdim(Z) \leq \Hdim(Z)$.
        The Hausdorff dimension of $Z$ follows from the calculations in \Cref{ex:HDimCantor} together with \cite[Example 7.6]{falconer}. The remainder of the proof follows the arguments of \cite[Lemme 1.6]{bourdon95} and \cite[Proposition 4.1.3]{mackaytyson}, restricted to this special case.

        Suppose towards a contradiction that $\Cdim(Z) < \Hdim(Z)$. Then, there exists a metric space $(Z',d')$ of Hausdorff dimension strictly less than $p:=1+ \frac{\log 2}{\log 3} = \Hdim(Z)$ and a quasisymmetry $f:Z \rightarrow Z'$.

        Let $\cE$ be the following natural curve family in $Z$:
            \[\cE:= \bigl\{ \{z\} \times [0,1] \,\,|\,\, z \in \cC \bigr\}.\]
        We will use the following two properties of the set $\cE$:

        \begin{enumerate}
            \item The diameter of each $E \in \cE$ is uniformly bounded from below, as $\diam(E) = 1$.

            Consequently, there exists $c'>0$ so that
                \[\diam\bigl(f(E)\bigr) \geq c'\]
            for each $E \in \cE$. Indeed, the map $f$ is a homeomorphism by definition, and since $Z$ and $Z'$ are compact, $f^{-1}$ is uniformly continuous. So, if $\epsilon = \frac{1}{2}$, then there exists $\delta>0$ so that if $d(f(x),f(x'))<\delta$, then $d(x,x')<\frac{1}{2}$. This implies the image of $E \in \cE$ cannot have diameter less than $\delta$, as desired.

            \item There is a probability measure $\mu$ on $\cE$ induced from the natural probability measure $\mu_{\cC}$ on the Cantor set (see \Cref{ex:HDimCantor}). Recall that with respect to the measure $\mu_{\cC}$, if $B_r(z)$ is a ball of radius $r$ in the Cantor set, then $\mu_{\cC}\bigl(B_r(z)\bigr) \leq (6r)^{\log 2/\log 3}$. Thus, there exist constants $A>0$ and $\alpha \geq 0$ so that with respect to the measure $\mu$, if $B_r(z) \subset Z$ is a ball of radius $r$ in $Z$, then
                \[ \mu \bigl\{ E \in \cE \,\,|\,\, E \cap B_r(z) \neq \emptyset  \bigr\} \leq Ar^{\alpha}. \]
            Indeed, $A = 6^{\log 2 / \log 3}$ and $\alpha = \frac{\log 2}{\log 3}$ in this case.

            We record a relationship between $\alpha$ and $p = 1 + \alpha$ used later in the proof. Let $p'$ be the H\"{o}lder conjugate of $p$, so that $\frac{1}{p} + \frac{1}{p'} = 1$ and $p' = \frac{p}{p-1}$. Then $\alpha p' = p$.
        \end{enumerate}

        Let $\epsilon >0$. Since $\Hdim(Z') <p$, there exists a cover of $Z'$ by a collection of balls $\{B_i'\}_{i \in I'}$ so that $B_i'$ has radius $r_i'$ and
            $ \sum_{i \in I'} (r_i')^p <\epsilon$; see \Cref{rem:balls}.
        By \Cref{lemma:ballscover} there exists a subfamility $\{B_i'\}_{i \in I} \subset \{B_i'\}_{i \in I'}$ so that the elements of $\{B_i'\}_{i \in I}$ are pairwise disjoint and $Z = \bigcup_{i \in I} 5B_i'$. (Recall that if $B = B_r(z)$, then $KB = B_{Kr}(z)$.)
        Then,
            \[ \sum_{i \in I} (r_i')^p < \epsilon.  \]

        The quasisymmetry $f$ yields a corresponding family of balls in the metric space $Z$. More specifically, by Lemma~\ref{lemma:annuli} there exists a constant $H$ and balls $\{B_i\}_{i \in I}$ in the metric space $Z$ so that
            \[B_i \subset f^{-1}(B_i') \subset f^{-1}(5B_i') \subset HB_i.\]
        Since $f$ is a homeomorphism, balls in the family $\{B_i\}_{i \in I}$ are pairwise disjoint.

        For each $i \in I$, define a function
            \[h_i: \cE \rightarrow \{0,1\}\]
        given by
            \begin{equation*}
            h_i(E)=
            \begin{cases}
                 1 & f(E) \cap 5B_i' \neq \emptyset\\
                 0 & \text{otherwise}
             \end{cases}
            \end{equation*}

        \begin{claim} \label{claim:c'}
            $\displaystyle \frac{1}{10}c' \leq \sum_{i \in I} r_i'h_i(E).$
        \end{claim}
        \begin{proof}[Proof of Claim.]
            Let $E \in \cE$. Then, we have the following inequalities. The first holds from the definition of $c'$, and the second holds since $\{5B_i'\}_{i \in I}$ covers $Z'$. The final equality follows from the definition of $h_i(E)$.
            \[c' \leq \diam\bigl(f(E)\bigr) \leq \sum_{f(E) \cap 5B_i' \neq \emptyset} \diam(5B_i') \leq \sum_{f(E) \cap 5B_i' \neq \emptyset} 10r_i' = \sum_{i \in I} 10r_i'h_i(E). \]
        \end{proof}

        The proof will conclude after the next sequence of inequalities. We provide justification below.
            \begin{eqnarray}
                \frac{1}{10}c' & \leq & \int_{\cE} \, \sum_{i \in I} r_i' h_i(E) \, d\mu(E)  \label{L1} \\
                & \leq & \sum_{i \in I} r_i' \, \mu \bigl\{ E \in \cE \,|\, f(E) \cap 5B_i' \neq \emptyset \bigr\} \label{L2}  \\
                & \leq & \sum_{i \in I} r_i' \mu \bigl\{ E \in \cE \,|\, E \cap HB_i \neq \emptyset \bigr\} \label{L3} \\
                & \leq & \sum_{i \in I} r_i' A(Hr_i)^{\alpha} \label{L4} \\
                & = & AH^{\alpha} \sum_{i \in I} r_i' r_i^{\alpha} \label{L5} \\
                & \leq & AH^{\alpha} \left( \sum_{i \in I} (r_i')^p \right)^{1/p}\left( \sum_{i \in I} (r_i^{\alpha})^{p'} \right)^{1/p'} \label{L6} \\
                & \leq & AH^{\alpha} \left( \sum_{i \in I} (r_i')^p \right)^{1/p}\left( \sum_{i \in I} r_i^p \right)^{1/p'} \label{L7} \\
                & \leq & AH^{\alpha} \epsilon^{1/p} \bigl(\cH^p(Z) \bigr)^{1/p'}.\label{L8}
            \end{eqnarray}

            Inequality~\ref{L1} follows from integrating both sides of the inequality of \Cref{claim:c'} with respect to the measure $\mu$. Inequality~\ref{L2} follows from the definition of $h_i$.  Inequality~\ref{L3} holds since $f$ is a homeomorphism and $f^{-1}(5B_i') \subset HB_i$. Inequality~\ref{L4} comes from the second property of the curve family. Inequality~\ref{L5} rearranges terms. Inequality~\ref{L6} follows from the H\"{o}lder Inequality. Inequality~\ref{L7} applies that $\alpha p' = p$ as discussed above. Finally, Inequality~\ref{L8} comes from the assumptions on $r_i'$ given above, and the fact that the Hausdorff measure of $Z$ is less than the term involving the $r_i$ because the sets in the collection $\{B_i\}_{i \in I}$ are pairwise disjoint and $p$ is the Hausdorff dimension of $Z$.

            The inequalities combine to $\frac{1}{10}c' \leq AH^{\alpha}\epsilon^{1/p}\bigl( \cH^p(Z)\bigr)^{1/p'}$, a contradiction as $\epsilon$ can be taken arbitrarily small, while the other constants are fixed.
    \end{proof}

    The proof of the previous theorem applied the following useful lemma.

    \begin{lem} \cite[Theorem 1.2]{heinonen-Lectures} \label{lemma:ballscover}
        Let $(Z,d)$ be a metric space. Let $\cF$ be a family of balls in $Z$ so that $Z = \bigcup_{B \in \cF} B$. Assume that the diameters of the balls in $\cF$ are uniformly bounded from above. Then, there exists a subfamily $\cG \subset \cF$ so that the elements of $\cG$ are pairwise disjoint and $Z = \bigcup_{B \in \cG} 5B$.
    \end{lem}

    We now state two generalizations of \Cref{thm:cantorInterval}. Their proofs are modeled on an argument due to Pansu \cite[Proposition 2.9]{pansu}, \cite[Lemma 6.3]{pansu89b} and follow the same reasoning as in the theorem above. The curve family and its measure allows one to compare covers by balls of the two spaces $Z$ and $Z'$ related by a quasisymmetry $f:Z \rightarrow Z'$. The H\"{o}lder inequality is useful to introduce the exponents appearing in the Hausdorff dimension computations.

    \begin{prop} \cite[Lemma 1.6]{bourdon95}
        Let $(Z,d)$ be a compact metric space. Suppose there exists the following:
            \begin{enumerate}
                \item A family of curves $\cC = \{\gamma_i \,|\, i \in I\}$ in $Z$ with diameters uniformly bounded from below.
                \item A probability measure $\mu$ on $\cC$ and constants $A>0$ and $\alpha \geq 0$ so that for every ball of radius $r$ in $Z$,
                    \[ \mu \{\gamma \in \cC\,|\, \gamma \cap B \neq \emptyset \} \leq Ar^{\alpha}.\]
            \end{enumerate}
            Let $\tau$ denote the packing dimension of $Z$. Then
            \[\Cdim(Z) \geq \frac{\tau}{\tau - \alpha} \quad \quad \textrm{ with } \tau-\alpha \geq 1.\]
    \end{prop}

    The second generalization uses a measure on the metric space to replace the radius of a ball.

    \begin{prop} \cite[Proposition 4.1.3]{mackaytyson}
        Let $(Z,d,\mu)$ be a compact, doubling metric measure space, and let $1<p<\infty$. Suppose that there exists a family $\cE$ of connected sets in $Z$ and a probability measure $\mu$ on $\cE$ with the following properties:
            \begin{enumerate}
                \item There exists $c>0$ so that $\diam E \geq c$ for all $E \in \cE$.
                \item There exists $A>0$ and $r_0>0$ so that
                    \[\nu\{E \in \cE \,|\, E \cap B \neq \emptyset\} \leq A\mu(B)^{1/p'} \]
                    for all balls $B$ in $Z$ of radius at most $r_0$.
            \end{enumerate}
        Then, $\Cdim(Z) \geq p$, where $p$ and $p'$ are H\"{o}lder conjugates.
    \end{prop}

\subsection{Gromov's round trees}

    \begin{figure}[t]
    \begin{centering}
	\begin{overpic}[width=.3\textwidth, tics=5]{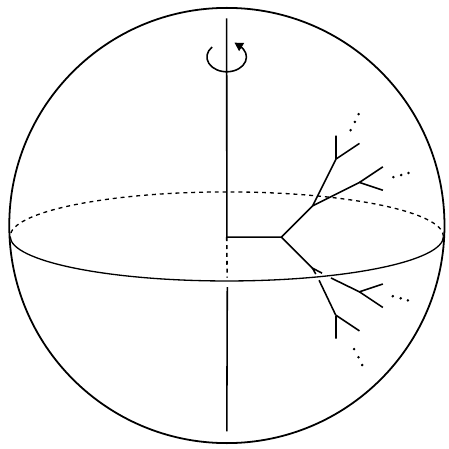} 
        \put(-5,90){$\Hy^3$}
    \end{overpic}
	\caption{\small{Gromov's round trees. }}
	\label{figure-GromovRoundTree}
    \end{centering}
    \end{figure}

    Gromov's round tree construction is a tool to build a family of arcs homeomorphic to the product of a Cantor set and an interval in the boundary of a hyperbolic metric space~\cite[Section 7.C3]{gromov93}. We start with Gromov's original definition and then present a combinatorial point of view. In particular, Mackay's {\it combinatorial round trees}, are useful in building subcomplexes of hyperbolic groups that are quasi-isometric to a round tree.

    \begin{example} \cite[Page 136]{gromov93}
        Gromov begins with the following example. Let $T \subset \Hy^2 \subset \Hy^3$ be a quasi-isometrically embedded rooted tree with boundary the Cantor set. Rotate $T$ around a geodesic as indicated in \Cref{figure-GromovRoundTree}.  The resultant space $\cR$ decomposes into unions of quasi-isometric copies of the hyperbolic plane, which intersect in disks. The boundary $\p \cR$ is homeomorphic to the product of a Cantor set and the circle.
    \end{example}

    \begin{defn} \cite[Page 137]{gromov93}
        A {\it two-dimensional round tree} $X_0$ is a two-dimensional space of negative curvature which admits an isometric action of $S^1$ with a single fixed point $x_0$ so that there is an isometrically embedded tree $T \subset X_0$ which intersects every $S^1$ orbit at exactly one point.
    \end{defn}

    Gromov's discussion focuses on the case that the metric is locally the metric of $\Hy^2$ away from the singular locus. This assumption is typically dropped in applications. To obtain bounds on conformal dimension, one does not require the full $S^1$-action; one can take a subcomplex with boundary the product of the Cantor set and an interval. Such subcomplexes are also referred to as round trees.

    Alternatively, one can construct a round tree by starting with a sector of the hyperbolic plane, viewed in the unit disk model. See \Cref{figure-GromovRoundTree_sector}. Take $b$ copies of this sector and glue them together by isometries along the intersection of the disk in the hyperbolic plane of radius $\ell$ with each sector. For each of the $b$ sectors, take $b-1$ additional sectors, and glue them together by isometries along the intersection of the disk in the hyperbolic plane of radius $2\ell$ with each sector. Continue in this way to obtain a round tree, built from sectors in the hyperbolic plane that naturally branch according to the $b$-valent rooted tree with edges of length $\ell$. The boundary is homeomorphic to the product of a Cantor set and an interval. Crucially, the Hausdorff dimension of the boundary of the round tree, and hence its conformal dimension by \Cref{thm:cantorInterval}, can be computed with respect to the degree of the tree and the distance between vertices.

     \begin{figure}[t]
    \begin{centering}
	\begin{overpic}[width=.3\textwidth, tics=5]{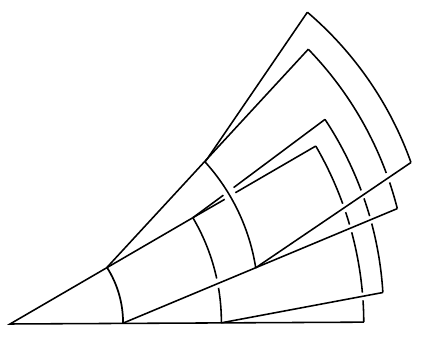} 

    \end{overpic}
	\caption{\small{A portion of a round tree obtained by gluing sectors of the hyperbolic plane (viewed in the disk model) along initial subcomplexes. The boundary of the limiting object is homeomorphic to the product of the Cantor set and an interval. }}
	\label{figure-GromovRoundTree_sector}
    \end{centering}
    \end{figure}

    To obtain lower bounds on conformal dimension, one seeks to quasi-isometrically embed a round tree in a hyperbolic metric space. Mackay~\cite{mackay16} gave a combinatorial version of round trees that are well-suited to polyhedral complexes. The results below generalize work of Bourdon~\cite{bourdon95} and Mackay~\cite{mackay}.

    If $A'\subset A$ is a subcomplex of a polygonal complex $A$, then the {\it star} of $A'$ in $A$, denoted $\St(A')$, is the union of all closed cells in $A$ that nontrivially intersect $A'$.  Let $\N = \Z_{\geq 1}$.

    \begin{defn}[Combinatorial round tree] \cite[Definition 7.1]{mackay16} \label{defn:RT}
        A polygonal $2$-complex $A$ is a {\it combinatorial round tree} with vertical branching $V \in \N$ and horizontal branching at most $H \in \N$ if the following holds. Let $T = \{1,2, \ldots, V\}$.   Then
            \[ A = \bigcup_{\ba \in T^{\N}} A_{\ba}, \]
        where
        \begin{enumerate}
            \item The complex $A$ has a basepoint $x_0$ contained in the boundary of a unique $2$-cell $A_0' \subset A$.
            \item Each complex $A_{\ba}$ is an infinite planar $2$-complex homeomorphic to a half-plane whose boundary is the union of two rays $L_{\ba}$ and $R_{\ba}$ with $L_{\ba} \cap R_{\ba} = \{x_0\}$.

            \item Let $A_0$ be a union of $2$-cells so that $A_0' \subseteq A_0$ and $A_0$ is homeomorphic to a closed disk.
            For $n>0$, the {\it round tree at step $n$} is defined as $A_n = \St(A_{n-1})$. Given $\ba = (a_1, a_2, \ldots) \in T^\N$, let $\ba_n = (a_1, \ldots, a_n) \in T^n$. If $\ba,\ba' \in T^\N$ satisfy $\ba_n = \ba'_n$ and $\ba_{n+1} \neq \ba'_{n+1}$, then
                \[ A_n \cap A_{\ba} \subseteq A_{\ba'} \cap A_{\ba} \subseteq A_{n+1} \cap A_{\ba}.\]
            \item Each $2$-cell $P \subseteq A_n$ meets at most $VH$ $2$-cells in $A_{n+1} \backslash A_n$.
        \end{enumerate}
    \end{defn}

    \begin{remark}
         A combinatorial round tree has the following structure, as  explained by Mackay \cite{mackay16}. Each complex $A_n$ is the union of $V^n$ distinct planar $2$-complexes $\{A_{\ba_n}\}_{\ba_n \in T^n}$. Each complex $A_{\ba_n} \subset A_n$ is homeomorphic to a disk. The complexes $A_{\ba_n} \subset A_n$ are glued together along their initial subcomplexes in a branching fashion. The boundary of $A_{\ba_n}$ decomposes as a union of three paths: $L_{\ba} \cap A_n$, $R_{\ba} \cap A_n$, and a connected path $E_{\ba_n}$. The complex $A_{\ba_{n+1}}$ is built from $A_{\ba_n}$ by attaching $2$-cells in $V$ strips along the path $E_{\ba_n}$. In each strip, each $2$-cell in $A_{\ba_n}$ is adjacent to at most $H$ new $2$-cells.
    \end{remark}

   Lower bounds on the conformal dimension of a hyperbolic polygonal complex then follow from the following theorem.

    \begin{thm} \cite[Theorem 7.2]{mackay16} \label{thm:mackayRT}
        Let $X$ be a hyperbolic polygonal $2$-complex. Suppose there is a combinatorial round tree $A$ with vertical branching $V \geq 2$ and horizontal branching $H \geq 2$. Suppose that $A^{(1)}$, with the natural length metric giving each edge length one, admits a quasi-isometric embedding into $X$. Then
            \[ \Cdim(\p X) \geq 1 + \frac{\log V}{\log H}.\]
    \end{thm}

    In the proof of the theorem above, one embeds a tree in the dual graph to the polygons. The Hausdorff dimension of the Cantor set boundary can be bounded in terms of $V$ and $H$.

\subsection{Bourdon buildings}

    Bourdon~\cite{bourdon95} utilized Gromov's round trees to give lower bounds on the conformal dimension of certain right-angled Fuchsian buildings, which are often referred to as {\it Bourdon buildings}. A discussion of constructing round trees in these examples also appears in Gromov's article. The boundary of each building is homeomorphic to the Menger curve~\cite{benakli, dymaraosajda}. Bourdon's result presented the first family of hyperbolic groups with one-dimensional topological boundary for which the conformal dimension tends to infinity.

    Subsequent work of Bourdon~\cite{bourdon} computed the conformal dimension of the buildings exactly. Interestingly, non-isometric buildings can have equal conformal dimension. Moreover, the result proved that the conformal dimension values over all spaces in this family forms a dense subset of $(1, \infty)$. Bourdon--Pajot~\cite{bourdonpajot} then proved a strong form of quasi-isometric rigidity in this setting, showing that any quasi-isometry between two such buildings is bounded distance from an isometry. Thus, in contrast to the setting of the rank-1 symmetric spaces, the homeomorphism type of the boundary together with its conformal dimension is not a complete quasi-isometry invariant within this family of spaces. See \cite{bourdonpajot-survey} for additional discussion of these examples and their context.

    \begin{defn}[Bourdon building $I_{p,q}$]
        Let $p \geq 5$ and $q \geq 3$ be integers. The {\it Bourdon building} $I_{p,q}$ is a $\CAT(-1)$ polyhedral complex where each $2$-cell is isometric to a regular right-angled hyperbolic $p$-gon. Each edge of each polygon is contained in exactly $q$ distinct polygons. The link of each vertex is isomorphic to the complete bipartite graph $K_{q,q}$ with two vertex sets of size $q$. As shown by Bourdon~\cite[Proposition 2.2.1]{bourdon}, there is a unique polyhedral complex with these properties.
    \end{defn}

        Note that if $q=2$, so the link of each vertex is a $4$-cycle, then the resulting space is isometric to the hyperbolic plane. As $K_{2,2}$ is a subgraph of $K_{q,q}$ for $q>2$ in many ways, the building contains many (isometrically embedded) copies of the hyperbolic plane.

        The following graph product acts geometrically on the building $I_{p,q}$,
            \[\Gamma_{p,q}:= \la \, s_1, \ldots, s_p \,|\, s_i^q, [s_i, s_{i+1}] \,\ra.\]
        A fundamental domain for the action is a single polygon in the building, and one can view each generator as acting by rotation around a convex subcomplex containing an edge of a polygon.

        \begin{thm} \cite[Theorem 1.1]{bourdon}
            Let $I_{p,q}$ be a building as described above, and let $d_a$ be a visual metric on its boundary. Then
            \[ \Cdim(\p I_{p,q}, d_a) = 1 + \frac{\log (q-1)}{\arccosh\left( \frac{p-2}{2}\right)}. \]
        \end{thm}

        Bourdon notes that for $p' = (p-2)^2$ and $q'-1 = (q-1)^2$, the boundaries of the associated buildings have equal conformal dimension. For example, $\Cdim(\p I_{5,3}) = \Cdim(\p I_{9,5})$. See also \cite[Section 3.5]{mackaytyson} for additional discussion of these examples.

    \begin{figure}[t]
    \begin{centering}
	\begin{overpic}[width=.5\textwidth,  tics=5]{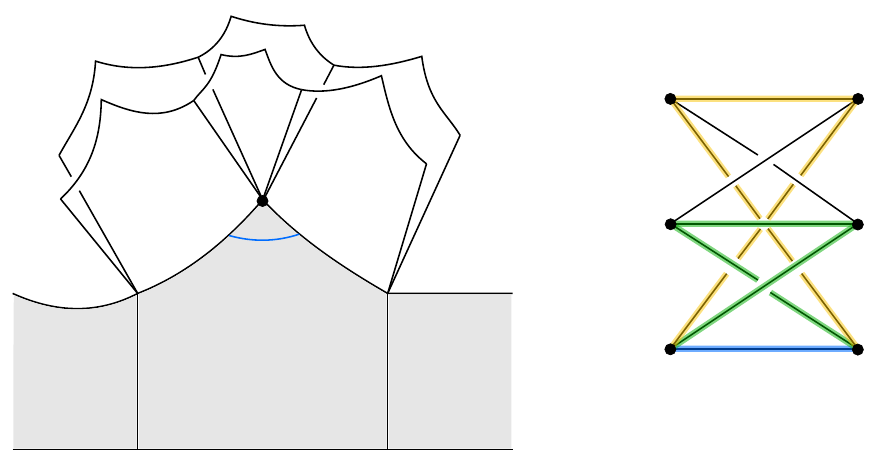} 
        \put(32,29){\small{$v$}}
        \put(5,10){\small{$A$}}
        \put(82,5){\small{$\Lk(v)$}}
    \end{overpic}
	\caption{\small{Building a round tree in a Bourdon building $I_{5,3}$. The complex is extended along strips at each vertex; this construction at $v$ is possible as the link of a vertex $v$ contains many 4-cycles that intersect only in an edge.}}
	\label{figure-BourdonBuilding}
    \end{centering}
    \end{figure}

        We indicate some of the intuition behind constructing round trees in this example. See \Cref{figure-BourdonBuilding}. The round tree $A$ is built recursively, by building branching sectors one strip at a time. The link of each vertex in the building is sufficiently rich that each strip can be extended to $q-1$ new strips. Suppose an initial portion of the round tree $A$ has been constructed and $P$ is a polygon on the outer boundary of this complex. There are various cases to consider, but to give one example, suppose $v \in P$ is a outer vertex, so that its incident edges are not glued to other polygons in the subcomplex previously constructed. The link of $v$ is $K_{q,q}$, and only one edge of this link corresponds to a polygon already in the subcomplex. This edge $e$ lies in $q-1$ $4$-cycles in the link that pairwise intersect only in $e$. These $4$-cycles correspond to polygons that can be added to the subcomplex to build pieces of $q-1$ new strips at the vertex $v$. One must then verify various compatibility conditions, as the complex is extended over various vertices.

\subsection{Mackay's applications}

    Mackay's work uses the round tree techniques in two quite different settings. The first result is remarkably general. Mackay produces round trees in the boundary of any one-ended hyperbolic group that does not split over a two-ended subgroup. As a consequence, Mackay proves the conformal dimension of any such group is strictly greater than one.
    The second application is to produce bounds on the conformal dimension of groups satisfying certain small cancellation conditions. As a consequence, Mackay obtains bounds on the conformal dimension of random groups, proving there are infinitely many quasi-isometry classes as certain parameters vary. The results are as follows.

    \begin{thm} \cite[Corollary 1.2]{mackay10}
        Suppose that $G$ is a hyperbolic group whose boundary is non-empty, connected, and has no local cut points. Then the conformal dimension of $\p G$ is strictly greater than one.
    \end{thm}

    The strategy to prove the above theorem begins with the existence of a quasi-arc in the boundary, as discussed in \Cref{sec:qs_invariants}. Mackay then performs a splitting operation, using the topology of the boundary, to obtain two arcs. He applies an arc straightening procedure to obtain two uniform quasi-arcs and repeats the process until the product of a Cantor set and an interval appears in the boundary. We note that Carrasco--Mackay characterized hyperbolic groups without $2$-torsion whose boundaries have conformal dimension equal to one~\cite{carrascomackay}. These groups admit a hierarchy over elementary edge groups that terminates in virtually cyclic or virtually Fuchsian vertex groups. This work uses combinatorial modulus techniques.

    For a detailed description on conformal dimension of boundaries of random groups we refer the reader to the survey of Mackay~\cite{mackay-survey};  we indicate some of the results below. First, recall the definition of the density model of a random group.

    \begin{defn}[Density model of random groups]
        Fix $m \geq 2$ and $S$ a generating set with $m$ elements. Fix a {\it density} $d \in (0,1)$. Consider all cyclically reduced words of length $\ell$ in $\la S \ra$. Consider all presentations which choose as relators $(2m-1)^{d\ell}$ of these words uniformly and independently at random. A property $\cP$ holds {\it generically} if the proportion of all such presentations at length $\ell$ which satisfy $\cP$ goes to $1$ as $\ell \rightarrow \infty$.
    \end{defn}

    \begin{thm} \cite[Theorem 1.3]{mackay16}
        There exists a constant $C>1$ so that for fixed $m \geq 2$, if $G$ is a random group at density $0<d<\frac{1}{8}$, then asymptotically almost surely
            \[ \frac{\log(2m-1) d\ell}{C|\log d|} \leq \Cdim(\p G). \]
        In particular, for each $d<\frac{1}{8}$, as $\ell \rightarrow \infty$, generic groups pass through infinitely many quasi-isometry classes.
    \end{thm}

    The bounds above extend those given by Mackay in \cite[Theorem 1.4]{mackay} for $0<d<\frac{1}{16}$. The theorem bounds has been further extended by Oppenheim~\cite{oppenheim} and Frost~\cite{frost} to $0<d<\frac{1}{2}$.

    Mackay uses round trees to produce these lower bounds on conformal dimension. The strategy to find round trees in the Cayley $2$-complex is to realize groups in this family as small cancellation groups that satisfy certain useful conditions on the relators. For example, in \cite{mackay} Mackay asks that there exists some $N\geq 12$ so that every reduced word in $\la S \ra$ of length $N$ appears at least once in a cyclic conjugate of a relator (or its inverse). The round tree is built recursively. At each step, the boundary arc of each previous strip is split into segments of length 6. At each endpoint the subcomplex is first extended by adding $b$ new outward edge paths of length 3. Because each word of length 12 appears in some relator, there exist polygons in the Cayley 2-complex to fill in these extensions to build $b$ new strips in the round tree. See \cite[Section 5]{mackay} and \cite[Figure 4]{mackay}.

    This approach was generalized in \cite{mackay16} and \cite{frost} by asking that all words of length $\frac{3d\ell}{4}$ appear in the presentation a.a.s, and subdividing the boundary arc into segments of length $\frac{d\ell}{4}$. See \cite[Figure 17]{mackay-survey}.

\subsection{Relatively hyperbolic Coxeter groups}

    In joint work with Elizabeth Field, Radhika Gupta, and Robbie Lyman, the author used the round tree technique to produce the first non-trivial conformal dimension bounds for the Bowditch boundary of non-hyperbolic relatively hyperbolic groups. As explained below, one application proved there are infinitely many quasi-isometry classes among hyperbolic groups with Pontryagin sphere boundary.

    We considered the following family of Coxeter groups.

    \begin{defn}[Coxeter group]
        Let $\Gamma$ be a finite, simplicial graph with vertex set $V\G = \{s_1, \ldots, s_m\}$ and edge set $E\G$. Assume that each unoriented edge $\{s_i,s_j\} \in E\Gamma$ is labeled by a positive integer $m_{ij} \geq 2$. The \emph{Coxeter group $W_\Gamma$ with defining graph $\Gamma$} is the group given by the presentation
    \[ W_\Gamma = \langle \, s_1, \ldots, s_m \mid s_i^2 = 1,\ (s_is_j)^{m_{ij}} = 1 \textrm{ for all } s_i \in V\G, \{s_i,s_j\} \in E\G \, \rangle. \]
    The generators $s_i$ are called \emph{standard} or \emph{Coxeter} generators of $W_\Gamma$.
    \end{defn}

    \begin{defn}[Family of groups $\cW$]
        Let $\cW$ be the family of Coxeter groups with defining graph a complete graph and so that each edge label is at least $3$.
    \end{defn}

    Groups in this family have been previously considered, and we note relevant properties.

    \begin{itemize}
        \item The Davis--Moussong complex for $W_\Gamma \in \cW$ is a $\CAT(0)$-polygonal complex where each 2-cell is a regular Euclidean polygon and each vertex link is a complete graph on $n$ vertices. The group acts geometrically on this complex; each generator acts by a reflection across a convex wall, which is isometric to a regular tree \cite{davis}.

        \item If $W_\Gamma \in \cW$ and there is no triangle in $\Gamma$ with all three edges labeled three, then $W_\Gamma$ is hyperbolic. Otherwise, $W_\Gamma$ is hyperbolic relative to virtually abelian groups of rank two. In this case, the group is $\CAT(0)$ with isolated flats \cite{wise96}, \cite{hruska04} \cite[Moussong's Theorem; Corollary 12.6.3]{davis}.

        Let $\cP$ be the collection of stabilizers of two-dimensional flats in the Davis complex (and let $\cP$ be empty if $W_\Gamma$ is hyperbolic). Then $(W_\Gamma, \cP)$ is a relatively hyperbolic group pair.

        \item If $|V\Gamma| = 4$, then $W_\Gamma$ is Kleinian, and the $\CAT(0)$ boundary of $W_\Gamma$ is homeomorphic to the Sierpinski carpet \cite{andreev, schroeder, ruane-sier, haulmarkhruskasathaye}.

        One can view groups in this family $\cW$ as being constructed from Kleinian groups. This point of view is detailed \cite{fieldguptalymanstark} and motivated the construction of a $\CAT(-1)$ space that the group pair $(W_\Gamma, \cP)$ acts on geometrically finitely.

        \item If $|V\Gamma|\geq 5$, then the $\CAT(0)$ boundary of $W_\Gamma$ is homeomorphic to the Menger curve as shown by Haulmark--Hruska--Sathaye~\cite{haulmarkhruskasathaye}. This family of groups gave the first examples of non-hyperbolic $\CAT(0)$ groups with Menger curve boundary.

        \item Bourdon--Kleiner used $\ell_p$-cohomology to produce upper bounds on the conformal dimension of hyperbolic groups in this family \cite{bourdonkleiner15}. In particular, they showed that if $W_\Gamma \in \cW$ is defined on a complete graph with $m$ vertices and every edge label is at least $M \geq 4$, then
        \[ \Cdim(\p W_\Gamma) \leq 1 + \frac{\log(m-1)}{\log(2M-5)}.\]
    \end{itemize}

    We produced upper and lower bounds on the conformal dimension of the Bowditch boundary of the relatively hyperbolic group pair $(W_\Gamma, \cP)$.

    \begin{remark}[Bowditch boundary]
        Bowditch~\cite{bowditch12} proved that a relatively hyperbolic group pair $(G, \cP)$ admits a well-defined boundary, called the {\it Bowditch boundary} and denoted $\p(G, \cP)$. This compact topological space is the visual boundary of any {\it cusped Cayley graph} for $G$, where a combinatorial horoball is attached to each coset of each peripheral subgroup~\cite{grovesmanning08}. The cusped Cayley graph is a proper geodesic $\delta$-hyperbolic metric space, and hence the boundary $\p(W_\Gamma, \cP)$ admits a family of visual metrics. Since each group in $\cP$ is not itself nontrivially relatively hyperbolic, the quasisymmetry type of $\p(W_\Gamma, \cP)$ is a quasi-isometry invariant of the group $W_\Gamma$ \cite{behrstockdrutumosher, groff13, healyhruska}. Thus, the conformal dimension of $\p(W_\Gamma, \cP)$ is a quasi-isometry invariant of the group $W_\Gamma \in \cW$.
    \end{remark}

    The lower bounds on the conformal dimension of $\p(W_\Gamma, \cP)$ were obtained by quasi-isometrically embedding round trees in the Davis--Moussong complex for these groups. As the groups are not hyperbolic, the complex was constructed to avoid large intersections with flats. The bounds are as follows.

\begin{thm} \cite[Theorem A]{fieldguptalymanstark} \label{thm:CoxBounds}
    Let $\Gamma$ be a complete graph with $m \geq  11$ vertices and edge labels $m_{ij} \geq 3$. Let $M = \max{m_{ij}}$. Then
        \[\Cdim\bigl(\p (W_\G, \cP)\bigr) \geq 1 + \frac{\log \bigl(\lfloor\frac{m-5}{3}\rfloor\bigr)}{\log (2M-1)}.\]
\end{thm}

    In particular, the theorem shows that if the maximum edge labels are fixed, then the conformal dimension goes to infinity as the number of vertices in the defining graph does. This implies there are infinitely many quasi-isometry classes among any such family of Coxeter groups. Combining this result with work of Bourdon--Kleiner~\cite{bourdonkleiner15} proved that the conformal dimension of the boundaries of hyperbolic groups in this family achieves a dense set in $(1,\infty)$.

    \begin{remark}[Round tree construction]
        Similar to the results of Bourdon and Mackay discussed above, the construction of the round tree subcomplexes was recursive and followed from the richness of each vertex link and its high symmetry. In particular, one can locally build strips by extending along an edge path. As the number of vertices in the defining graph increases, the number of strips that can be constructed increases. The subtlety in the construction lies in ensuring that the subcomplex is convex in the 1-skeleton. The main tool to study convexity in Coxeter group is the structure of walls in the Davis--Moussong complex, and this is utilized in the proof.
    \end{remark}

    Additional motivation for studying this family of groups $\cW$ was to provide lower bounds on the conformal dimension of the boundary of arbitrary hyperbolic Coxeter groups with edge labels at least three. Indeed, if $\Gamma$ is a simplicial graph and $\Lambda \subset \Gamma$ is an induced subgraph, then if $W_\Gamma$ is hyperbolic, $W_\Lambda \leq W_\Gamma$ is a quasi-convex subgroup. Thus, if $W_\Gamma$ is a hyperbolic Coxeter group whose nerve contains an induced complete graph $\Gamma'$, then the lower bound on the conformal dimension of $\p W_{\Gamma'}$ yields a lower bound on the conformal dimension of $\p W_{\Gamma}$.

    \begin{remark}[Pontryagin sphere boundaries] \label{rem:Pont}
        The Pontryagin sphere is a compact 2-dimensional fractal that arises as the boundary of $3$-dimensional hyperbolic pseudo-manifolds, where the link at each point in the space is either a sphere or a torus. This space can be viewed as an inverse limit of an inverse system of tori. Alternatively, the space obtained by gluing opposite sides of each boundary square in the standard square Sierpinski carpet is homeomorphic to the Pontryagin sphere. See \cite{jakobsche,swiatkowski-trees, zawislak, fischer, DoubaLeeMarquisRuffoni} for details of its construction.

        The Pontryagin sphere arises as the boundary of Coxeter groups, as shown by Fischer~\cite{fischer} in the right-angled case and by \'Swi\k{a}tkowski~\cite{swiatkowski-trees} more generally.  That is, the boundary of a hyperbolic Coxeter group whose nerve is a closed surface of genus at least one is homeomorphic to the Pontryagin sphere~\cite[Theorem 2]{swiatkowski-trees}. Every complete graph embeds as an induced subgraph in the 1-skeleton of a triangulation of a surface, and edge labels can be chosen so that the nerve of the corresponding Coxeter group is homeomorphic to the surface. Therefore, \Cref{thm:CoxBounds} implies there are infinitely many quasi-isometry classes among hyperbolic Coxeter groups with Pontryagin sphere boundary. Work with Cashen--Dani--Schreve~\cite{cashendanischrevestark} proves the analogous statement for right-angled Coxeter groups, by embedding round trees in the Davis complex within a particular infinite family of these groups.

        The quasi-isometry classification, abstract commensurability classification, and questions of rigidity for hyperbolic groups with Pontryagin sphere boundary are broadly open.
    \end{remark}

\bibliographystyle{alpha}
\bibliography{refs}

\newcommand{\etalchar}[1]{$^{#1}$}
\begin{thebibliography}{ECH{\etalchar{+}}92}

\bibitem[And70]{andreev}
E.~M. Andreev.
\newblock Convex polyhedra of finite volume in {L}oba\v cevski\u i\ space.
\newblock {\em Mat. Sb. (N.S.)}, 83(125):256--260, 1970.

\bibitem[Ass83]{assouad}
Patrice Assouad.
\newblock Plongements lipschitziens dans {${\bf R}\sp{n}$}.
\newblock {\em Bull. Soc. Math. France}, 111(4):429--448, 1983.

\bibitem[BA56]{beurlingahlfors}
A.~Beurling and L.~Ahlfors.
\newblock The boundary correspondence under quasiconformal mappings.
\newblock {\em Acta Math.}, 96:125--142, 1956.

\bibitem[BDM09]{behrstockdrutumosher}
Jason Behrstock, Cornelia Dru\c{t}u, and Lee Mosher.
\newblock Thick metric spaces, relative hyperbolicity, and quasi-isometric
  rigidity.
\newblock {\em Math. Ann.}, 344(3):543--595, 2009.

\bibitem[Ben93]{benakli}
Nadia Benakli.
\newblock Poly\`edres avec un nombre fini de types de sommets.
\newblock {\em C. R. Acad. Sci. Paris S\'er. I Math.}, 317(9):863--866, 1993.

\bibitem[BH99]{bridsonhaefliger}
Martin~R. Bridson and Andr{\'e} Haefliger.
\newblock {\em Metric spaces of non-positive curvature}, volume 319 of {\em
  Grundlehren der Mathematischen Wissenschaften [Fundamental Principles of
  Mathematical Sciences]}.
\newblock Springer-Verlag, Berlin, 1999.

\bibitem[BJ97]{bishopjones}
Christopher~J. Bishop and Peter~W. Jones.
\newblock Hausdorff dimension and {K}leinian groups.
\newblock {\em Acta Math.}, 179(1):1--39, 1997.

\bibitem[BK02]{bonkkleiner02}
Mario Bonk and Bruce Kleiner.
\newblock Quasisymmetric parametrizations of two-dimensional metric spheres.
\newblock {\em Invent. Math.}, 150(1):127--183, 2002.

\bibitem[BK05a]{bonkkleiner05S2}
Mario Bonk and Bruce Kleiner.
\newblock Conformal dimension and {G}romov hyperbolic groups with 2-sphere
  boundary.
\newblock {\em Geom. Topol.}, 9:219--246, 2005.

\bibitem[BK05b]{bonkkleiner-planes}
Mario Bonk and Bruce Kleiner.
\newblock Quasi-hyperbolic planes in hyperbolic groups.
\newblock {\em Proc. Amer. Math. Soc.}, 133(9):2491--2494, 2005.

\bibitem[BK13]{bourdonkleiner13}
Marc Bourdon and Bruce Kleiner.
\newblock Combinatorial modulus, the combinatorial {L}oewner property, and
  {C}oxeter groups.
\newblock {\em Groups Geom. Dyn.}, 7(1):39--107, 2013.

\bibitem[BK15]{bourdonkleiner15}
Marc Bourdon and Bruce Kleiner.
\newblock Some applications of {$\ell_p$}-cohomology to boundaries of {G}romov
  hyperbolic spaces.
\newblock {\em Groups Geom. Dyn.}, 9(2):435--478, 2015.

\bibitem[BM91]{bestvinamess}
Mladen Bestvina and Geoffrey Mess.
\newblock The boundary of negatively curved groups.
\newblock {\em J. Amer. Math. Soc.}, 4(3):469--481, 1991.

\bibitem[Bon06]{bonk-ICM}
Mario Bonk.
\newblock Quasiconformal geometry of fractals.
\newblock In {\em International {C}ongress of {M}athematicians. {V}ol. {II}},
  pages 1349--1373. Eur. Math. Soc., Z\"urich, 2006.

\bibitem[Bou95a]{bourdon95}
Marc Bourdon.
\newblock Au bord de certains poly\`edres hyperboliques.
\newblock {\em Ann. Inst. Fourier (Grenoble)}, 45(1):119--141, 1995.

\bibitem[Bou95b]{bourdon-flot}
Marc Bourdon.
\newblock Structure conforme au bord et flot g\'eod\'esique d'un {${\rm
  CAT}(-1)$}-espace.
\newblock {\em Enseign. Math. (2)}, 41(1-2):63--102, 1995.

\bibitem[Bou97]{bourdon}
M.~Bourdon.
\newblock Immeubles hyperboliques, dimension conforme et rigidit\'e de
  {M}ostow.
\newblock {\em Geom. Funct. Anal.}, 7(2):245--268, 1997.

\bibitem[Bou98]{bourbaki}
Nicolas Bourbaki.
\newblock {\em General topology. {C}hapters 1--4}.
\newblock Elements of Mathematics (Berlin). Springer-Verlag, Berlin, 1998.
\newblock Translated from the French, Reprint of the 1989 English translation.

\bibitem[Bou18]{bourdon-mostowtype}
Marc Bourdon.
\newblock Mostow type rigidity theorems.
\newblock In {\em Handbook of group actions. {V}ol. {IV}}, volume~41 of {\em
  Adv. Lect. Math. (ALM)}, pages 139--188. Int. Press, Somerville, MA, 2018.

\bibitem[Bow98]{bowditch}
Brian~H. Bowditch.
\newblock Cut points and canonical splittings of hyperbolic groups.
\newblock {\em Acta Math.}, 180(2):145--186, 1998.

\bibitem[Bow99]{bowditch-treelike}
B.~H. Bowditch.
\newblock Treelike structures arising from continua and convergence groups.
\newblock {\em Mem. Amer. Math. Soc.}, 139(662):viii+86, 1999.

\bibitem[Bow12]{bowditch12}
B.~H. Bowditch.
\newblock Relatively hyperbolic groups.
\newblock {\em Internat. J. Algebra Comput.}, 22(3):1250016, 66, 2012.

\bibitem[BP00]{bourdonpajot}
Marc Bourdon and Herv\'{e} Pajot.
\newblock Rigidity of quasi-isometries for some hyperbolic buildings.
\newblock {\em Comment. Math. Helv.}, 75(4):701--736, 2000.

\bibitem[BP02]{bourdonpajot-survey}
Marc Bourdon and Herv\'e Pajot.
\newblock Quasi-conformal geometry and hyperbolic geometry.
\newblock In {\em Rigidity in dynamics and geometry ({C}ambridge, 2000)}, pages
  1--17. Springer, Berlin, 2002.

\bibitem[Bri00]{brinkmann00}
P.~Brinkmann.
\newblock Hyperbolic automorphisms of free groups.
\newblock {\em Geom. Funct. Anal.}, 10(5):1071--1089, 2000.

\bibitem[BS00]{bonkschramm}
M.~Bonk and O.~Schramm.
\newblock Embeddings of {G}romov hyperbolic spaces.
\newblock {\em Geom. Funct. Anal.}, 10(2):266--306, 2000.

\bibitem[BS07]{buyaloschroeder}
Sergei Buyalo and Viktor Schroeder.
\newblock {\em Elements of asymptotic geometry}.
\newblock EMS Monographs in Mathematics. European Mathematical Society (EMS),
  Z\"urich, 2007.

\bibitem[Cal13]{calegari13}
Danny Calegari.
\newblock The ergodic theory of hyperbolic groups.
\newblock In {\em Geometry and topology down under}, volume 597 of {\em
  Contemp. Math.}, pages 15--52. Amer. Math. Soc., Providence, RI, 2013.

\bibitem[Can84]{cannon84}
James~W. Cannon.
\newblock The combinatorial structure of cocompact discrete hyperbolic groups.
\newblock {\em Geom. Dedicata}, 16(2):123--148, 1984.

\bibitem[Can91]{cannon-conj}
James~W. Cannon.
\newblock The theory of negatively curved spaces and groups.
\newblock In {\em Ergodic theory, symbolic dynamics, and hyperbolic spaces
  ({T}rieste, 1989)}, Oxford Sci. Publ., pages 315--369. Oxford Univ. Press,
  New York, 1991.

\bibitem[CDSS]{cashendanischrevestark}
Christopher Cashen, Pallavi Dani, Kevin Schreve, and Emily Stark.
\newblock Work in progress.

\bibitem[CJ94]{cassonjungreis}
Andrew Casson and Douglas Jungreis.
\newblock Convergence groups and {S}eifert fibered {$3$}-manifolds.
\newblock {\em Invent. Math.}, 118(3):441--456, 1994.

\bibitem[CM22]{carrascomackay}
Matias Carrasco and John~M. Mackay.
\newblock Conformal dimension of hyperbolic groups that split over elementary
  subgroups.
\newblock {\em Invent. Math.}, 227(2):795--854, 2022.

\bibitem[Coo93]{coornaert93}
Michel Coornaert.
\newblock Mesures de {P}atterson-{S}ullivan sur le bord d'un espace
  hyperbolique au sens de {G}romov.
\newblock {\em Pacific J. Math.}, 159(2):241--270, 1993.

\bibitem[CP13]{carrasco}
Matias Carrasco~Piaggio.
\newblock On the conformal gauge of a compact metric space.
\newblock {\em Ann. Sci. \'{E}c. Norm. Sup\'{e}r. (4)}, 46(3):495--548 (2013),
  2013.

\bibitem[Dav08]{davis}
Michael~W. Davis.
\newblock {\em The geometry and topology of {C}oxeter groups}, volume~32 of
  {\em London Mathematical Society Monographs Series}.
\newblock Princeton University Press, Princeton, NJ, 2008.

\bibitem[Deh12]{dehn12}
M.~Dehn.
\newblock Transformation der {K}urven auf zweiseitigen {F}l\"achen.
\newblock {\em Math. Ann.}, 72(3):413--421, 1912.

\bibitem[DG11]{dahmaniguirardel-iso}
Fran\c~cois Dahmani and Vincent Guirardel.
\newblock The isomorphism problem for all hyperbolic groups.
\newblock {\em Geom. Funct. Anal.}, 21(2):223--300, 2011.

\bibitem[DK18]{drutukapovich}
Cornelia Dru\c{t}u and Michael Kapovich.
\newblock {\em Geometric group theory}, volume~63 of {\em American Mathematical
  Society Colloquium Publications}.
\newblock American Mathematical Society, Providence, RI, 2018.
\newblock With an appendix by Bogdan Nica.

\bibitem[DLMR]{DoubaLeeMarquisRuffoni}
Sami Douba, Gye-Seon Lee, Ludovic Marquis, and Lorenzo Ruffoni.
\newblock Convex cocompact groups in real hyperbolic spaces with limit set a
  pontryagin sphere.
\newblock arXiv:2502.09470.

\bibitem[DO07]{dymaraosajda}
Jan Dymara and Damian Osajda.
\newblock Boundaries of right-angled hyperbolic buildings.
\newblock {\em Fund. Math.}, 197:123--165, 2007.

\bibitem[DS97]{davidsemmes}
Guy David and Stephen Semmes.
\newblock {\em Fractured fractals and broken dreams}, volume~7 of {\em Oxford
  Lecture Series in Mathematics and its Applications}.
\newblock The Clarendon Press, Oxford University Press, New York, 1997.
\newblock Self-similar geometry through metric and measure.

\bibitem[Dun85]{dunwoody}
M.~J. Dunwoody.
\newblock The accessibility of finitely presented groups.
\newblock {\em Invent. Math.}, 81(3):449--457, 1985.

\bibitem[DY05]{dahmaniyaman}
Fran\c{c}ois Dahmani and Asl\i Yaman.
\newblock Bounded geometry in relatively hyperbolic groups.
\newblock {\em New York J. Math.}, 11:89--95, 2005.

\bibitem[ECH{\etalchar{+}}92]{wordprocessing}
David B.~A. Epstein, James~W. Cannon, Derek~F. Holt, Silvio V.~F. Levy,
  Michael~S. Paterson, and William~P. Thurston.
\newblock {\em Word processing in groups}.
\newblock Jones and Bartlett Publishers, Boston, MA, 1992.

\bibitem[Fal90]{falconer}
Kenneth Falconer.
\newblock {\em Fractal geometry}.
\newblock John Wiley \& Sons, Ltd., Chichester, 1990.
\newblock Mathematical foundations and applications.

\bibitem[FGLS]{fieldguptalymanstark}
Elizabeth Field, Radhika Gupta, Robert~Alonzo Lyman, and Emily Stark.
\newblock Conformal dimension bounds for certain {C}oxeter group {B}owditch
  boundaries.
\newblock arXiv:2504.12404.

\bibitem[Fis03]{fischer}
Hanspeter Fischer.
\newblock Boundaries of right-angled {C}oxeter groups with manifold nerves.
\newblock {\em Topology}, 42(2):423--446, 2003.

\bibitem[Fro]{frost}
Jordan Frost.
\newblock Round trees and conformal dimension in random groups: low density to
  high density.
\newblock arXiv:2204.05165.

\bibitem[Gab92]{gabai}
David Gabai.
\newblock Convergence groups are {F}uchsian groups.
\newblock {\em Ann. of Math. (2)}, 136(3):447--510, 1992.

\bibitem[GdlH90]{ghysdelaharpe}
\'E. Ghys and P.~de~la Harpe, editors.
\newblock {\em Sur les groupes hyperboliques d'apr\`es {M}ikhael {G}romov},
  volume~83 of {\em Progress in Mathematics}.
\newblock Birkh\"auser Boston, Inc., Boston, MA, 1990.
\newblock Papers from the Swiss Seminar on Hyperbolic Groups held in Bern,
  1988.

\bibitem[GM08]{grovesmanning08}
Daniel Groves and Jason~Fox Manning.
\newblock Dehn filling in relatively hyperbolic groups.
\newblock {\em Israel J. Math.}, 168:317--429, 2008.

\bibitem[Gro87]{gromov}
M.~Gromov.
\newblock Hyperbolic groups.
\newblock In {\em Essays in group theory}, volume~8 of {\em Math. Sci. Res.
  Inst. Publ.}, pages 75--263. Springer, New York, 1987.

\bibitem[Gro93]{gromov93}
M.~Gromov.
\newblock Asymptotic invariants of infinite groups.
\newblock In {\em Geometric group theory, {V}ol. 2 ({S}ussex, 1991)}, volume
  182 of {\em London Math. Soc. Lecture Note Ser.}, pages 1--295. Cambridge
  Univ. Press, Cambridge, 1993.

\bibitem[Gro13]{groff13}
Bradley~W. Groff.
\newblock Quasi-isometries, boundaries and {JSJ}-decompositions of relatively
  hyperbolic groups.
\newblock {\em J. Topol. Anal.}, 5(4):451--475, 2013.

\bibitem[Ha{\"i}09]{haissinsky-bourbaki}
Peter Ha{\"i}ssinsky.
\newblock G\'eom\'etrie quasiconforme, analyse au bord des espaces m\'etriques
  hyperboliques et rigidit\'es [d'apr\`es {M}ostow, {P}ansu, {B}ourdon,
  {P}ajot, {B}onk, {K}leiner{$\ldots$}].
\newblock Number 326, pages Exp. No. 993, ix, 321--362. 2009.
\newblock S\'eminaire Bourbaki. Vol. 2007/2008.

\bibitem[Ha{\"i}18]{haissinsky-QM}
Peter Ha{\"i}ssinsky.
\newblock Actions of quasi-{M}\"obius groups.
\newblock In {\em Handbook of group actions. {V}ol. {IV}}, volume~41 of {\em
  Adv. Lect. Math. (ALM)}, pages 23--94. Int. Press, Somerville, MA, 2018.

\bibitem[Hei01]{heinonen-Lectures}
Juha Heinonen.
\newblock {\em Lectures on analysis on metric spaces}.
\newblock Universitext. New York, NY: Springer, 2001.

\bibitem[HH]{healyhruska}
Burns Healy and G.~Christopher Hruska.
\newblock Cusped spaces and quasi-isometries of relatively hyperbolic groups.
\newblock arXiv:2010.09876.

\bibitem[HHS20]{haulmarkhruskasathaye}
Matthew Haulmark, G.~Christopher Hruska, and Bakul Sathaye.
\newblock Nonhyperbolic {C}oxeter groups with {M}enger boundary.
\newblock {\em Enseign. Math.}, 65(1-2):207--220, 2020.

\bibitem[HK98]{heinonenkoskela}
Juha Heinonen and Pekka Koskela.
\newblock Quasiconformal maps in metric spaces with controlled geometry.
\newblock {\em Acta Math.}, 181(1):1--61, 1998.

\bibitem[Hru04]{hruska04}
G.~Christopher Hruska.
\newblock Nonpositively curved 2-complexes with isolated flats.
\newblock {\em Geom. Topol.}, 8:205--275, 2004.

\bibitem[HS]{hodaswiatkowski}
Nima Hoda and Jacek \'Swi{\k{a}}tkowski.
\newblock Trees of graphs as boundaries of hyperbolic groups.
\newblock arXiv:2312.15827.

\bibitem[Jak91]{jakobsche}
W.~Jakobsche.
\newblock Homogeneous cohomology manifolds which are inverse limits.
\newblock {\em Fund. Math.}, 137(2):81--95, 1991.

\bibitem[Kat92]{katok}
Svetlana Katok.
\newblock {\em Fuchsian groups}.
\newblock Chicago Lectures in Mathematics. University of Chicago Press,
  Chicago, IL, 1992.

\bibitem[KB02]{kapovichbenakli}
Ilya Kapovich and Nadia Benakli.
\newblock Boundaries of hyperbolic groups.
\newblock In {\em Combinatorial and geometric group theory ({N}ew {Y}ork,
  2000/{H}oboken, {NJ}, 2001)}, volume 296 of {\em Contemp. Math.}, pages
  39--93. Amer. Math. Soc., Providence, RI, 2002.

\bibitem[KK00]{kapovichkleiner}
Michael Kapovich and Bruce Kleiner.
\newblock Hyperbolic groups with low-dimensional boundary.
\newblock {\em Ann. Sci. \'Ecole Norm. Sup. (4)}, 33(5):647--669, 2000.

\bibitem[Kle06]{kleiner-ICM}
Bruce Kleiner.
\newblock The asymptotic geometry of negatively curved spaces: uniformization,
  geometrization and rigidity.
\newblock In {\em International {C}ongress of {M}athematicians. {V}ol. {II}},
  pages 743--768. Eur. Math. Soc., Z\"urich, 2006.

\bibitem[Lev98]{levitt-nonnesting}
Gilbert Levitt.
\newblock Non-nesting actions on real trees.
\newblock {\em Bull. London Math. Soc.}, 30(1):46--54, 1998.

\bibitem[Mac08]{mackay-arcs}
John~M. Mackay.
\newblock Existence of quasi-arcs.
\newblock {\em Proc. Amer. Math. Soc.}, 136(11):3975--3981, 2008.

\bibitem[Mac10]{mackay10}
John~M. Mackay.
\newblock Spaces and groups with conformal dimension greater than one.
\newblock {\em Duke Math. J.}, 153(2):211--227, 2010.

\bibitem[Mac12]{mackay}
John~M. Mackay.
\newblock Conformal dimension and random groups.
\newblock {\em Geom. Funct. Anal.}, 22(1):213--239, 2012.

\bibitem[Mac16]{mackay16}
John~M. Mackay.
\newblock Conformal dimension via subcomplexes for small cancellation and
  random groups.
\newblock {\em Math. Ann.}, 364(3-4):937--982, 2016.

\bibitem[Mac25]{mackay-survey}
John~M. Mackay.
\newblock Conformal dimension and hyperbolic groups.
\newblock volume~63 of {\em Panor. Synth\`eses}, pages 145--170. Soc. Math.
  France, Paris, [2025] \copyright 2025.

\bibitem[Mos73]{mostow}
G.~D. Mostow.
\newblock {\em Strong rigidity of locally symmetric spaces}.
\newblock Annals of Mathematics Studies, No. 78. Princeton University Press,
  Princeton, N.J.; University of Tokyo Press, Tokyo, 1973.

\bibitem[Mou88]{moussong-thesis}
Gabor Moussong.
\newblock {\em Hyperbolic {C}oxeter groups}.
\newblock ProQuest LLC, Ann Arbor, MI, 1988.
\newblock Thesis (Ph.D.)--The Ohio State University.

\bibitem[MS15]{martinswiatkowski}
Alexandre Martin and Jacek \'Swi{\k{a}}tkowski.
\newblock Infinitely-ended hyperbolic groups with homeomorphic {G}romov
  boundaries.
\newblock {\em J. Group Theory}, 18(2):273--289, 2015.

\bibitem[MS20]{mackaysisto-planes}
John~M. Mackay and Alessandro Sisto.
\newblock Quasi-hyperbolic planes in relatively hyperbolic groups.
\newblock {\em Ann. Acad. Sci. Fenn. Math.}, 45(1):139--174, 2020.

\bibitem[MS24]{mackaysisto-maps_relhyp}
John~M. Mackay and Alessandro Sisto.
\newblock Maps between relatively hyperbolic spaces and between their
  boundaries.
\newblock {\em Trans. Amer. Math. Soc.}, 377(2):1409--1454, 2024.

\bibitem[MT10]{mackaytyson}
John~M. Mackay and Jeremy~T. Tyson.
\newblock {\em Conformal dimension}, volume~54 of {\em University Lecture
  Series}.
\newblock American Mathematical Society, Providence, RI, 2010.
\newblock Theory and application.

\bibitem[Opp]{oppenheim}
Izhar Oppenheim.
\newblock Banach zuk's criterion for partite complexes with application to
  random groups.
\newblock arXiv:2112.02929.

\bibitem[Pan89a]{pansu}
Pierre Pansu.
\newblock Dimension conforme et sph\`ere \`a l'infini des vari\'{e}t\'{e}s \`a
  courbure n\'{e}gative.
\newblock {\em Ann. Acad. Sci. Fenn. Ser. A I Math.}, 14(2):177--212, 1989.

\bibitem[Pan89b]{pansu89b}
Pierre Pansu.
\newblock M\'{e}triques de {C}arnot-{C}arath\'{e}odory et quasiisom\'{e}tries
  des espaces sym\'{e}triques de rang un.
\newblock {\em Ann. of Math. (2)}, 129(1):1--60, 1989.

\bibitem[Pat76]{patterson76}
S.~J. Patterson.
\newblock The limit set of a {F}uchsian group.
\newblock {\em Acta Math.}, 136(3-4):241--273, 1976.

\bibitem[Pau96]{paulin}
Fr\'{e}d\'{e}ric Paulin.
\newblock Un groupe hyperbolique est d\'{e}termin\'{e} par son bord.
\newblock {\em J. London Math. Soc. (2)}, 54(1):50--74, 1996.

\bibitem[Pau97]{paulin97}
Fr\'ed\'eric Paulin.
\newblock On the critical exponent of a discrete group of hyperbolic
  isometries.
\newblock {\em Differential Geom. Appl.}, 7(3):231--236, 1997.

\bibitem[PY19]{potyagailoyang}
Leonid Potyagailo and Wen-yuan Yang.
\newblock Hausdorff dimension of boundaries of relatively hyperbolic groups.
\newblock {\em Geom. Topol.}, 23(4):1779--1840, 2019.

\bibitem[Rua05]{ruane-sier}
Kim Ruane.
\newblock C{AT}(0) boundaries of truncated hyperbolic space.
\newblock volume~29, pages 317--331. 2005.
\newblock Spring Topology and Dynamical Systems Conference.

\bibitem[\'S20a]{swiatkowski-treesOfManifolds}
Jacek \'Swi{\k{a}}tkowski.
\newblock Trees of manifolds as boundaries of spaces and groups.
\newblock {\em Geom. Topol.}, 24(2):593--622, 2020.

\bibitem[\'S20b]{swiatkowski-trees}
Jacek \'Swi\k{a}tkowski.
\newblock Trees of manifolds as boundaries of spaces and groups.
\newblock {\em Geom. Topol.}, 24(2):593--622, 2020.

\bibitem[Sch06]{schroeder06}
Viktor Schroeder.
\newblock Quasi-metric and metric spaces.
\newblock {\em Conform. Geom. Dyn.}, 10:355--360, 2006.

\bibitem[Sch09]{schroeder}
Timothy~A. Schroeder.
\newblock Geometrization of 3-dimensional {C}oxeter orbifolds and {S}inger's
  conjecture.
\newblock {\em Geom. Dedicata}, 140:163--174, 2009.

\bibitem[Sel95]{sela95}
Z.~Sela.
\newblock The isomorphism problem for hyperbolic groups. {I}.
\newblock {\em Ann. of Math. (2)}, 141(2):217--283, 1995.

\bibitem[Sel97]{sela-JSJ}
Z.~Sela.
\newblock Structure and rigidity in ({G}romov) hyperbolic groups and discrete
  groups in rank {$1$} {L}ie groups. {II}.
\newblock {\em Geom. Funct. Anal.}, 7(3):561--593, 1997.

\bibitem[Sta68]{stallings}
John~R. Stallings.
\newblock On torsion-free groups with infinitely many ends.
\newblock {\em Ann. of Math. (2)}, 88:312--334, 1968.

\bibitem[Sul79]{sullivan79}
Dennis Sullivan.
\newblock The density at infinity of a discrete group of hyperbolic motions.
\newblock {\em Inst. Hautes \'Etudes Sci. Publ. Math.}, (50):171--202, 1979.

\bibitem[Sul84]{sullivan84}
Dennis Sullivan.
\newblock Entropy, {H}ausdorff measures old and new, and limit sets of
  geometrically finite {K}leinian groups.
\newblock {\em Acta Math.}, 153(3-4):259--277, 1984.

\bibitem[Swa96]{swarup-cutpoint}
G.~A. Swarup.
\newblock On the cut point conjecture.
\newblock {\em Electron. Res. Announc. Amer. Math. Soc.}, 2(2):98--100, 1996.

\bibitem[Tuk86]{tukia86}
Pekka Tukia.
\newblock On quasiconformal groups.
\newblock {\em J. Analyse Math.}, 46:318--346, 1986.

\bibitem[Tuk88]{tukia}
Pekka Tukia.
\newblock Homeomorphic conjugates of {F}uchsian groups.
\newblock {\em J. Reine Angew. Math.}, 391:1--54, 1988.

\bibitem[Tuk96]{tukia-turning}
Pekka Tukia.
\newblock Spaces and arcs of bounded turning.
\newblock {\em Michigan Math. J.}, 43(3):559--584, 1996.

\bibitem[TV80]{tukiavaisala}
P.~Tukia and J.~V\"ais\"al\"a.
\newblock Quasisymmetric embeddings of metric spaces.
\newblock {\em Ann. Acad. Sci. Fenn. Ser. A I Math.}, 5(1):97--114, 1980.

\bibitem[V\"85]{vaisala-QM}
Jussi V\"ais\"al\"a.
\newblock Quasi-{M}\"obius maps.
\newblock {\em J. Analyse Math.}, 44:218--234, 1984/85.

\bibitem[Wis96]{wise96}
Daniel~T. Wise.
\newblock {\em Non-positively curved squared complexes: {A}periodic tilings and
  non-residually finite groups}.
\newblock ProQuest LLC, Ann Arbor, MI, 1996.
\newblock Thesis (Ph.D.)--Princeton University.

\bibitem[Xie06]{xie06}
Xiangdong Xie.
\newblock Quasi-isometric rigidity of {F}uchsian buildings.
\newblock {\em Topology}, 45(1):101--169, 2006.

\bibitem[Zaw10]{zawislak}
Pawe\l{} Zawi\'slak.
\newblock Trees of manifolds and boundaries of systolic groups.
\newblock {\em Fund. Math.}, 207(1):71--99, 2010.

\end{thebibliography}

 \end{document}